\documentclass[reqno,english]{amsart}
\usepackage{amsfonts,amsmath, latexsym,verbatim,amscd,mathrsfs,xcolor,array, mathabx, graphicx}

\usepackage{hyperref}

\begin{document}

\title{Adiabatic Approximation of Abelian Higgs models}

\author{Amirmasoud Geevechi}
\address{Beijing Institute of Mathematical Sciences and Applications, Beijing, China}
\email{amirmasoud@bimsa.cn}

\author{Robert L. Jerrard}
\address{Department of Mathematics, University of Toronto, 
			Toronto, Ontario M5S 2E4, Canada.}
\email{rjerrard@math.toronto.edu}

\date{\today}

\newcommand{\tU}{\tilde U}
\newcommand{\tPhi}{\tilde\Phi}
\newcommand{\tA}{\tilde A}
\newcommand{\tF}{\tilde F}
\newcommand{\tu}{\tilde u}

\newcommand{\tV}{\tilde V}
\newcommand{\tW}{\tilde W}

\newcommand{\modu}{\mathfrak u}
\newcommand{\mphi}{\varphi}
\newcommand{\mA}{\mathcal A}
\newcommand{\mF}{\mathcal F}

\newcommand{\aU}{U_\ep^{ap}}
\newcommand{\aPhi}{\Phi_\ep^{ap}}
\newcommand{\aA}{{A_\ep^{ap}}}
\newcommand{\aAa}{{A_{\ep a}^{ap}}}
\newcommand{\aAzero}{{A_{\ep0}^{ap}}}
\newcommand{\aAalpha}{{A_{\ep \alpha}^{ap}}}
\newcommand{\aAbeta}{{A_{\ep \beta}^{ap}}}
\newcommand{\aAbara}{{A_{\ep\bar a}^{ap}}}
\newcommand{\au}{u_\ep^{ap}}
\newcommand{\aF}{{F^{ap}_\ep}}

\newcommand{\teta}{\tilde \eta}
\newcommand{\txi}{\tilde \xi}
\newcommand{\tq}{\tilde q}
\newcommand{\tta}{\tilde{\tilde  a}}

\newcommand{\norm}[1]{\left\Vert#1\right\Vert}
\newcommand{\ip}[1]{({#1}\rangle}
\newcommand{\p}[1]{\partial_{#1}}
\newcommand{\abs}[1]{\left\vert#1\right\vert}

\newcommand{\kz}{\kappa_0}
\newcommand{\ko}{\kappa_1}
\newcommand{\kt}{\iota}

\newcommand{\barint}{\overline{\hspace{.65em}}\!\!\!\!\!\!\int}
\newcommand{\e}{\epsilon}
\newcommand{\ve}{v_\epsilon}
\newcommand{\jep}{j_\epsilon}
\newcommand{\loc}{ {\mbox{\scriptsize{loc}}} }
\newcommand{\R}{{\mathbb R}}
\newcommand{\C}{{\mathbb C}}
\newcommand{\J}{{\mathbb J}}
\newcommand{\T}{{\mathbb T}}
\newcommand{\Z}{{\mathbb Z}}
\newcommand{\N}{{\mathbb N}}
\newcommand{\Q}{{\mathbb Q}}
\newcommand{\Hdf}{{\mathcal{H}}}
\newcommand{\calD}{{\mathcal{D}}}
\newcommand{\calL}{{\mathcal{L}}}
\newcommand{\calS}{{\mathcal{S}}}
\newcommand{\calB}{{\mathcal{B}}}
\newcommand{\calE}{{\mathcal{E}}}
\newcommand{\calF}{{\mathcal{F}}}
\newcommand{\calQ}{{\mathfrak{Q}}}
\newcommand{\calG}{{\mathcal{G}}}
\newcommand{\calH}{{\mathcal{H}}}
\newcommand{\calJ}{{\mathcal{J}}}
\newcommand{\calN}{{\mathcal{N}}}
\newcommand{\calT}{{\mathcal{T}}}
\newcommand{\calM}{{\mathcal{M}}}
\newcommand{\calZ}{{\mathcal{Z}}}
\newcommand{\De}{D^\e}
\newcommand{\dist}{\operatorname{dist}}
\newcommand{\spt}{\operatorname{spt}}
\def\rest{\hskip 1pt{\hbox to 10.8pt{\hfill
\vrule height 7pt width 0.4pt depth 0pt\hbox{\vrule height 0.4pt
width 7.6pt depth 0pt}\hfill}}}
\newcommand{\bd}{\partial}
\newcommand{\PD}{P\hspace{-.2em}D}

\long\def\green#1{{\color{green!45!black}#1}}

\long\def\red#1{{\color{red}#1}}
\long\def\violet#1{{\color{violet}#1}}
\long\def\blue#1{{\color{blue}#1}}
\long\def\yellow#1{{\color{black}#1}}
\long\def\comment#1{\marginpar{\raggedright\small$\bullet$\ #1}}

\newcommand{\vol}{\,\mbox{vol}}

\newcommand{\bn}[1]{ ({\texttt{#1}})}

\newcommand{\beq}{\begin{equation}}
\newcommand{\eeq}{\end{equation}}
\newcommand{\pp}{\partial}
\newcommand{\ep}{\varepsilon}
\newcommand{\vp}{\varphi}

\def\logeps{{|\!\log\epsilon|}}

\theoremstyle{plain}
\newtheorem{theorem}{Theorem}
\newtheorem{proposition}{Proposition}
\newtheorem{lemma}{Lemma}
\newtheorem{corollary}{Corollary}
\newtheorem{example}{Example}
\newtheorem*{thm1full}{Theorem 1, full version}

\theoremstyle{definition}
\newtheorem{definition}{Definition}
\newtheorem{notation}{Notation}

\theoremstyle{remark}
\newtheorem{remark}{Remark}
\newtheorem{warning}{Warning}

\numberwithin{equation}{section}
\setcounter{tocdepth}{3}

 
\newcommand{\wpp}{\textsf{\reflectbox{6}}}
\newcommand{\wD}{{\rm D}}
\newcommand{\wS}{{\rm S}}
\newcommand{\wZ}{{\rm Z}}
\newcommand{\wSe}{{\rm S}^\ep}
\newcommand{\wU}{{\rm U}}
\newcommand{\wV}{{\rm V}}
\newcommand{\wu}{{\rm u}}
\newcommand{\wtu}{\tilde {\rm u}}
\newcommand{\wtU}{\tilde{\rm U}}
\newcommand{\wtZ}{\tilde{\rm Z}}
\newcommand{\wF}{{\rm F}}
\newcommand{\wtF}{\tilde {\rm F}}
\newcommand{\wA}{{\rm A}}
\newcommand{\wtA}{\tilde{\rm A}}

\begin{abstract}
    We construct novel solutions in $d\ge 3$ space dimensions of a family of nonlinear evolutions equations that includes the critical hyperbolic Abelian Higgs model (AHM). For the AHM, these solutions exhibit an ensemble of  $N\ge 1$ slowly-moving, nearly parallel vortex filaments, whose leading-order dynamics are described by a wave map from $\R^{d-2}$ into the Abelian Higgs moduli space, a manifold carrying a natural Riemannian structure that parametrizes stationary 2d solutions of the AHM. We also prove extremely similar results that relate the critical Abelian Higgs heat flow, modeling certain superconductors, to the harmonic map heat flow into the Moduli space, as well as some parallel results for near-critical equations.  When $d=3$, these results allow for the study of the poorly-understood phenomenon of vortex reconnection in this setting.
\end{abstract}

\maketitle

\tableofcontents

\section{Introduction}

In this paper we construct novel solutions of a family of evolution equations arising in gauge theory in $\R^{d+1}$ for $d\ge 3$. In all the equations we consider, the unknowns are a complex-valued function $\Phi$ and a (real valued) $1$-form $A=A_1 dx_1+\ldots + A_d dx_d + A_0 dt$ on $\R^{d+1}$ which is used to form a covariant derivative
\[
D_\alpha = \pp_\alpha - i A_\alpha, \qquad\mbox{ where $\pp_\alpha$ denotes } \begin{cases}
\pp_t &\mbox{ if }\alpha = 0\\
\pp_{x_\alpha}&\mbox{if }\alpha\in \{1,\ldots, d\}.
\end{cases}
\] 
We will often tacitly identify $A$ with a map taking values in $\R^{d+1}$.

We are chiefly interested in two related equations. 
The first of these is the Abelian Higgs model:
\begin{equation}\label{AHM}
\left. \begin{aligned}
D_0 D_0 \Phi -  D_kD_k \Phi
+\frac \lambda 2(|\Phi|^2-1)\Phi  &=0\\
\pp_0 F_{0\alpha} - \pp_k F_{k\alpha} - (i\Phi, D_\alpha\Phi)  &=0, \qquad \alpha=0,\ldots, d. 
\end{aligned}\ \ \right\}
\end{equation}
where $\lambda>0$ is a real parameter, the repeated index $k$ is summed from $1$ to $d$, and we use the notation
\[
F_{\alpha\beta} = \pp_\alpha A_\beta - \pp_\beta A_\alpha,
\]
and 
\[
\mbox{ for $v,w\in \C$,} \qquad (v,w) := \frac 12 (v\bar w + w\bar v).
\]
This is a nonlinear hyperbolic equation arising in high energy physics, as perhaps the simplest instance of a Yang-Mills-Higgs theory. 
We will also study its parabolic counterpart, the Abelian Higgs heat flow:
\begin{equation}\label{PAH}
\left. \begin{aligned}
D_0 \Phi - D_kD_k \Phi +\frac \lambda 2(|\Phi|^2-1)\Phi  &=0\\
F_{0j} - \pp_kF_{kj} - (i\Phi, D_j\Phi)  &=0, \qquad j=1,\ldots, d. 
\end{aligned}\ \ \right\}
\end{equation}
This arises as a description of the evolution of superconductors in Ginzburg-Landau theory.

Both  \eqref{AHM} and \eqref{PAH} are invariant with respect to gauge transformations
\begin{equation}\label{gauge}
\Phi\mapsto e^{i\chi}\Phi,\qquad A \mapsto A + d\chi
\end{equation}
where $\chi$ is a sufficiently smooth function with the same domain as $(\Phi, A)$, in the sense that $(\Phi, A)$ is a solution of \eqref{dAHM} if and only if $(e^{i\chi}\Phi, A+d\chi)$ is a solution. As with all gauge theories, this has important consequence for the analysis of the equations. For example, it is not possible to prove uniqueness or to estimate any quantities that are not gauge-invariant without fixing the degree of freedom represented by \eqref{gauge}.

For both \eqref{AHM} and \eqref{PAH} it is known that for suitable initial data, solutions possess vortex submanifolds --- codimension $2$ concentration sets of the vorticity and energy -- whose dynamics is governed in a suitable limit by a geometric evolution equation: the timelike Minkowskian minimal surface equation\footnote{this is only proved when $d=3$ but is expected to be valid in all higher dimensions as well.} for \eqref{AHM}, see \cite{CzubakJerrard}, and the mean curvature flow for \eqref{PAH}, see \cite{ParisePigatiStern}. In this paper we study
the dynamics of vortex filaments in a regime that differs from those considered in \cite{CzubakJerrard}, \cite{ParisePigatiStern}, in which we are able to resolve the fine structure of multiple interacting filaments.


Our starting point is a classification of solutions $(\Phi, A_a): \R^2\to \C\times \R^2$ of the critical 2d system
\begin{equation}\label{2dAHM}
\left. \begin{aligned}
-  D_aD_a \Phi
+\frac 1 2(|\Phi|^2-1)\Phi  &=0\\
\- \pp_a F_{ab} - (i\Phi, D_b\Phi)  &=0
\end{aligned}\ \ \right\}
\end{equation}
where indices $a,b,...$ run from 1 to 2 and repeated indices are implicity summed.  In Section \ref{sec:2dahm} we recall in detail this classification and other relevant background. 
Briefly, one can associate an integer $N$, called the {\em vortex number}, to every solution of \eqref{2dAHM}, indeed to every finite-energy pair $(\Phi, A_a)$ on $\R^2$. The space of finite-energy solutions\footnote{More precisely, gauge equivalnce classes of finite-energy solutions, see \eqref{gauge.equivclass}.} with vortex number $N$ is parametrized by a Riemannian manifold called the {$N$-vortex moduli space}, denoted $(M_N, g)$, that we describe in Section \ref{subsec:mods}. Given a point $q\in M_N$, we will write
$\modu(x;q) = \binom{\mphi(x;q)}{\mA_a(x;q)}$ to denote the corresponding solution. The {\em vortex centers}  associated to $\modu(\cdot;q)$
are exactly the zeroes of $\mphi(\cdot;q)$. For $q\in M_N$, there are exactly $N$ vortex centers, counting multiplicity. 

The construction and classification of solutions of \eqref{2dAHM} was completed by Jaffe and Taubes \cite{JaffeTaubes} in 1980. Shortly thereafter, Manton \cite{Manton} proposed that when $\lambda =1$, low-energy solutions of the Abelian Higgs model \eqref{AHM} in $2$ space dimensions could be of the form
\beq\label{Manton}
\binom{\Phi}{A_a}_{a=1,2} (x_1,x_2,t) \approx  \modu(x_1, x_2 ; q_\ep(t))
\eeq
where $q_\ep(t) = q(\ep t)$ and $t\mapsto q(t)\in M_N$ is a geodesic in $(M_N, g)$, that is, a solution of
\[
\nabla^g_t \pp_t q = 0, \qquad\mbox{where $\nabla^g$ is the Levi-Civita connection on $(M,g)$}.
\]
Stuart \cite{Stuart} gave a rigorous proof of this scenario for vortex number $N=2$. In fact his results also cover the near-critical case  and $\lambda = 1+\kt\ep^2$ for $\kt\in \R$, in which case the geodesic flow is modified by adding a potential reflecting inter-vortex forces, not present when $\lambda=1$, resulting in
\beq\label{potential}
\nabla^g_t \pp_t q + \kt \nabla_g V_0(q)= 0, \qquad V_0(q) := \frac 1 8 \int_{\R^2}(|\mphi(q)|^2-1)^2,
\eeq
where $\nabla_g$ denotes the gradient with respect to the Riemannian structure on $(M,g)$.
Later work of Palvelev \cite{Palvelev1}, see also \cite{Palvelev2, PalvelevSergeev}, extended Stuart's work in the critical case to cover  arbitrary vortex number $N$.

For the Abelian Higgs model \eqref{AHM} with $\lambda=1$, in our main theorem we fix a vortex number $N\ge 1$, and  we construct solutions in $d\ge 3$ space dimensions such that for every $(x_3,\ldots, x_d , t) = (x_{\bar a}, t)$
\beq\label{pre-thm}
\binom{\Phi}{A_a} (\cdot, x_{\bar a} ,t) \approx \modu( \cdot , q_\ep(x_{\bar a},t)) ,   
\eeq
where\footnote{Here and below, $A_a$ denotes $(A_a)_{a=1,2}$ and $x_{\bar a}$ denotes $(x_{\bar a})_{\bar a=3}^d$. In addition, we implicitly sum over repeated indices.} $q_\ep(x_{\bar a}, t) = q(\ep x_{\bar a}, \ep t)$ and $(x_{\bar a} ,t)\mapsto q(x_{\bar a},t)\in M_N$ is a wave map into $(M_N, g)$, that is, a solution of
\[
\nabla^g_t \pp_t q - \nabla^g_{\bar a}\pp_{\bar a} q = 0.
\]
Our results also apply to the parabolic equation \eqref{PAH}, with $q$ solving the harmonic map heat flow rather than the wave map equation. 
 
Thus, in both \eqref{AHM} and \eqref{PAH}, when $d=3$,  at every time $t$ and height $x_3$ the
restriction to the horizontal slice  $\R^2 \times \{(x_3, t)\}$ of the $(\Phi, A_1, A_2)$ components of the solution that we construct is close to a stationary $2d$ solution, and the moduli space parameters that characterize these 2d slices are slowly-varying functions of $(x_3, t)$. Fixing $t$, the $N$ (not necessarily distinct) vortex centers on the $2d$ slices at height $x_3$, determined by the parameters $q(x_3,t)$, trace out $N$ curves as $x_3$ varies. As $t$ varies, this ensemble of $N$ curves evolves in a way governed by the appropriate geometric evolution equation into $(M_N,g)$.

The statement and proof of our main theorem for the hyperbolic \eqref{AHM} and parabolic \eqref{PAH} equations  are extremely similar. In order to give a unified treatment, and because the added generality may be of interest, we embed both of these equations\footnote{The Abelian Higgs heat flow as normally written does not include the equation $-\pp_k F_{k0}- (i\Phi, D_0\Phi) =0$, see \eqref{PAH}. In fact, when $\ko>\kz=0$, this equation is a consequence of the other equations, as follows from Lemma \ref{lem:identity} below. Thus, in the pure parabolic case it is redundant but not incorrect to include it in the system \eqref{dAHM}.
}  in a larger family:
\[ 
\label{dAHM}\tag{$dAHM$}
\left. \begin{aligned}
(\kz D_0+ \ep\ko)D_0 \Phi - D_j D_j \Phi+\frac{\lambda_\ep}2(|\Phi|^2-1)\Phi  &=0\\
(\kz\pp_0+\ep\ko)F_{0j} - \pp_k F_{kj} - (i\Phi, D_j\Phi)  &=0, \qquad j=1,\ldots, d\\
-\pp_k F_{k0}- (i\Phi, D_0\Phi) &=0
\end{aligned}\ \ \right\}
\] 
where  
\[
0\le \kz, \ko \le 1, \qquad \ko>0\mbox{ if }\kz=0, \qquad\lambda_\ep = 1+\kt\ep^2 \mbox{ for some }\kt\in \R.
\]
The most interesting cases
are $\kz = 1, \ko=0$ and $\kz=0, \ko=1$. The upper bound $\kz,\ko\le 1$ can always be achieved by rescaling in the $t$ variable. 

The scaling in \eqref{dAHM}, in particular the factor of $\ep$ multiplying $\ko D_0\Phi$, is dictated by the ansatz \eqref{pre-thm}. 
If $\kz=0$ then one can rescale in the $t$ variable to bring the equation into the form \eqref{PAH}, corresponding to a parabolic rescaling  $q_\ep( x_{\bar a}, t) = q(\ep x_{\bar a}, \ep^2 t)$.

In our first main result, assuming $\lambda = 1$, we construct a solution of \eqref{dAHM}  satisfying \eqref{pre-thm}, where $q_\ep(x_{\bar a},t) = q(\ep x_{\bar a}, \ep t)$ for $q:\R^{d-2}\times (0,T)\to M_N$ solving the geometric evolution equation
\begin{equation}\label{dwm}
\kz \nabla^g_{t}\pp_{t}q + \ko \pp_{t}q - \nabla^g_{y_{\bar a}}\pp_{\bar a}q  = 0.
\end{equation}
This is the wave map equation when $0 = \ko<\kz$ and the harmonic map heat flow when $0= \kz<\ko$. We start from a solution of \eqref{dwm} that we will always assume satisfies
\begin{equation}\label{q.compact0}
\mbox{ $\exists$ compact }K\subset M_N\mbox{ such that }q(y_{\bar a}, y_0)\in K\mbox{ for all }(y_{\bar a}, y_0)\in \R^{d-2}\times (0,T).
\end{equation}
As we discuss in Section \ref{sec:2dahm}, $M_N$ is canonically isomorphic to $\C^N\cong \R^{2N}$, and in the canonical coordinates, which we always use, the metric is uniformly comparable to the Eucildean metric on every compact $K\subset M_N$. For purposes of defining Sobolev spaces of maps into $M_N$, once $K$ in \eqref{q.compact0} is fixed, it is thus consistent with the geometry of $(M_N,g)$ to treat maps into $K\subset M_N$ as maps into $\R^{2N}$, with the Euclidean connection. We will always do this.
We assume that the solution $q$ of \eqref{dwm} satisfies
\beq\label{q.decay}
\pp_{\bar a}q  
\in C(0,T; H^{L}(\R^{d-2})),
\qquad
\kz \pp_0 q \in C(0,T; H^{L}(\R^{d-2})).
\eeq
for some $L$  to be specified below. The assumption about $\kz\pp_0q$ is of course
vacuous if $\kz=0$. 
We will write
\beq\label{Theta.def0}
\Theta :=  \| \pp_{\bar a}q\|_{L^\infty_T H^L(\R^{d-2})} +  \kz \| \pp_0 q\|_{L^\infty_T H^L(\R^{d-2})} 
\eeq
where, if $\tau>0$ and $\|\cdot \|_X$ is a norm on a space of functions,
\beq\label{eq:LinfX}
\| v\|_{L^\infty_\tau X} = \| v\|_{L^\infty(0,\tau; X)} := \operatorname{ess\,sup}_{0<t<\tau} \| v(\cdot, t)\|_{X}.
\eeq
We will always assume $T$ is chosen so that $\Theta<\infty$.

Our main result may be summarized as follows:

\begin{theorem}\label{thm:1}
Let $q:\R^{d-2}\times (0,T)\to M_N$ be a solution of \eqref{dwm} satisfying \eqref{q.compact0} and \eqref{q.decay} for $L=2d+11$,
and assume that $\Theta<\infty$.
Then there exists $T_0\in (0,T)$ and $\ep_0>0$ and, for $0<\ep<\ep_0$, a solution $U_\ep  = (\Phi_\ep, A_{\ep a}, A_{\ep \bar a}, A_{\ep 0})$  of \eqref{dAHM} for $\lambda=1$ such that
\beq\label{t1.estimate}
\begin{aligned}
\| \Phi_\ep - \mphi(q_\ep)\|_{L^\infty_{T_0/\ep}W^{2,\infty}(\R^d)} 
+\kz \|\pp_t( \Phi_\ep - \mphi(q_\ep))\|_{L^\infty_{T_0/\ep}W^{1,\infty}(\R^d)} 
&\le C \ep^2, \\
\| A_{\ep a} - \mA_a(q_\ep)\|_{L^\infty_{T_0/\ep}W^{2,\infty}(\R^d)} 
+ \kz \| \pp_t(A_{\ep a} - \mA_a(q_\ep))\|_{L^\infty_{T_0/\ep}W^{1,\infty}(\R^d)} 
&\le C \ep^2 \\
\| A_{\ep \alpha} \|_{L^\infty_{T_0/\ep}W^{2,\infty}(\R^d)}
+\kz \| \pp_t A_{\ep \alpha} \|_{L^\infty_{T_0/\ep}W^{1,\infty}(\R^d)}
&\le C \ep\\
\end{aligned}
\eeq
for $a=1,2$ and $\alpha = 0,3,\ldots, d$, where $q_\ep(x_{\bar a},t) = q(\ep x_{\bar a}, \ep t)$ and
\beq\label{mphi.mA}
\mphi(q_\ep)(x,t) := \mphi(x_a ; q_\ep(x_{\bar a}, t)), \qquad
\mA_a(q_\ep)(x,t) := \mA_a(x_a ; q_\ep(x_{\bar a}, t)).
\eeq
A lower bound for $T_0$ can be estimated from $K$ and $\Theta$, and upper bounds for the other constants depend only on $K,\Theta$ and $T_0$. 
\end{theorem}

For $\alpha = 0, 3,\ldots, d$, the $1$-form coefficients $A_{\ep \alpha}$ have the form $A_{\ep\alpha} = \mA_\alpha[q_\ep] + O(\ep^3)$, where the leading-order term $\mA_\alpha[q_\ep]$ encodes aspects of the geometry of $(M_N,g)$ and the map $q$, see \eqref{U0.form} and \eqref{Aabar.choice}. 

\smallskip

For $d=3$ and $\ko=0$, a version of Theorem \ref{thm:1} was first proved in the PhD thesis \cite{Masoud-thesis} of the first author, and our proof follows the ideas developed there very closely.

Our proof also yields, in addition to \eqref{t1.estimate}, estimates in Sobolev norms such as $L^\infty_{T_0/\ep}H^{\frac d2+2}$, and $L^\infty_{T_0/\ep}H^{\frac d2+1}$ for the time derivatives. These can easily be extracted from \eqref{addlabel}, \eqref{good3}, and \eqref{final.tU}.

\begin{remark}\label{rmk:0}
In Section \ref{sec:mainproofs} we restate the theorem with more details. Although here we have only asserted the existence of a single solution satisfying \eqref{t1.estimate}, the full version of the theorem shows that the initial data implicit in Theorem \ref{thm:1} as formulated above can be perturbed somewhat while preserving the conclusions of the theorem.
That is, if we write $(U_0, U_1) := (U_\ep(\cdot, 0), \pp_t U_\ep(\cdot, o))$ for the solution $U_\ep$ constructed in Theorem \ref{thm:1}, and if we seek a solution $V_\ep$ of \eqref{dAHM} with initial data satisfying
\[
\| V_\ep(\cdot, 0) -  U_1\|_{H^{2\ell+1}(\R^d)} + \kz \| \pp_t V_\ep(\cdot, 0) - U_0\|_{H^{2\ell}(\R^d)}\le c_0\ep^2
\]
for a suitably small $c_0$, and if an additional orthogonality condition is satisfied (see \eqref{cmu.initialdata}), then \eqref{t1.estimate} still holds with the same $T_0$, and the constants $C$ all multiplied by $2$.

The orthogonality condition \eqref{cmu.initialdata} says, heuristically, that the initial data for $q$ gives the initial vortex curve location and velocity that best approximates the vortex structure in the initial data for \eqref{dAHM}. 
\end{remark}

\smallskip

\smallskip

\begin{remark}\label{rmk:1}
The proof is carried out by first constructing an approximate solution $\aU:  \R^{d}\times (0,T_0/\ep)\to \C\times \R^{d+1}$, then finding the actual solution as a perturbation of $\aU$. The approximate solution is constructed iteratively, by a rigorous asymptotic expansion. One step in the expansion requires solving a nonlinear evolution equation. The time $T_0\le T$ arises because the solution of this equation, and hence the approximate solution, may lose regularity or cease to exist before time $T/\ep$. If the approximate solution exists and is smooth enough {\em only} up to time $T_1/\e$, then any $T_0\in (0,T_1)$ is an acceptable choice for the theorem, with the other constants diverging as $T_0\nearrow T_1$.

In the physical case $d=3$, global smooth solutions of \eqref{dwm} from $\R^{d-2}=\R$ into any smooth manifold exist for smooth data, and hence $T$ may be taken arbitrarily large. 
\end{remark}

\begin{remark}
In our iterative construction of an approximate solution, the leading terms inherit $H^L_{loc}$ regularity from $q$, and at each stage of the construction, some regularity is lost. At the final stage we use a fixed-point argument to prove existence of a perturbation $\tilde U$ 
such that $U^\ep +\tilde U$ solves \eqref{dAHM}. This fixed-point argument needs to be carried out in a rather smooth space -- we use the space of maps $\tU$ such that $(\tU, \pp_t\tU) \in L^\infty_{T_0}H^{2\ell+1}\times L^\infty_{T_0}H^{2\ell}$ for $\ell$ such that $2\ell >\frac d2+1$ -- so we need to begin with a large number $L$ of derivatives.

In the statement of the theorem, for simplicity we have chosen a value of $L$ that works in all dimensions but is sometimes larger than necessary. We also have not made a great effort to keep $L$ as small as possible, and some of our arguments could be sharpened to allow for a smaller $L$. 
\end{remark}

In near-critical regime $\lambda = 1+\kt\ep^2$, the equation governing $q$ is expected to be modified by the addition of a potential:
\begin{equation}\label{dwmp}
\kz \nabla^g_{t}\pp_{t}q + \ko \pp_{t}q - \nabla^g_{y_{\bar a}}\pp_{\bar a}q  +\kt \nabla_g V_0(q)= 0
\end{equation}
for $V_0$ as in \eqref{potential}.  Theorem \ref{thm:1} as stated does not extend to this case, as the assumption $\Theta <\infty$, used throughout our proof, is
generally incompatible with equation \eqref{dwmp}, at least when $\kz>0$. For example, if the initial data for \eqref{dwmp} is independent
of $y_{\bar a}\in \R^{d-2}$, the solution will have the form $q(y_{\bar a}, y_0) = q_0(y_0)$, where $q_0$ solves
the ODE  obtained by setting $\pp_{\bar a}q=0$ in \eqref{dwmp}.  In this case $\Theta = + \infty$ whenever it is nonzero, if $\kz>0$.

It would be interesting for this and other reasons to relax the assumption in Theorem \ref{thm:1} that $\| \pp_0 q\|_{L^\infty_T H^L(\R^{d-2})}<\infty$.

However, the following variant of Theorem \ref{thm:1} holds in the near-critical case. In the statement we use the notation
$
\T^{d-2}_R :=\R^{d-2}/R\Z^{d-2}$.


\begin{theorem}\label{thm:2}
Fix $N\in \N$, $\kt\in \R$ and $R>0$, and let $q: \T^{d-2}_R\times (0,T)\to M_N$ be a solution of \eqref{dwmp} such that
\begin{equation}\label{q.compact2}
\mbox{ $\exists$ compact }K\subset M_N\mbox{ such that }q(y_{\bar a}, y_0)\in K\mbox{ for all }(y_{\bar a}, y_0)\in  \T^{d-2}_R\times (0,T)
\end{equation}
and
\beq\label{Theta.def2}
\Theta :=  \| \pp_{\bar a}q(\cdot, t)\|_{L^\infty_T H^L(\T^{d-2}_R)} +  \kz \| \pp_0 q(\cdot, t)\|_{L^\infty_T H^L(\T^{d-2}_R)}  <\infty
\eeq
for $L=2d+11$.
Then there exists $T_0\in (0,T)$ and $\ep_0>0$ and, for $0<\ep<\ep_0$, a solution $U_\ep  = (\Phi_\ep, A_{\ep a}, A_{\ep \bar a}, A_{\ep 0})$  of \eqref{dAHM} for $\lambda=1+\kt\ep^2$ such that
\begin{align*}
\| \Phi_\ep - \mphi(q_\ep)\|_{L^\infty_{T_0/\ep}W^{2,\infty}(\R^2\times \T^{d-2}_{R/\ep})} &\le C \ep^2, \\
\| A_{\ep a} - \mA_a(q_\ep)\|_{L^\infty_{T_0/\ep}W^{2,\infty}(\R^2\times \T^{d-2}_{R/\ep})} &\le C \ep^2, \qquad a=1,2 \\
\| A_{\ep \alpha} \|_{L^\infty_{T_0/\ep}W^{2,\infty}(\R^2\times \T^{d-2}_{R/\ep})} &\le C \ep , \qquad \alpha = 0,3,\ldots, d.
\end{align*}
A lower bound for $T_0$ can be estimated from $K$ and $\Theta$, and upper bounds for the other constants depend only on $K,\Theta$ and $T_0$. 
\end{theorem}

\subsection{Physical consequences}

As noted above, earlier works \cite{CzubakJerrard, ParisePigatiStern} characterize the evolution of vortex filaments in solutions of \eqref{AHM}, \eqref{PAH}. In both cases these evolutions can lead to collisions of vortex filaments in 3 space dimensions, and it is natural to ask what the outcome of such a collision might be. 
%
%
One possibility is that the filaments might simply pass through each other with minimal interaction. Another is that each filament might split into two halves at the moment of collision, with the connectivity of the two halves of the filaments switching.

\begin{figure}[ht!]\label{fig:one}
\centering
\includegraphics[scale=0.5]{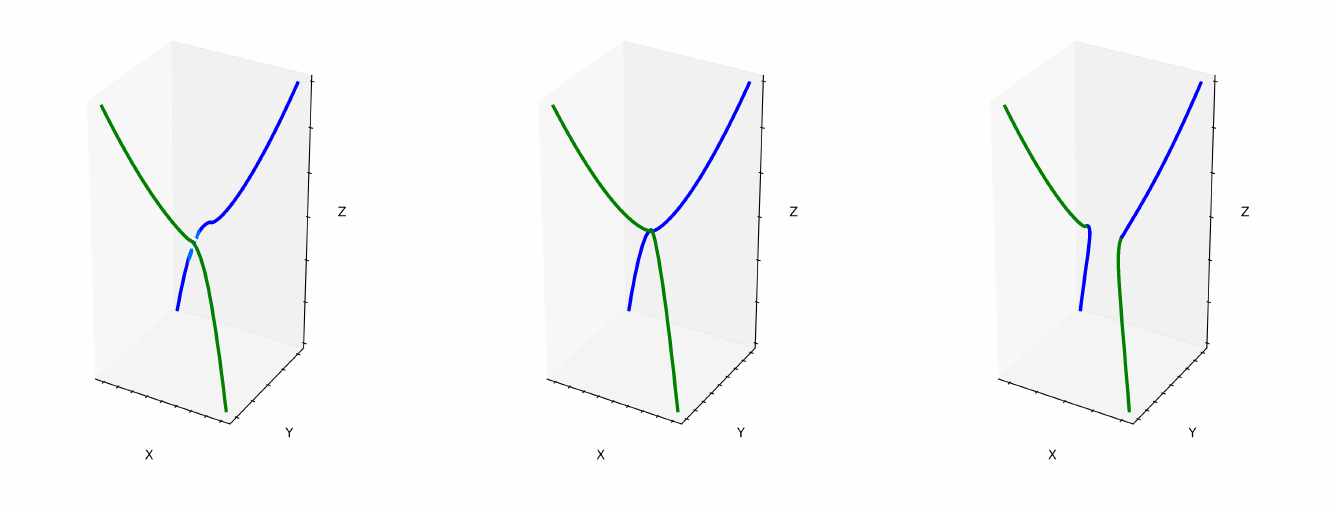}
\caption{filament collision and reconnection}
\label{outcomes}
\end{figure}

This second scenario, pictured in Figure \ref{fig:one},  is known as {\em reconnection}\footnote{also sometimes called intercommutation, recombination, or vortex cutting.}, and it is believed, based on physical arguments and extensive experimental and numerical evidence, to occur very widely in both quantum 
and classical fluids, see for example \cite{KoplikLevine},\cite{lathropetal} and \cite{oshima1977interaction}, \cite{KidaTakaoka}, \cite{saffman1990model}, \cite{kleckner2013creation} respectively. In particular, empirical evidence for the ubiquity of reconnection in fluids such as liquid helium, air, and water is abundant and convincing, though rigorous proof is completely lacking. Collisions of vortex filaments in the Abelian Higgs model \eqref{AHM} have been intensively studied in the physics community, motivated by models of cosmic strings and their possible cosmological signatures, and here too the evidence points to the inevitability of reconnection, at least for collisions at subrelativistic velocities, see for example \cite{shellard1987cosmic}, \cite{Matzner}, \cite{HananyHashimoto}. Similarly, numerical studies, see for example \cite{bou2001vortex}, \cite{Superconductors},  suggest that reconnection should be the generic outcome of vortex filament collisions in superconductors, as modelled by \eqref{PAH} and related equations. 

Prior mathematical work on \eqref{AHM} and \eqref{PAH} does not yield any information  as to whether vortex filament collisions result in reconnection. The only rigorous work of which we are aware on vortex dynamics in  \eqref{AHM}  (see \cite{CzubakJerrard}) ceases to hold before the first vortex collision, and the strongest results we know of on vortices in \eqref{PAH} (see \cite{ParisePigatiStern})  yield a description of vortex filament dynamics that holds globally in time and allows for collisions, but loses uniqueness when collisions occur and is compatible with either reconnection or its nonoccurence.

Our main result about reconnection is a straightforward corollary of Theorems \ref{thm:1} and \ref{thm:2}, together with standard
facts about the geometry of the moduli space $M_N$. The natural setting is $\R^3$, and for simplicity we consider solutions of \eqref{dAHM} with $2$ vortex filaments, as illustrated in Figure \ref{outcomes}, corresponding to maps into the $N=2$ Moduli space $M_2$.  In stating the theorem below we use the fact, reviewed in Section \ref{sec:2.1} below, that $M_2\cong \C^2$, and we use a particular global choice of coordinates for $M_2$ described in Theorem \ref{thm1.new}.

\begin{theorem}\label{thm:3}
Assume that $d = 3$. Let $q = (q_1, q_2): \R\times (0,T) \to M_2\cong \C^2$ be a solution of \eqref{dwm} satisfying \eqref{q.compact0} and \eqref{q.decay} with $\Theta<\infty$, and let $U_\ep:\R^3\times (0,T_0/\ep)\to \C\times \R^{3+1}$ be the solution of \eqref{dAHM} with $\lambda=1$ constructed in Theorem \ref{thm:1}.
Further assume that there exists $(\bar x_3, \bar t) \in \R\times (0,T_0)$ such that  $D = q_1^2-4q_2$ satisfies
\beq\label{suff.reconn}
D(\bar x_3, \bar t) = 0 , \qquad \{ \pp_t D(\bar x_3, \bar t) , \pp_3 D(\bar x_3, \bar t)\}\mbox{ are linearly independent over $\R$} .
\eeq

Then for all $\delta>0$, there exists $\ep_0 > 0$ such that the vortex filaments of $U_\ep$  reconnect in
\[
\{ (x_3,t) : |x_3 - \bar x_3|< \frac \delta\ep, \ |t - \bar t|< \frac \delta \ep\} \qquad\mbox{ if }0<\ep<\ep_0.
\]

The same result holds in the setting of Theorem 2.
\end{theorem}

The proof consists in large part of a discussion intended to clarify what we mean by reconnection.

\begin{remark}
Although we require  $(q_{x_3},\kz q_t)|_{t=0}\in H^{12}\times H^{12}$  when $d=3$  (see \eqref{fix.parameters}), condition \eqref{suff.reconn} is open with respect to  topologies such as $(q, \kz \pp_t q)|_{t=0}\in H^2\times H^1$, in the sense that a small $H^2\times H^1$ perturbation of the initial data will lead to the existence of a perturbed point $(\tilde x_3, \tilde t)$ such that $|(\tilde x_3, \tilde t) - (\bar x_3, \bar t)|$ is small and \eqref{suff.reconn} holds at $(\tilde x_3, \tilde t)$.

Moreover, as noted in Remark \ref{rmk:0}, the estimates of Theorem \ref{thm:1} are stable with respect to perturbations of order $\ep^2$ in  $H^5\times H^4$, as long as they satisfy compatibility and gauge conditions appearing in \eqref{modified.data} and an orthogonality condition appearing in 
\eqref{cmu.initialdata}. This last condition says in effect that the perturbations leave $(q,\kz \pp_tq)|_{t=0}$ invariant. 

Theorems \ref{thm:3} and \ref{thm:1} thus provide some stability for the phenomenon of reconnection with respect to perturbations, albeit rather smooth ones.
\end{remark}

We know of only a few other works that give rigorous treatments of issues connected to reconnection. A reduced model for a single vortex filament reconnecting with a domain boundary is derived formally in \cite{CJO} and rigorously analyzed in \cite{MerleZaag}. Solutions of the Navier-Stokes  and Gross-Pitaevskii equations with reconnecting vortex filaments are constructed in \cite{EncisoPeraltaSalas1}, \cite{EncisoPeraltaSalas2}.

\subsection{Notation}

If it is unclear which connection $1$-form is used to form a covariant derivatives, we will explicitly write the $1$-form as a subscript, rather than simply indicating the index. For example, when considering a perturbation $A+\ep^2\tA$ of a $1$-form $A$, we may write
\[
D_{A_a +\ep^2 \tA_a}\Phi = D_{A_a} - i\ep^2 \tA_a \Phi
\]

We implicitly sum over repeated indices. 
We do not raise and lower indices with a metric tensor, and we do not insist that 
lower indices be paired with upper indices in implicit sums.
We impose the conventions that
\begin{itemize}
    \item indices $i,j,k$ run from $1$ to $d$. These correspond to all spatial directions. They will appear infrequently.
    \item once we fix a vortex number $N$ and consider maps into $M_N$ or perturbations thereof (see Section \ref{subsec:mods} below), indices $\mu,\nu, \ldots$ are understood to run from $1$ to $2N$, the dimension of $M_N$.
    \item indices $a,b, \ldots$ run from $1$ to $2$. These may be thought of as corresponding to ``horizontal" directions, roughly orthogonal to the vortex filaments.
    \item indices $\bar a, \bar b, \ldots$ run from $3$ to $d$.
\end{itemize}
We may sometimes explicitly indicate sums, as a reminder of the range over which the sum runs.

We will often write $U$ to denote an element of $\C\times \R^{d+1}$ or a function taking values in $\C\times \R^{d+1}$,
with components typically written
\[
U = \begin{pmatrix}
    \Phi\\A
\end{pmatrix} =
    \begin{pmatrix}
    \Phi\\
    A_a\\
    A_{\bar a}\\
    A_0
\end{pmatrix} ,\qquad\mbox{ also written as }\quad U = \begin{pmatrix}
    u\\
    A_{\bar a}\\
    A_0
\end{pmatrix}, \qquad\mbox{ where }u = \binom{\Phi}{A_a}.
\]
Here $A_a$ denotes $(A_a)_{a=1,2}$ and $A_{\bar a}$ denotes $(A_{\bar a})_{\bar a=3,\ldots L, d}$.

Given a map $U:\R^d\times (0,T)\to \C\times \R^{d+1}$, we will write
\begin{equation}\label{S.def}
\begin{aligned}
S_\phi[U]& := 
(\kz D_0+ \ep\ko)D_0 \Phi - D_j D_j \Phi+\frac12(|\Phi|^2-1)\Phi  \\
S_j[U] &:= 
(\kz\pp_0+\ep\ko)F_{0j} - \pp_k F_{kj} - (i\Phi, D_j\Phi) , \qquad j=1,\ldots, d\\
S_0[U] &:= -\pp_k F_{k0}- (i\Phi, D_0\Phi)
\end{aligned}\ 
\end{equation}
and
\beq\label{S.components}
S[U] 
=  \begin{pmatrix}
    S_\phi[U]\\
    S_a[U]\\
    S_{\bar a}[U]\\
    S_0[U]
\end{pmatrix}
\ = \  \begin{pmatrix}
    S_u[U]\\
    S_{\bar a}[U]\\
    S_0[U]
\end{pmatrix}, \quad \mbox{ with } S_u[U] = \binom {S_\phi[U]}{S_a[U]} .
\eeq
With this notation,  \eqref{dAHM} can be written concisely as
\[
S[U]=0.
\]

We will often encounter functions with exponential decay in the $x_a\in \R^2$ directions. For these we will use the weighted norm
\beq\label{weightednorm}
\| f \|_{W^{k,\infty}_\gamma(\R^d)}
:= \sum_{|r|\le k}\sup_{y\in \R^d} |\pp^r f(y)| e^{\gamma|x_a|}
\eeq
where the sum is over multiindices $r = (r_1,\ldots, r_d)$ on $\R^d$.
It is clear that if $\gamma_1+\gamma_2\ge \gamma$, then
\[
\| f g \|_{W^{k,\infty}_\gamma(\R^d)} 
\le C_k \| f \|_{W^{k,\infty}_{\gamma_1}(\R^d)}\| g \|_{W^{k,\infty}_{\gamma_2}(\R^d)}
\]
In particular $W^{k,\infty}_\gamma(\R^d)$ is an algebra for $\gamma\ge 0$.

For $k\ge j\ge 0$ we will write $X^{k,j}_{\gamma}$ to denote $H^k\cap W^{k-j,\infty}_\gamma$, which has the natural norm
\beq\label{Xkj}
\| \psi\|_{X^{k,j}_{\gamma}} := \| \psi\|_{H^k} +\| \psi\|_{W^{k-j,\infty}_\gamma}.
\eeq

\subsection{outline of the paper, and elements of the proof}

Our overall strategy is similar to the one used in \cite{delPino-J-Musso} for a semilinear scalar wave equation.
We seek a solution of the form
$U_\ep = \aU +\tU$, 
where $\aU$ is an approximate solution that we construct by a rigorous asymptotic expansion, and $\tU$ solves a perturbation equation.
Thus we aim to solve $S[\aU+\tU]=0$. By expanding the differential operator $S[\, \cdot\,]$ around $\aU$ and rearranging, we rewrite the equation in the form
\beq\label{lin1}
S'[\aU](\tU) =  -S[\aU] - \calN_0[\aU](\tU)
\eeq
where $S'[\aU]$ is a linear differential operator, displayed in \eqref{Sprime}, with coefficients that depend on $\aU$, and $\calN_0[\aU](\tU)$ is a remainder in which we collect term that are nonlinear in $\tU$.

To carry this out, we start in Section \ref{sec:2dahm} by reviewing background about the 2d Abelian Higgs model.
In Section \ref{sec:construction} we construct the approximate solution $\aU$, starting from a solution $q$ of \eqref{dwm}. We require it to have several properties: 
\begin{itemize}
\item The $(\aPhi, \aAa)$ components must be close to $(\mphi(q_\ep), \mA_a(q_\ep))$ as defined in \eqref{mphi.mA}.
\item The error term $S[\aU]$, as defined in \eqref{S.def}, must be small -- this is what it means for $\aU$ to be an approximate solution.
\item The differential operator $S'[\aU]$ needs certain good properties. Some of these follow from the closeness of
$(\aPhi, \aAa)$ to $(\mphi(q_\ep), \mA_a(q_\ep))$, but we also require that $\aU$ be sufficiently smooth.
\end{itemize}

In Section \ref{sec:auxsys} we rewrite \eqref{lin1} in a form more suitable for estimates, which end up being
\beq\label{lin2}
(\kz \pp_0^2+ \ep\ko\pp_0)\tU + \calM\tU + P\tU = -\calN(\tU) - S[\aU]    
\eeq
where $\calM = \calM[\aU]$ is a second-order elliptic operator and $P$ is a first-order operator with small coefficients.

In Section \ref{sec:linests} we establish estimates for a linear model of \eqref{lin1}:
\[
(\kz \pp_0^2+ \ep\ko\pp_0)\tU + \calM\tU + P\tU = \eta
\]
where we regard the right-hand side as given. 

In Section \ref{sec:mainproofs} we use the linear estimates of Section \ref{sec:linests} in a contraction mapping argument to prove existence and estimates of $\tU$ solving \eqref{lin1}, thereby completing the proofs of Theorems \ref{thm:1} and \ref{thm:2}.
This argument determines the degree of smoothness needed in the linear estimates of Section \ref{sec:linests}, and hence the smoothness required for the approximate solution $\aU$, which appears in the coefficients of the modified system \eqref{lin2}. It also imposes constraints on the smallness we require of $S[\aU]$.

Section \ref{sec:reconnection} contains the proof of Theorem \ref{thm:3}, and the proofs of various auxiliary facts, such as linear estimates used in the construction of $\aU$, are presented in an appendix.

Gauge symmetry \eqref{gauge}  gives rise to analytical issues at several stages in our argument. For example, \eqref{dAHM} as rewritten in \eqref{lin1} retains its gauge symmetry. This has well-known unpleasant analytic consequences, including non-uniqueness of solutions and the impossibility of estimating any quantity that is not gauge-invariant. In passing from \eqref{lin1} to \eqref{lin2}, 
we modify the equation in a way that enforces a choice of gauge, guarantees uniqueness of solutions and facilitates estimates. The gauge we use, see \eqref{dynamic.gauge}, may be described as  a perturbative Lorenz gauge in the hyperbolic case and a perturbative caloric gauge in the parabolic case. In the general cases it contains elements of both. We have not seen it in the prior literature. 

The necessity of fixing a gauge also arises implicitly in the construction of $\aU$; it can be seen in Lemmas \ref{lem:J2} and \ref{lem:Sk0.go}, as well as conclusion \eqref{system.gauge2} of Proposition \ref{system.pde}. The gauge we use in this construction, which may be described as a perturbative horizontal Coulomb gauge, differs from that built into the modified system \eqref{lin2}.

A key subtlety arises in the linear estimates of Section \ref{sec:linests}, due to the failure of the quadratic form associated to the operator $\calM$ in \eqref{lin2} to fully control the $L^2$ norm of mappings on which it operates. This necessitates estimates of a projection of $\tU$ onto the poorly-controlled part, which we write 
as 
\[
\sum_\mu c_\mu n_\mu = \sum_\mu c_\mu(y_{\bar a}, y_0) n_\mu(y_a; q_\ep(y_{\bar a }, y_0))\ ,
\]
where $n_\mu:\R^2\times M_N\to \C\times \R^{2}$ for $\mu=1,\ldots, 2N$ are mappings introduced in Section \ref{sec:2dahm} and $(c_\mu):\R^{d-2}\to \R^{2N}$ are coefficients arising from the projection of $\tU$.
The most delicate part of the linear estimates involves controlling norms of $c_\mu$. These functions satisfy a wave equation 
of the form
\[
        (\kz \pp_{tt} +\ep \ko \pp_t - \Delta_{\bar a}) c_\mu = H
\]
where, roughly speaking, $H = O(\ep)$. It follows from this that
\[
\| c_\mu(\cdot, t)\|_{L^2}  \approx t^3\int_0^t  \| H (\cdot, t)\|_{L^2}^2 \, dt \approx t^4\ep^2
\]
which is too weak to control solutions for times of order $1/\ep$, as we require. The same issue appears when trying to control higher derivatives. We address this problem by carefully exploiting the structure of the operator $P$ in \eqref{lin2} to rewrite the equation satisfied by $c_\mu$ in the form
\[
        (\kz \pp_{tt} +\ep \ko \pp_t - \Delta_{\bar a}) c_\mu =
        \pp_0 h_0 + \nabla_{\bar a} \cdot h_{\bar a} + h
\]
where now $h = O(\ep^2)$. This and the divergence form of the other terms allow us to establish estimates showing that $c_\mu$ remains small for times of order $1/\ep$, as needed. These arguments appear in Section \ref{sec:Qtop}.

Our strategy differs from that employed by Stuart \cite{Stuart}, who employs a modulation argument to obtain an exact splitting of a solution of \eqref{AHM} into a part in the moduli space and an orthogonal part. He then estimates the growth of the orthogonal part. We believe that this would be very hard to carry out in $3$ or more space dimensions due to loss of derivatives in the modulation equation. Nonetheless, our work relies heavily on the foundational ideas developed in \cite{Stuart}.

\subsection{A useful lemma}

We conclude this introduction with a lemma that is used in enforcing the choice of gauge built into the modified system \eqref{lin2}.
It also shows that, as noted above,
the equation $S_0[U] =0$ is redundant but not inconsistent in the parabolic case.

\begin{lemma}\label{lem:identity}
For every sufficiently smooth $U = \binom \Phi A$,
\begin{equation}\label{identity}
( S_\phi[U], i\Phi) = -(\kz\pp_0 +\ep\ko)S_0[U] + \pp_j S_j[U]
\end{equation}
where $j$ is summed from $1$ to $d$.
In particular this holds for $U$ of the form
$U = \aU +\tU$, where
\beq\label{suff.smooth1}
\aU\in L^\infty_T W^{2\ell+1,\infty}\mbox{ with }
\pp_t\aU\in L^\infty_T W^{2\ell,\infty}\mbox{ and }
\pp_{tt}\aU\in L^\infty_T W^{2\ell-1,\infty}
\eeq
and
\beq\label{suff.smooth2}
\tU\in L^\infty_T H^{2\ell+1}\mbox{ with }
\pp_t\tU\in L^\infty_T H^{2\ell}\mbox{ and }
\pp_{tt}\tU\in L^\infty_T H^{2\ell-1}.
\eeq
with $2\ell > d/2$.
\end{lemma}

We emphasize that \eqref{identity} is an algebraic identity that does not rely on $U$ solving any equation.

\begin{proof}
We compute
\begin{align*}
( S_\phi[U], i\Phi) &= ((\kz D_0+ \ep\ko)D_0 \Phi,i\Phi)  - (D_jD_j \Phi , i\Phi)\\
&= (\kz\pp_0+\ep\ko)(D_0\Phi, i\Phi) - \pp_j(D_j\Phi, i\Phi).
\end{align*}
On the other hand, the antisymmetry of $F_{jk}$ implies that
\[
\pp_j S_j[U] = - \pp_j (i\Phi, D_j\Phi) +(\kz\pp_0+\ep\ko) \pp_j  F_{0j},
\]
and it is clear that
\[
-(\kz \pp_0 +\ep\ko) S_0[U] = (\kz \pp_0 +\ep\ko)(i\Phi, D_0\Phi)  + (\kz\pp_0+\ep\ko)  \pp_j F_{j0}.
\]
The lemma follows directly from these three identities, since $F_{j0}=-F_{0j}$.

It is straightforward to check that \eqref{suff.smooth1}, \eqref{suff.smooth2} sufficient to justify the computations above.
\end{proof}

The above lemma has the following consequence, which we will use later to fix the gauge in the perturbation equation \eqref{lin1}.  

\begin{lemma}\label{lem:Z}
Assume that $U = \binom \Phi A : \R^{d+1}\to \C\times \R^{d+1}$ and $Z: \R^{d+1}\to \R$ are sufficiently smooth solutions of the equations
\begin{equation}\label{Z.system1}
\left.\begin{aligned}
S_\phi[U] &= -i\Phi Z\\
S_j[U] &= - \pp_j Z\\
S_0[U] &= -\pp_0 Z
\end{aligned} \ \ \right\}
\end{equation}
Then $Z$ solves
\beq\label{Z.eqn}
(\kz\pp_0+\ep\ko) \pp_0Z - \Delta Z + |\Phi|^2 Z =0.
\eeq
In particular this holds if $U$ has the form
$U=\aU+\tU$ for $\aU$ and $\tU$ satisfying \eqref{suff.smooth1}, \eqref{suff.smooth2}, and if 
\beq\label{suff.smooth3}
Z\in L^\infty_T H^{2\ell}\quad\mbox{ with } \ \ Z_t\in L^\infty_T H^{2\ell-1}
\eeq
for $2\ell>d/2$.

Finally, assume in addition $Z  = \kz Z_t = 0$ at $t=0$, and that there exists some 
$c>0$ such that for every $t\in (0,T)$
\beq\label{unif.gap}
\int_{\R^d} |\nabla f(x)|^2 +| \Phi(x,t)|^2 f(x)^2\,dx\  \ge \  c \int_{\R^d} f^2(x)\, dx \qquad\mbox{ for all }f\in H^1(\R^d)
\eeq
then $Z$ vanishes identically.
\end{lemma}

\begin{proof}
Equation \eqref{Z.eqn} follows directly from combining \eqref{identity} and \eqref{Z.system1}. It is straightforward to see that \eqref{suff.smooth1}, \eqref{suff.smooth2}, \eqref{suff.smooth3} are sufficient to justify these computations.

Next, we define
\[
e(t) := \int_{\R^d} \kz (\pp_0 Z)^2 +|\nabla Z|^2 +|\Phi|^2 Z^2\, dx + \kz \int_{\R^d} (\pp_0 Z)^2\, dx\Big|_{t}
\]
We multiply \eqref{Z.eqn} by $\pp_{0}Z$ and integrate by parts to find that
\beq\label{Z.energy1}
\frac 12
\frac d{dt} e(t)
= \int_{\R^d} \frac 12 (\pp_0|\Phi|^2)Z^2\, dx 
\eeq
In view of a Sobolev embedding and our assumption that $2\ell>\frac d2$,
\begin{align*}
\| \pp_t U \|_{L^\infty_TL^\infty} &\le
\| \pp_t \aU \|_{L^\infty_TL^\infty} +
\| \pp_t \tU \|_{L^\infty_TL^\infty} \\
&\le\| \pp_t \aU \|_{L^\infty_TW^{2\ell, \infty}} +
\| \pp_t \tU \|_{L^\infty_TH^{2\ell}} .
\end{align*}
It easily follows that $\| \pp_t|\Phi|^2 \|_{L^\infty(\R^d\times (0,T))} \le C$.
Thus if \eqref{unif.gap} holds, we can integrate \eqref{Z.energy1} from $0$ to $t$ to find that
\[
e(t) \le  e(0) + C \int_0^t\int_{\R^d} Z^2\, dx\, dt \le e(0)+ C\int_0^t e(s)\,ds.
\]
If in addition $Z = \kz \pp_t Z=0$ at $t=0$, then $e(0)=0$, and Gr\"onwall's ineqiuality implies that 
$e(t)=0$ for all $t$.
\end{proof}

\section{the 2d Abelian Higgs model}\label{sec:2dahm}

The 2d Euclidean Abelian Higgs model is a system of equations for a pair $u = (\phi, A):\R^2\to \C\times \R^2$
obtained as the Euler-Lagrange equations of the functional
\begin{equation}\label{2denergy}
E(u) := \int_{\R^2} e(u)\, dx, \qquad e(u) := \frac 12 \sum_{a=1}^2 |D_a \phi|^2 + \frac 12 (F_{12})^2 + \frac 18(|\phi|^2-1)^2,
\end{equation}
where $F_{12} = \pp_1A_2-\pp_2A_1$ and $D_a = \pp_a - i A_a$. We will sometimes tacitly identify $A$ with the $1$-form $A_1 dx_1+A_2 dx_2$.
We define the {\em energy space} 
\[
\calE :=  \left\{ u\in H^1_{loc}(\R^2; \C\times \R^2) : \int_{\R^2} e(u) <\infty \right\}.
\]
For $u\in\calE$. we define the current
\[
j(u) := ( (i\phi, D_1\phi), (i\phi, D_2\phi))
\]
and the vorticity
\[
\omega(u) : = \ \frac 12 \big[ \partial_1 j_2(u) - \partial_2j_1(u) + F_{12}\big].
\]
The following identity is due to Bogomolnyi:
\beq\label{Btrick}
e(u)=\pm \omega(u) + \frac 12 \abs{(D_{1}\pm iD_{2})\phi}^{2}+\frac 12 \big ( F_{12}\pm\frac 1{2}(\abs{\phi}^{2}-1)\big)^{2}. 
\eeq
In particular, if $u\in \calE$, then $\omega(u)\in L^1(\R^2)$. Moreover, every $u\in \calE$ belongs to
exactly one of the sets
\[
\calE_N := \left\{ u\in \calE : \int_{\R^2} \omega(u) \ dy = \pi N\right\}
\]
where $N$ is an integer called the {\em vortex number}\footnote{If $|\phi|^2- 1$  and $D_A\phi\to 0$ sufficiently rapidly at $\infty$, then one can integrate by parts to find that
$\int_{\R^2} \omega = \frac 12\int_{\R^2} F_{12}$. It is traditional to define the vortex number only for $u$ satisfying these decay conditions, and to define it as $\frac 1{2\pi} \int F_{12}$. The definition we employ, following \cite{CzubakJerrard},  is valid for all finite-energy $u$ and agrees with the traditional definition whenever the latter makes sense.}, see for example Lemma 2.4 in \cite{CzubakJerrard}.
Noting that $\omega(\bar \phi, -A) = - \omega(\phi ,A)$, and because the phenomena we wish to study do not appear in $\calE_0$ we will always restrict our attention to $u\in \calE_N$ for $N\ge 1$. Then it follows from \eqref{Btrick} that $u\in \calE_N$ minimizes the energy in $\calE_N$ exactly when
\begin{equation}\label{eq:Bogsys}
    \begin{aligned}
        (D_1 + i D_2)\phi &= 0\\
        F_{12} + \frac 12 (|\phi|^2-1) &=0.
    \end{aligned}
\end{equation}

\subsection{the Moduli space}\label{sec:2.1}

The first three conclusions of the following theorem constitute a mild reformulation of classical results due to Taubes, see \cite{JaffeTaubes}, and the final conclusion is taken from Theorem 2 in \cite{Palvelev2}.

\begin{theorem}
\label{thm1.new}\quad

{\bf 1}. For every $N\ge 1$ and $q = (q_1,\ldots, q_N)\in \C^N$, there exists a unique solution $\modu = (\mphi, \mA) \in \calE_N$ of \eqref{eq:Bogsys}
such that
\begin{equation}\label{eq:TaubesSol1}
\mphi(z) = f(z) \frac{p(z;q)}{\sqrt{1+|p(z;q)|^2}} \quad \mbox { with } f(z)>0 \mbox{ everywhere, }
\end{equation}
where
\begin{equation}\label{pzq.def}
p(z;q) :=  z^N+ q_1z^{N-1} + \cdots + q_N.
\end{equation}
Moreover, $f$ is smooth and satisfies
\beq\label{eq:TaubesSol2}
c\le f(z) \le c^{-1} \qquad\mbox{ for some }c = c(q)\in (0,1).
\eeq
This solution will be denoted $\modu(\cdot; q) = (\mphi, \mA)(\cdot; q)$ when we wish to indicate its dependence on $q$, and for every $q$ it satisfies
\beq\label{eq:coulomb}
\nabla \cdot \mA(\cdot; q) = 0.
\eeq

{\bf 2}. Every solution of \eqref{eq:Bogsys} in $\calE_N$ is gauge-equivalent to $\modu(\cdot; q)$ for some $q\in \C^N$, and every critical point in $\calE_N$ of the $2d$ energy \eqref{2denergy} solves the first-order system \eqref{eq:Bogsys}.

{\bf 3}. For every $\delta>0$ and $k>0$, if $|q|\le K$ then
$\modu(\cdot; q)$ satisfies
\begin{equation}
\label{eq10}
|D_1\mphi|+|D_2\mphi| +\big(1-|\mphi|\big) \le C \exp\big(-(1-\delta)|z|\big).
\end{equation}
for $C= C(K,\delta)$.

{\bf 4}. For every $z\in \C$, $\modu(z;q)$ depends smoothly on $q\in \C^N$.

\end{theorem}

\bigskip

For any $q\in \C^N$, if $z_j$ is a root of $p(\cdot; q)$ of multiplicity $n_j$, then
\eqref{eq:TaubesSol1}, \eqref{eq:TaubesSol2} imply that
\begin{equation}
    \phi(z;q)\sim c_j\big(z-z_j\big)^{n_j}\qquad\mbox{ near }z_j, \ \mbox{ for some $c_j\neq 0$. }
\end{equation}
The roots $\{z_1,z_2,\cdots,z_N\}$, repeated according to multiplicity, are called the vortex centers and $\modu(\cdot;s)$ is called an $N$-vortex solution. (Note that $q_j$ is exactly the $j$th elementary symmetric polynomial of the roots $z_1,\ldots, z_N$.)
The $N$-vortex {\em moduli space} $M_N$ is defined to be the space of equivalence classes of $N$-vortex solutions
modulo the gauge symmetry:
\beq\label{gauge.equivclass}
(\mphi(q), \mA(q)) \sim (e^{i\eta}\mphi(q), \mA(q)+d\eta).
\eeq

We will often identify $\C^n$ with $\R^{2N}$, and we will write  $q_\mu$, $\mu = 1,\ldots, 2N$, to denote (real) coordinates on $M_N$.

It is also convenient to write $\modu = (\mphi, \mA)$ as functions of $x = (x_1, x_2)\in \R^2$ rather than $z = x_1+ix_2\in \C$. From now on, we will thus write for example $\modu(x;q)$
in place of $\modu(z; q)$.

Theorem \ref{thm1.new} implies that $M_N$ is naturally parametrized by $\R^{2N}$, via the  bijections 
\begin{align*}
q\in \R^{2N} \cong\  C^N
&\ \longleftrightarrow \ \mbox{roots (with multiplicity) of }p(\cdot;q)  \\
&\ \longleftrightarrow \ \mbox{the equivalence class of solutions represented by }\modu(\cdot ;q).    
\end{align*}

\subsection{geometry of the moduli space}\label{subsec:mods}

We next describe a Riemannian structure on $M_N$ that arises from its character as a quotient space of a submanifold\footnote{A detailed development of this geometric perspective can be found in Section 5.2 of  \cite{DemouliniStuart}.} of $\calE_N$. 
A precise statement is given in Theorem \ref{thm:MNmetric} below, but the idea may be loosely described as follows: given $v = (v_1,\ldots, v_{2N})\in T_q M_N$, one may find a smooth curve $q(t)$ in $M_N$, defined for $t$ in an interval $I$ containing the origin,  such that $q(0) = q$ and $q'(0)=v$. For any sufficiently smooth $\eta: I\times \R^2\to \R$ such that $\eta(0)=0$, this can be lifted to a curve in $\calE_N$ of the form 
\[
t\mapsto  w_{v,\eta}(t) =(e^{i \eta(t) }\mphi(q(t)), \mA(q(t))+ d\eta(t)).
\]
This lifted curve has tangent vector 
\[
w'_{v,\eta}(0) = (v_\mu \pp_\mu \mphi(q) + i \tilde \chi \mphi(q), v_\mu \pp_\mu \mA(q) + d\tilde \chi) \quad \mbox{ for }\tilde \chi = \pp_t\eta|_{t=0}.
\]
The Riemannian metric $g$ is determined by requiring that $\eta$ be chosen to minimize the $L^2$ norm of $w_{v,\eta}'(0)$. That is, we define
\[
|v|_g^2 = \min \{ \int_{\R^2} |w_{v,\eta}'(0)|^2 \, dx : w\mbox{ as above} \}
\]
It is straightforward to check that this amounts to requiring that $\eta$ be chosen so that the lifted tangent vector $w'_{v,\eta}(0)=: (\tPhi, \tA)$ satisfies
\beq\label{eq:go}
\nabla\cdot \tA - (i\mphi(q), \tPhi) =0
\eeq
This amounts to a PDE for $\tilde \chi  = \pp_t\eta|_{t=0}$, see \eqref{chimu.def} below.
We will refer to this as the {\em gauge orthogonality condition}, as it
states exactly that the lifted tangent vector is orthogonal to all 
{\em infinitessimal gauge transformations} at $\modu(q)$, that is, all vectors of the form
\beq\label{infinis.gt}
(i\tilde \eta \mphi(q) , d \tilde \eta), \qquad \tilde \eta:\R^2\to \R \mbox{ sufficiently smooth}.
\eeq
Such vectors are tangent at $\modu(q)$ to the gauge equivalence class \eqref{gauge.equivclass}.
See also \cite{Stuart}, Section 2 for a more complete discussion.

A precise statement of most of the facts we will need is given in the following

\begin{theorem}\label{thm:MNmetric}
For every $q\in M_N$ and  $\mu = 1,\ldots, 2N$, there exists a unique solution $\chi_\mu(\cdot; q):\R^2\to \R$ of the equation
    \beq\label{chimu.def}
    -\Delta \chi_\mu + |\mphi(q)|^2 \chi_\mu - (i \mphi(q), \pp_\mu \mphi(q)) = 0
    \eeq
    that decays exponentially to $(i \mphi(q), \pp_\mu \mphi(q))$ as $|z|\to \infty$. For this function, the vectors
    \beq\label{nmu.def}
    \tilde n_\mu = \tilde n_\mu(q) = \binom{\tilde n_{\mu,\phi}}{\tilde n_{\mu,a}}(q) :=
 \binom{\pp_\mu \mphi(q) - i\mphi(q)\chi_\mu}{\pp_\mu \mA_a(q) - \pp_a\chi_\mu} 
    \eeq
    have the following properties:
    \begin{enumerate}
        \item[\bf 1.] For every $\mu$, the map $q\mapsto \tilde n_\mu(q)$ is a smooth map from $M_N \cong \R^{2N}$ into $L^2(\R^2)^4$ with values in $H^1(\R^2)^4$.
        \item[\bf 2.] Each $\tilde n_\mu(q)$ satisfies the {\em gauge orthogonality condition}  \eqref{eq:go}.
        \item[\bf 3.] At every $q$, the vectors $\{ \tilde n_\mu\}_{\mu=1}^{2N}$ are linearly independent in $L^2(\R^2)^4$.
         \end{enumerate}
    
\end{theorem}

This is proved in \cite{Palvelev2}, Theorem  2, relying in part on results of \cite{Stuart} and other earlier work.

Given a tangent vector $v = \sum_\mu v^\mu \frac \pp{\pp q_\mu} \in T_q M_N$, we will write
\begin{equation}\label{chiv.def}
\chi_v = \sum_\mu v^\mu \chi_\mu, \qquad \qquad  \tilde n_v := \sum_\mu v^\mu \tilde n_\mu,
\end{equation}
Note that $\chi_v, \tilde n_v$ are independent of the choice of coordinates on $M_N$.

\begin{definition}\label{def:metric}
 For $\frac \pp{\pp q_\mu} ,\frac \pp{\pp q_\nu}\in T_q M_N$,
we define a Riemannian metric $g$ on $M_N$ by setting
\[
g(\frac \pp{\pp q_\mu} ,\frac \pp{\pp q_\nu} ) := g_{\mu\nu}(q) := \int_{\R^2} \tilde n_\mu(q) \cdot \tilde n_\nu(q) \,dx_a = 
(\tilde n_\mu, \tilde n_\nu)_{L^2(\R^2)} 
\]
where we use the shorthand $\tilde n_\mu \cdot \tilde n_\mu = (\tilde n_{\mu,\phi}, \tilde n_{\nu,\phi}) + \tilde n_{\mu,a} \tilde n_{\nu,a}$. We will normally not explicitly indicate the dependence of $g_{\mu\nu}$ on $q$.
\end{definition}

The linear independence of $\{ \tilde n_\mu\}_{\mu=1}^{2N}$ implies that $(g_{\mu\nu})$ is a nonsingular matrix for every $q$, which in turn implies that the coordinates $(q_\mu)$ fixed above are good coordinates for $M_N$ in the sense that the metric tensor is smooth and nondegenerate when written with respect to this coordinate system.

It is sometimes convenient to use the notation
\[
\pp_\mu := \frac \pp{\pp q^\mu}, \qquad \qquad \calD_\mu \mphi(q) := (\pp_\mu - i \chi_\mu)\mphi(q), 
\]
\[
\mF_{\mu\nu} (q):= \pp_\mu \chi_\nu(q) - \pp_\nu\chi_\mu(q), \qquad\qquad
\mF_{\mu a}(q)  := \pp_\mu \mA_a(q)  - \pp_a\chi_\mu(q).
\]
We emphasize that derivatives $\pp_\mu, \pp_\nu$ with indices $\mu,\nu\in 1,\ldots, 2N$ are derivatives  with respect to $q\in M_N$, whereas $\pp_a$ is a derivative with respect to the spatial variables $x_a, a=1,2$.
With this notation,
\[
\tilde n_\mu(q) = \binom {\calD_\mu \mphi(q)}{\mF_{\mu a}(q)}.
\]

\begin{lemma}\label{lem:connection}
Writing $\nabla^g$ to denote the Levi-Civita connection on $(M_N, g)$, the identity
\begin{equation}\label{lifted.nabla}
g( \nabla^g_{\pp_\mu} \pp_\nu , \pp_\lambda) =  
\left(\binom{\calD_\mu \calD_\nu \mphi(q)}{\pp_\mu \mF_{\nu a}(q))} , \tilde n_\lambda(q) \right)_{L^2(\R^2)}
\end{equation}
holds. More generally, for any two vector fields $v = v^\mu\pp_\mu$ and $w = w^\nu \pp_\nu$ on $M_N$, 
\begin{equation}\label{lifted2}
g(\nabla^g_v w, \pp_\lambda) = \left(\binom{\calD_v \calD_w \mphi(q)}{\pp_v \mF_{w a}(q))} , \tilde n_\lambda(q) \right)_{L^2(\R^2)}
\end{equation}
where $\calD_v := v^\mu\calD_\mu$ and $\mF_{w a} = w^\nu \mF_{\nu a}$.
\end{lemma}

\begin{proof}
We will write $J_{\mu\nu} := \binom{\calD_\mu \calD_\nu \mphi(q)}{\pp_\mu \mF_{\nu a}(q))}$. 
Note that for any $\mu, \nu$, 
\beq\label{Jcomm}
J_{\mu \nu} - J_{\nu \mu}  =  -\binom{i \mphi \mF_{\mu\nu}}{ \pp_a \mF_{\mu\nu}}
\eeq
which is an infinitessimal gauge transformation, see \eqref{infinis.gt}, Since the gauge tangent vectors $\tilde n_\lambda$ satisfy the gauge orthogonality condition \eqref{eq:go}, it follows that
\begin{equation}\label{commute}
(J_{\mu\nu}, \tilde n_\lambda)_{L^2(\R^2)} = (J_{\nu\mu}, \tilde n_\lambda)_{L^2(\R^2)}  \qquad\mbox{ for all }\mu, \nu, \lambda \in \{1, \ldots, 2N\}.
\end{equation}
It is clear that
\beq\label{Jmunu}
J_{\mu\nu} = \pp_\mu \tilde n_\nu - \binom{i \chi_\mu \calD_\nu\mphi(q)}{0},
\eeq
so it follows from \eqref{commute} that
\begin{equation}\label{right-handside}
2( J_{\mu\nu}, \tilde n_\lambda)_{L^2} = (\pp_\mu \tilde n_\nu + \pp_\nu\tilde n_\mu,\tilde n_\lambda)_{L^2}
 -(i(\chi_\mu \calD_\nu \mphi + \chi_\nu \calD_\mu\mphi) , \calD_\lambda\mphi)_{L^2}.
\end{equation}
We now consider the right-hand side of \eqref{lifted.nabla}. From the Koszul formula and the definition of the metric,
\[
\begin{aligned}
2 g( \nabla^g_{\pp_\mu} \pp_\nu , \pp_\lambda)  
&= \pp_\mu g_{\nu\lambda} + \pp_\nu g_{\lambda\nu} - \pp_\lambda g_{\mu\nu} \\
&=
\pp_\mu (\tilde n_\nu, \tilde n_\lambda)_{L^2} +
\pp_\nu (\tilde n_\nu, \tilde n_\lambda)_{L^2} -
\pp_\lambda (\tilde n_\mu, \tilde n_\nu)_{L^2}\\
&= 
(\pp_\mu \tilde n_\nu + \pp_\nu\tilde n_\mu,\tilde n_\lambda)_{L^2} \\
&\qquad\qquad\qquad
+
(\pp_\mu \tilde n_\lambda - \pp_\lambda\tilde n_\mu,\tilde n_\nu)_{L^2}
+
(\pp_\nu \tilde n_\lambda - \pp_\lambda\tilde n_\nu,\tilde n_\mu)_{L^2}.
\end{aligned}
\]
In view of \eqref{right-handside}, to prove \eqref{lifted.nabla}  it now suffices to check that
\begin{equation}\label{reducesto}
(\pp_\mu \tilde n_\lambda - \pp_\lambda\tilde n_\mu,\tilde n_\nu)_{L^2}
+
(\pp_\nu \tilde n_\lambda - \pp_\lambda\tilde n_\nu,\tilde n_\mu)_{L^2}
=
 -(i(\chi_\mu \calD_\nu \mphi + \chi_\nu \calD_\mu\mphi) , \calD_\lambda\mphi)_{L^2}.
\end{equation}
To do this, note from \eqref{Jcomm} and \eqref{Jmunu} that
\[
\pp_\mu \tilde n_\lambda - \pp_\lambda\tilde n_\mu
=
-\binom{i\mphi \mF_{\mu\lambda}}{\pp_a \mF_{\mu\lambda}} - \binom{i(\chi_\lambda \pp_\mu\mphi -\chi_\mu \pp_\lambda\mphi)}{0}.
\]
Again using gauge orthogonality, we deduce that 
\begin{align*}
(\pp_\mu \tilde n_\lambda - \pp_\lambda\tilde n_\mu, \tilde n_\nu)_{L^2}
&=
 -\chi_\lambda (i\pp_\mu \mphi, \calD_\nu \mphi)_{L^2}
+ \chi_\mu (i\pp_\lambda \mphi, \calD_\nu \mphi)_{L^2}
\\
&=
 -\chi_\lambda (i\pp_\mu \mphi, \pp_\nu \mphi)_{L^2}
 +\chi_\lambda \chi_\nu (\pp_\mu\mphi,  \mphi)_{L^2}
+ \chi_\mu (i\pp_\lambda \mphi, \calD_\nu \mphi)_{L^2}.
\end{align*}
Since $(\pp_\mu\mphi,  \mphi) = (\calD_\mu \mphi, \mphi) = (i\calD_\mu \mphi, i\mphi)$, we can rewrite as
\[
(\pp_\mu \tilde n_\lambda - \pp_\lambda\tilde n_\mu, \tilde n_\nu)_{L^2}
=
 -\chi_\lambda (i\pp_\mu \mphi, \pp_\nu \mphi)_{L^2}
 - \chi_\nu (-i\chi_\lambda  \mphi, i D_\mu\mphi)_{L^2}
- \chi_\mu (\pp_\lambda \mphi, i\calD_\nu \mphi)_{L^2}.
\]
We arrive 
at \eqref{reducesto} and complete the proof of  \eqref{lifted.nabla}  by permuting the $\mu$ and $\nu$ indices,  adding, and recalling that $(i \pp_\mu\mphi, \pp_\nu \mphi)_{L^2} = -(i \pp_\nu\mphi, \pp_\mu \mphi)_{L^2}$. 

To prove \eqref{lifted2}, note that 
\[
\nabla^g_v \pp_w = v^\mu\nabla^g_\mu(w^\nu\pp_\mu) = v^\mu\pp_\mu w^\nu \pp_\nu + v^\mu w^\nu \nabla^g_\nu \pp_\mu
= \pp_v w^\nu\pp_\nu + v^\mu w^\nu \nabla^g_\nu \pp_\mu
\]
and similarly
\[
\binom{ D_v D_w\mphi(q)}{\pp_v \mF_{wa}(q)} = v^\mu\binom{ D_\mu (w^\nu \calD_\nu \mphi(q)}{\pp_\mu(w^\nu  \mF_{\nu a}(q)}
= v^\mu \pp_\mu w^\nu \,\tilde n_\nu + v^\mu w^\nu \binom{ D_\mu D_\nu\mphi(q)}{\pp_\mu \mF_{\nu a}(q)}.
\]
So \eqref{lifted2} follows from \eqref{lifted.nabla} and Definition \ref{def:metric}.
\end{proof}

\begin{definition}\label{def:nmu}
    We will write $\{ n_\mu(q) \}_{\mu = 1}^{2N}$ to denote a fixed orthonormal basis for the vector space generated by $\{ \tilde n_\mu(q)\}_{\mu=1}^{2N}$ and depending smoothly on $q$, for example defined by $n_\mu := \sum_{\alpha=1}^{2N}A_{\mu\alpha} \tilde n_\alpha$ for the unique symmetric positive definite $(A_{\mu\alpha})$  such that $\sum_{\alpha\beta}A_{\mu\alpha} A_{\nu\beta} g_{\alpha\beta} = \delta_{\mu\nu}$.
\end{definition}

We close this section by recording more results on regularity and decay of the functions introduced above. 

We may consider $\mphi, \mA, n_\mu,$ \ldots as functions on $\R^2\times M_N$. In fact we will always restrict our attention to situations where $q\in K\subset M_N$, where $K$ is a fixed compact set. 
It is clear that for any compact $K\subset M_N$ there exists $C = C(K)$ such for any $q\in K$ and $v = v^\mu \pp_\mu\in T_qM_N$, 
\[
C^{-1}|v|^2_{\R^{2N}} = \delta_{\mu\nu} v^\mu v^\nu
\le 
|v|^2_{g} = g_{\mu\nu}(q)v^\mu v^\nu \le 
C|v|^2_{\R^{2N}}
\]
In view of this, for $q\in K$ we can and will use the Euclidean norm (but writing $| \cdot|$ rather than $|\cdot|_{\R^{2N}}$) in defining norms of derivatives with respect to $q_\mu, \mu=1,\ldots, 2N$
of $\mphi, \mA, n_\mu$ etc, and we will often consider derivatives with respect to the Euclidean structure rather than covariant derivatives on $M_N$.

\begin{lemma}
\label{lem:decayq}
    Let $K$ be a compact subset of $M^n$, and let  $\pp^r$ denote differentiation with respect to $x\in \R^2$ and/or $q\in \R^{2N}\cong M_N$ as indicated by a multi-index $r$.
    Then there exists $C, \gamma$, possibly depending $K$ and $r$, such that for any $q\in K$,
    \begin{align*}
        &|\pp^r\mphi(x,q)| \le C(|x|+1)^{-|r|}, \\
        &|\pp^r\chi_\lambda(x,q)|, |\pp^r\mA(x,q)|\le C(|x|+1)^{-|r|-1},\\
        &|\pp^r(|\mphi(x;q)|^2-1)|, \ |\pp^r(D_{a}\mphi(x;q)|,\   |\pp^r \tilde n_\mu(x;q)|   \le C e^{- \gamma |x|}.
    \end{align*}
\end{lemma}

\begin{proof}
As noted in Proposition 3 of \cite{Masoud-thesis}, this is a straightforward consequence of basic theory as developed in \cite{JaffeTaubes, Stuart, Palvelev2},  etc.
\end{proof}

\subsection{The Higgs mechanism}

Our main results rely heavily on the analysis of a linearization of the \eqref{dAHM}, and the linearization of the 2d elliptic Abelian Higgs model plays a crucial role in this analysis. 

We will write
the $2d$ equations as $S^{2d}[u] = 0$, where 
\begin{equation}\label{2d.AHM}
S^{2d}[u] := 
\left(\begin{array}{c}
S^{2d}_\phi[u] \\
S^{2d}_b [u]
\end{array}\right)
=
\left(\begin{array}{c}
-D_aD_a \Phi
+\frac \lambda 2(|\Phi|^2-1)\Phi \\
-\pp_a F_{ab} - (i\Phi, D_b\Phi)
\end{array}\right)_{b=1,2}
\end{equation}
For $u=(\Phi,A_a)$ and $\tilde u = (\tPhi, \tA_a)$, we define the linearized operator as
\beq\label{calL.def0}
\calL[u](\tilde u)
:= \frac d{ds}\Big|_{s=0} S^{2d}[u+s\tilde u] 
\eeq
A calculation shows that
\beq\label{calL.def}
\begin{aligned}
\calL_\phi[u](\tilde u)
&= 
-D_{A_a} D_{A_a}  \tPhi  + 2 i\tA_a D_{A_a}\Phi
+ i \Phi \pp_a\tA_a
+ \frac 12 (|\Phi|^2-1)\tPhi + ( \Phi, \tPhi) \Phi
\\
\calL_a[u](\tilde u)
&=
-\pp_b \tF_{ba}
-( i\Phi, D_{A_a} \tPhi) 
-( i\tPhi, D_{A_a} \Phi) 
+ \tA_a |\Phi|^2, 
\end{aligned}
\eeq
where we recall that repeated indices $a,b,...$ are implicitly summed from $1$ to $2$, and
\[
D_{A_a} = \pp_a - i A_a, \qquad \tilde F_{ab} = \pp_a\tA_b - \pp_b \tA_a.
\]
Awkwardly, $\calL[\modu(q)]$ has a huge nullspace for $q\in M_N$, or for any gauge-equivalent $u$.
Indeed, 
because 
$S^{2d}[e^{i s \eta}\mphi(q), \mA(q)+s d\eta] = 0$ for all $s$, and noting that
\[(e^{i s \eta}\mphi(q), \mA(q)+s d\eta) = \modu(q) + s (i\eta\mphi(q), d\eta) + O(s^2),
\]
we find that
\beq\label{bigkernel}
\calL[\modu(q)](i\eta\mphi(q), d \eta) = 0 \qquad\mbox{ for all smooth enough }\eta.
\eeq
In other words, the kernel of $\calL[\modu(q)]$ contains all infinitessimal gauge transformations at $\modu(q)$. To remedy this, following Stuart \cite{Stuart}, we will modify $\calL$ to obtain an operator with a finite-dimemsional nullspace. We define
\beq\label{Lu.def}
L[u](\tilde u) := \calL[u](\tilde u) +
\left( \begin{array}{c}
    i\Phi (\nabla\cdot \tilde A - (i\Phi, \tPhi))\\
    \pp_a (\nabla\cdot \tilde A - (i\Phi, \tPhi))
\end{array}
\right)_{a=1,2}.
\eeq
One readily checks that
\begin{equation}
\label{eq33.14}
{L}[u](\tu)=
\begin{pmatrix}
-D_{A_a} D_{A_a}  \tPhi  + 2 i\tA_a D_{A_a}\Phi
+\frac{1}{2}\big(3|\Phi|^2-1\big)\tPhi
\\
-\Delta \tA_b+|\Phi|^2\tA_b-2\big(i\tPhi,D_{A_b}\Phi\big)
\end{pmatrix}_{_{b=1,2}}
\end{equation}
It is immediate from the definitions that
\begin{equation}\label{L=calL}
\calL[\modu(q)](\tilde u) = L[u](\tilde u) \qquad\mbox{ if $\tilde u$ is gauge-orthogonal, see \eqref{eq:go}.}
\end{equation}

The following theorem plays a key role in our analysis.

\begin{theorem}[\cite{Stuart}, Theorems 2.6 and 3.1]\label{thm:hessian}
    For $q\in M_N$, the nullspace of $L[\modu(q)]$ is $2N$-dimensional and is spanned by 
    $\{ \tilde n_\mu\}_{\mu=1}^{2N}$. Moreover, for every compact $K\subset M_N$ there exists $\gamma  =\gamma(K)>0$ such that if  $q\in K$ and $\tilde u\in (H^1)^4$ satisfies $(\tilde u, \tilde n_\mu(q))_{L^2} = 0$ for all $\mu=1,\ldots, 2N$, then
    \beq\label{coercivity}
    \int_{\R^2} \tilde u\cdot L[\modu(q)]\tilde u \, dx \ \ge \ \gamma \| \tilde u\|_{H^1}^2.
    \eeq
\end{theorem}

Stuart \cite{Stuart} establishes \eqref{coercivity} with the norm
\[
\|\tilde u\|_{H^1_q}^2 := \|\tilde A\|_{H^1}^2 + \int_{\R^2} |\tPhi|^2 + |(\nabla - i \mA(q))\tPhi|^2\, dx
\]
on the right-hand side, but using the fact that $\sup_{q\in K} \| \mA_a(\cdot;q)\|_{L^\infty(\R^2)}<\infty$ for $K$ compact, this is easily seen to imply the estimate as we have stated it.

We remark that, since $\tilde n_\mu$ is gauge-orthogonal, the theorem and \eqref{L=calL} immediately imply that
\begin{equation}\label{kercalL}
\calL[\modu(q)] \tilde n_\mu = 0 \qquad\mbox{ everywhere}.
\end{equation}

\section{construction of approximate solutions}\label{sec:construction}

In this section we construct an approximate solution of the Abelian Higgs model. 

We start with a solution of $(y_{\bar a}, y_0)\in \R^{d-2}\times (0,T)\mapsto q(y_{\bar a}, y_0)\in M_N$ of the geometric evolution equation \eqref{dwm} satisfying assumptions \eqref{q.compact0} and \eqref{q.decay}
and that $\Theta$ as defined in \eqref{Theta.def0} is finite.

In coordinates, equation \eqref{dwm} takes the form
\[
(\kz \pp_{0}\pp_0+\ko \pp_0 - \Delta_{\bar a}) q^\mu = - \Gamma^\mu_{\nu\lambda}(\kz \pp_0q^\nu \pp_0 q^\lambda - \pp_{\bar a}q^\nu \pp_{\bar a}q^\lambda),
\]
where $\Gamma^\mu_{\nu\lambda}$ are the Christoffel symbols. It follows 
that $\pp_0^2q$ has the same regularity as $\Delta_{\bar a} q$ if $\kz>0$:
\beq
\kz \| \pp_0^2 q\|_{L^\infty_T H^{L-1}(\R^{d-2})} \le C.
\label{p00q}\eeq
We will write 
\[
q_\ep(x_{\bar a},t) = q( \ep x_{\bar a}, \ep t) = q(y_{\bar a},y_0), \qquad ( y_a, y_{\bar a}, y_0) := ( x_a, \ep x_{\bar a}, \ep t).
\]
The leading-order term in our approximate solution will have the form
\begin{equation}\label{U0.form}
U^\ep_0(x,t) = 
\begin{pmatrix}
u^\ep\\
A^\ep_{\bar a}\\
A^\ep_0
\end{pmatrix}
=
\begin{pmatrix}
 \modu(x_a ; q_\ep(x_{\bar a}, t))\\
\chi_{\pp_{\bar a} q_\ep(x_{\bar a}, t)}(x_a); q_\ep(x_{\bar a},t)) \\
\chi_{\pp_t q_\ep(x_{\bar a}, t)}(x_a; q_\ep(x_{\bar a},t)) \\
\end{pmatrix}
\end{equation}
using notation introduced in Theorems \ref{thm1.new} and \ref{thm:MNmetric}, see also \eqref{chiv.def}. Since $|\pp_t q_\ep|, |\partial_{\bar a}q_\ep| \le C \ep$, it follows that
\[
|A^\ep_{\bar a}|, |A^\ep_0| \le C \ep.
\]
From the definitions, see again \eqref{chiv.def}  and Theorem \ref{thm:MNmetric}, one finds that $A^\ep_0$ and $A^\ep_{\bar a}$ are chosen exactly to guarantee that
\beq\label{Aabar1}
\binom{D_\alpha \Phi^\ep}{\pp_\alpha A^\ep_a - \pp_a A^\ep_\alpha} = 
\sum_\mu \pp_\alpha q_\ep^\mu \,  \tilde n_\mu(q_\ep)= \tilde n_{\pp_\alpha q_\ep}
\qquad \qquad\alpha = 0, 3,\ldots, d .
\eeq
In particular, for every $(x_{\bar a}, t)$,
\begin{equation}\label{Aabar.choice}
\binom{D_\alpha \Phi^\ep}{\pp_\alpha A^\ep_a - \pp_a A^\ep_\alpha} (\cdot , x_{\bar a}, t) \ 
\mbox{ satisfies gauge-orthogonality condition \eqref{eq:go}}.
\end{equation}

We will construct an approximate solution by a formal expansion in powers of $\ep$.  In the statement below, $U^\ep_m$ collects all the terms up to order $\ep^{2m}$ in the expansion,  apart from the leading term. Thus the size of $U^\ep_m$ reflects the size of the first correction term.   The full ansatz, which involves some notation to be introduced below, is displayed in \eqref{expand.Uke}, and estimates for individual terms in the expansion appear in \eqref{cmmu.est}, \eqref{Um.est}.

In the result below, for a multiindex $r= (r_0,\ldots, r_d)$  with each $r_j\in \Z_{\ge 0}$, we use the notation
\[
\pp^r = \pp_{t}^{r_0}\pp_{x_1}^{r_1} \cdots\pp_{x_d}^{r_d}, 
\qquad\qquad 
|r_{\bar a}| = r_3+\cdots +r_d.
\]
Our estimates will control one time derivative if $\kz>0$ and no time derivatives if $\kz=0$, so
we always assume
\beq\label{rconstraint}
r_0\le 1,\qquad\mbox{ and }r_0=0\ \ \mbox{ if }\kz=0.
\eeq

\begin{proposition}\label{prop:approx.sol}
Let $j=\lfloor d/2\rfloor$, and fix positive integers
$L,m$ such that
\beq\label{parameters}
k_m:= L-4m >\frac d2 \ge j > \frac d2-1.
\eeq
Let $q: \R^{d-2}\times (0,T)\to M_N$ be a solution of \eqref{dwm} satisfying \eqref{q.compact0}, \eqref{q.decay} for this choice of $L$, with $\Theta<\infty$.

There exists $T_0\in (0,T]$ and $\ep_0>0$ such that for every positive $\ep<\ep_0$, there exists a map $U_m^\ep : (0,T_0/\ep)\times \R^d\to \C \times \R^{d+1}$ with the following properties:

\medskip

{\bf 1.} The map $U_{m,\ep}$ defined by 
\[
U_{m,\ep} = U^\ep_0 + \ep^2 U_m^\ep
\]
is an approximate solution of \eqref{dAHM} in the sense that for any $\gamma\in (0,1/2)$ and any multi-index $|r|\le k_m-j$ satisfying \eqref{rconstraint}, there exist a positive constant\footnote{All constants in the proposition, and throughout this section, can in principle be estimated in terms of  $\kz,\ko$, the set $K$ from \eqref{q.compact0}, $\Theta$ from \eqref{q.decay} and $\gamma\in (0,1)$, where the latter is only needed for estimates that contain a weight $e^{-\gamma|x_a|}$, such as \eqref{tilde.size}.} $C$ such that (using notation introduced in \eqref{S.def}, \eqref{S.components})
\begin{equation}\label{S.est}
\begin{aligned}
|\pp^r S_u[U_{m,\ep}]|(x, t) &\le C \ep^{2(m+1) + r_0 + |r_{\bar a}|} e^{- \gamma |x_a|}\\
|\pp^r S_\alpha[U_{m,\ep}]|(x, t)
&\le C \ep^{2m+3+ r_0 + |r_{\bar a}|} e^{- \gamma |x_a|}, \qquad\quad
\alpha = 0, 3,\ldots, d
\end{aligned}
\end{equation}
for all $(x,t)\in (0,T_0/\ep)\times \R^d$. Moreover,
\beq\label{H.Hkest}
\| S[U_{m,\ep}] (\cdot, t)\|_{H^{k_m}(\R^d)}
+  \kz\| \pp_t S[U_{m,\ep}] (\cdot, t)\|_{H^{k_m-1}(\R^d)} 
\le C\ep^{2m+2 - \frac d2} 
\eeq
for all $t\in (0,T_0/\ep)$.
\medskip

{\bf 2} There exists a constant $C$ such that for any multi-index $|r| \le k_m+2-j$ satisfying \eqref{rconstraint},
\begin{equation}\label{tilde.size}
\begin{aligned}
|\pp^r u_m^\ep|(x,t) &\le C \ep^{r_0+|r_{\bar a}|} e^{- \gamma |x_a|}\\
 |\pp^r A_{m,\alpha}^\ep|(x,t) 
 &\le C\ep^{1+r_0+|r_{\bar a}|} e^{- \gamma |x_a|}\quad\mbox{ for }\alpha \in \{0,3,\ldots, d\} 
\end{aligned}
\end{equation}
for all $(x,t)\in (0,T_0/\ep)\times \R^d$, where 
\[
U^m_\ep = \begin{pmatrix}u_m^\ep\\ A_{m, {\bar a}}^\ep \\ A_{m,0}^\epsilon
\end{pmatrix}
\quad\mbox{ with }
u^\ep_m = \begin{pmatrix}
 \Phi^\ep_m \\ A^\ep_{m,a}     
\end{pmatrix}   .
\]
\end{proposition}

\begin{remark}
Our full results, which are sharper and more detailed than the conclusions stated above,
are stated (after the change of variables described in
Section \ref{sec:3.1} below) in \eqref{expand.Uke} -- \eqref{wSmj.est}.
For example, they imply that if $r$ is any multi-index such that $|r|\le k_m$ and satisfying \eqref{rconstraint}, then
\beq\label{strongerS}
\begin{aligned}
\| \pp^r S_u[U_{m,\ep}] (\cdot, t)\|_{L^2(\R^d)} &\le  C\ep^{2m+2 + r_0 + |r_{\bar a}| - \frac d2} ,\\
\| \pp^r S_\alpha[U_{m,\ep}] (\cdot, t)\|_{L^2(\R^d)} &\le  C\ep^{2m+3 + r_0 + |r_{\bar a}| - \frac d2} , \qquad \alpha = 0, 3, \ldots, d
\end{aligned}
\eeq
which is stronger than  \eqref{H.Hkest}.

Note that as a result of \eqref{tilde.size} and the definition \eqref{U0.form} of the leading-order term $U^\ep_0$, for $|r|\le k_m+2-j$ such that $r_0\le 1$ the approximate solution
$U_{m,\ep}$ satisfies
\begin{equation}\label{Uep.size}
\begin{aligned}
\|\pp^r u_{m,\ep} \|_{L^\infty((0,T_0)\times \R^d)} &\le C \ep^{r_0+|r_{\bar a}|}\\
\|\pp^r A_{m,\ep,\alpha}\|_{L^\infty((0,T_0)\times \R^d)} 
&\le C\ep^{1+r_0+|r_{\bar a}|} \quad\mbox{ for }\alpha \in \{0,3,\ldots, d\}.
\end{aligned}
\end{equation}
\end{remark}

\medskip

The rest of this section is devoted to the proof of the proposition.

\medskip

\subsection{Change of variables}\label{sec:3.1}
In our construction of an approximate solution,  we will work with the variables $y = (y_a, y_{\bar a}, y_0)$ defined by
\begin{equation}\label{sy.vs.tx}
(y_a, y_{\bar a}, y_0) = (x_a, \ep x_{\bar a}, \ep t).
\end{equation}
We will use the notation
\[
\wpp_0 = \frac\pp {\pp y_0} = \frac 1\ep \frac\pp {\pp t} = \frac 1 \ep \pp_0, \qquad\quad
\wpp_{\bar a} = \frac \pp {\pp y_{\bar a} }= \frac 1\ep \frac \pp {\pp x_{\bar a}} = \frac 1 \ep \pp_{\bar a}, \qquad\quad \wpp_a = \pp_a.
\]
We will use the fonts\footnote{If these fonts look uncomfortably similar to $A,D,F$, one should just keep in mind that all computations in all proofs in this section are carried out with respect to the $y$ variables.}
 $\wA, \wD, \wF$ to denote the connection $1$-form, covariant derivative, and curvature $2$-form expressed with respect to these new variables. 
Thus
\[
A_0 dt +\sum_{j=1}^d A_j dx^j = 
\wA_0 dy_0 +\sum_{j=1}^d \wA_j dy^j 
\]
which implies that
\begin{align*}
 \wA_\alpha(y_a, y_{\bar a}, y_0) &= \frac 1 \ep A_\alpha(x_a, \ep x_{\bar a},\ep t) \quad\mbox{ for }\alpha=0,3,\ldots, d, \\
\wA_a(y_a, y_{\bar a}, y_0) &= A_a(x_a, \ep x_{\bar a}, \ep t), \quad \mbox{ for }a=1,2.
\end{align*}
Henceforth omitting the arguments when no confusion can result, we similarly have
\[
\wD_{a} = D_a, \qquad \qquad\qquad\wD_{\alpha} = \wpp_{\alpha} - i \wA_{\alpha} = \frac 1 \ep D_{\alpha} \ \ \mbox{ for } \ \ \alpha=0,3,\ldots, d
\]
and $\wF_{\alpha\beta} = \wpp_{\alpha} \wA_{\beta} - \wpp_{\beta} \wA_{\alpha} $, so that
\[
\wF_{\alpha\beta}
= \begin{cases} 
F_{\alpha\beta}&\mbox{ if }\alpha, \beta \in \{1,2\} \\
\ep^{-1}F_{\alpha\beta}& \mbox{ if }\alpha \in \{1,2\}\mbox{ and }\beta \not\in \{1,2\}\mbox{ or vice versa}\\
\ep^{-2}F_{\alpha\beta}&\mbox{ if }\alpha, \beta \not\in \{1,2\} 
\end{cases}
\]
We will also write $\wU$ rather than $U$ to denote a map $\binom \Phi \wA$ where the $1$-form coefficients are written with respect to the $y = (y_a, y_{\bar a}, y_0)$ variables.
We define
\beq\label{wSe.def}
\wSe_u[\wU](y) := S_u[U](x,t), \qquad\qquad
\wSe_\alpha[\wU](y) := \frac 1\ep  S_\alpha[U](x,t), \ \ \ 
\alpha=0,3,\ldots,d
\eeq
where $\wU$ and $U$ related as above, with $\wSe_u = \binom{\wSe_\phi}{\wSe_a}$ as usual.
In the new variables, we will write our ansatz as
\begin{equation}\label{r.ansatz}
\wU_{m,\ep} = \wU^0 +\ep^2 \wU_m^\ep \qquad\mbox{ where }
\wU^0(y) := 
\begin{pmatrix}
\wu^0\\
\wA^0_{\bar a}\\
\wA^0_0
\end{pmatrix}(y)
:=
\begin{pmatrix}
 \modu(y_a ; q(y_{\bar a}, y_0))\\
\chi_{\pp_{\bar a} q(y_{\bar a}, y_0)}(y_a) \\
\chi_{\pp_t q(y_{\bar a}, y_0)}(y_a) \\
\end{pmatrix}.
\end{equation}
To establish the estimates in Proposition \ref{prop:approx.sol}, it now suffices to prove, using notation introduced in \eqref{Xkj}, \eqref{weightednorm},  that
\begin{equation}\label{wS.est}
\| \wSe[\wU_{m,\ep}] \|_{L^\infty_{T_0} X^{k_m,j}_\gamma} + \kz\|\wpp_0 \wSe[\wU_{m,\ep}] \|_{L^\infty_{T_0} X^{k_m-1,j}_\gamma} \le C \ep^{2(m+1)}
\end{equation}
and
\begin{equation}\label{wUep.size1}
\| \wU^\ep_m \|_{L^\infty_{T_0} X^{k_m+2,j}_\gamma} + \kz\|\wpp_0\wU^\ep_m  \|_{L^\infty_{T_0} X^{k_m+1,j}_\gamma} \le C. 
\end{equation}
Indeed, \eqref{wS.est} contains estimates of 
$\| \wSe[\wU_{m,\ep}]\|_{L^\infty_{T_0}H^{k_m}}$, 
and \eqref{wUep.size1} contains estimates of 
$\| \wU^\ep_m\|_{L^\infty_{T_0}W^{k_m+2-j}_\gamma}$
that are equivalent to 
\eqref{strongerS} and \eqref{tilde.size} respectively.

\medskip

\subsection{Some preliminary computations} \label{sec:3.2}
It is straightforward to check that
\begin{equation}\label{eq:wSe1}
\wSe[\wU] = \wS^0[\wU] + \ep^2 \wS^1[\wU]
\end{equation}
for
\begin{equation}\label{eq:wSe2}
\wS^0[\wU] =
\begin{pmatrix}
\wS^0_\phi[\wU]\\
\wS^0_a[\wU]\\
\wS^0_{\bar a}[\wU]\\
\wS^0_0[\wU]
\end{pmatrix}
=
\begin{pmatrix}
 - \wD_a\wD_a \Phi + \frac 12 (|\Phi|^2-1)\Phi \\
 - \wpp_{b} \wF_{ba} - (i\Phi, \wD_a\Phi)  \\
  - \wpp_{a}  \wF_{a \bar a} - (i\Phi, \wD_{\bar a}\Phi) \\
  -  \wpp_{a}  \wF_{a 0} - (i\Phi, \wD_0\Phi) 
\end{pmatrix}
\end{equation}
and
\begin{equation}\label{eq:wSe3}
\wS^1[\wU] =
\begin{pmatrix}
(\kz \wD_0\wD_0 + \ko \wD_0 - \wD_{\bar a}\wD_{\bar a})\Phi \\
(\kz\wpp_0+\ko)\wF_{0a} - \wpp_{\bar a} \wF_{\bar a a} \\
(\kz\wpp_{0}+\ko) \wF_{0  \bar a}  - \wpp_{\bar b}  \wF_{\bar b \bar a} \\
-\wpp_{\bar b}  \wF_{\bar b 0 }
\end{pmatrix}.
\end{equation}
More than once we will encounter expressions of the form $\wS^\ep[\wU + \ep^p\wtU]$, which we will expand in powers of $\ep$.
To this end we introduce the notation
\[
\wZ = (\wU, \kz\wpp_0 \wU, \nabla \wU), \qquad 
\wtZ = (\wtU, \kz \wpp_0\wtU, \nabla \wtU), 
\]
and we note that all nonlinear terms in $S^0[ \wU ]$ and $S^1[\wU]$ are polynomials in $\wZ$ of degree at most 3.
For $m=0,1$ we can therefore expand
\begin{equation}\label{eq:wSe4}
S^m[\wU+\ep^p\wtU] = S^m[\wU] + \ep^p \ell^m[\wU](\wtU) + \ep^{2p}p^m_2(\wZ,\wtZ) + \ep^{3p} p^m_3(\wtZ)
\end{equation}
where $\ell^m[\wU]$ is a linear operator on $\wtU$ with coefficients depending on $\wU$, and $p^m_j$ is a polynomial in its arguments, at most cubic, and of order $j$ in $\wtZ$, all uniquely determined by \eqref{eq:wSe4} and the definitions of $\wS^m$, $m=0,1$.  Noting that $\wS^0_u[\wU] = S^{2d}[\wu]$ as defined in \eqref{2d.AHM},
we find that
\begin{equation}\label{eq:wSe5}
\begin{aligned}
\ell^0[\wU](\wtU)
&= 
\begin{pmatrix}
\ell^0_u[\wU]\\
\ell^0_{\bar a}[\wU]\\
\ell^0_0[\wU]\\
\end{pmatrix}(\wtU)
&=
\begin{pmatrix}
\calL[\wu](\wtu)\\
-\wpp_b \wtF_{b \bar a} + |\Phi|^2 \wtA_{\bar a} - (i\tPhi, D_{\wA_{\bar a}}\Phi) - (i\Phi, D_{\wA_{\bar a}}\tPhi) \\
-\wpp_b \wtF_{b 0} + |\Phi|^2 \wtA _{0}- (i\tPhi, D_{\wA_{0}}\Phi) - (i\Phi, D_{\wA_{0}}\tPhi)
\end{pmatrix}
\end{aligned}
\end{equation}
where  $\ell^m_u = \binom {\ell^m_\phi}{\ell^m_a}$ and
\[
\wtF_{\alpha\beta} = \wpp_\alpha\wtA_\beta - \wpp_\beta\wtA_\alpha . 
\]
Next, observing that $\wS^1_\alpha$ is linear for all $\alpha = 0,\ldots, d$, we  compute
\begin{equation}\label{eq:wSe6}
\ell^1[\wU](\wtU)
= 
\begin{pmatrix}
\ell^1_\phi[\wU](\wtU)\\
(\kz\wpp_0+\ko)\wtF_{0a} - \wpp_{\bar a} \wtF_{\bar a a} \\
(\kz\wpp_{0}+\ko) \wtF_{0  \bar a}  - \wpp_{\bar b}  \wtF_{\bar b \bar a} \\
-\wpp_{\bar b}  \wtF_{\bar b 0 }
\end{pmatrix}
\end{equation}
where
\begin{equation}\label{eq:wSe7}
\begin{aligned}
\ell^1_\phi[\wU](\wtU) &= (\kz \wD_{A_0}\wD_{A_0} + \ko \wD_{A_0} - \wD_{A_{\bar a}}\wD_{A_{\bar a}})\tPhi 
\\
&\hspace{1em}\ \  
-
 (\kz\wpp_0 \wtA_0 + \ko \wtA_0 -  \wpp_{\bar a}\wtA_{\bar a} )i\Phi -2i(\kz\wtA_0 \wD_0\Phi  - \wtA_{\bar a}\wD_{\bar a}\Phi) .
\end{aligned}
\end{equation}

The leading-order term in the ansatz is designed exactly to satisfy the following lemma.

\begin{lemma}\label{lem:kerS0}
$\wS^0[\wU^0] = 0$.
\end{lemma}

\begin{proof}
The terms $\wS^0_u = \binom{\wS^0_\phi}{\wS^0_a}$ act only on $\wu = \binom{\Phi}{\wA_a}$ and coincide with
$S^{2d}$ as defined in \eqref{2d.AHM} on every 2d slice $\R^2\times \{ (y_{\bar a,a)}\}$. 
Since $\wu^0(\cdot, y_{\bar a}, y_0) = \modu(\cdot; q(y_{\bar a}, y_0))$ belongs to the moduli space for every $(y_{\bar a}, y_0)$,
it follows that $\wS^0_u[\wU^0]= 0$.

Next, recall from \eqref{Aabar.choice}, rescaled to the $y$ variables, that $\wA^0_\alpha$ for $\alpha=0,3,\ldots, d$ is chosen to guarantee that
\[
\wpp_\alpha \wu^0 - \binom{i \wA^0_\alpha\Phi^0}{d\wA^0_\alpha}  \ \ 
\mbox{ satisfies  gauge-orthogonality condition \eqref{eq:go}}
\]
with respect to the $y_a$ variables for every $(y_{\bar a}, y_0)$.
This states exactly that
\[
-\wpp_\alpha ( \wpp_a \wA^0_{\alpha} - \wpp_{\alpha} \wA^0_a) + (i\Phi^0, \wD_{\alpha}\Phi^0) = 0
\]
which is equivalent to the assertion that $\wS^0_{\alpha}[U^0] = 0$ for $\alpha = 0, 3, \ldots d$.
\end{proof}

The next lemma will be used at a later stage of the construction of approximate solutions.

\begin{lemma}\label{lem:kerell0}
For $\mu = 1,\ldots, 2N$, let
\[
N_\mu(y)
 := \begin{pmatrix}
N_{\mu,u}\\
N_{\mu, \bar a}\\
N_{\mu, 0}
\end{pmatrix}(y)
 := \begin{pmatrix}
n_{\mu}(y_a; q(y_{\bar a}, y_0)) \\
\xi_{\mu, \bar a}\\
\xi_{\mu, 0}
\end{pmatrix}
\]
where  $\xi_{\mu,\alpha}$ solves, for every $(y_{\bar a}, y_0)\in \R^{d-2}\times (0,T)$ and  $\alpha=0,3,\ldots, d$, 
\begin{equation}\label{ximualpha}
(-\Delta_a + |\mphi(q) |^2)\xi_{\mu,\alpha} = 2(i n_{\mu,\phi}(\cdot, q) , \wD_{\wA^0_\alpha} \mphi(\cdot; q) ) \qquad\mbox{ in }\R^2.
\end{equation}
Then for any smooth functions $c_\mu:\R^{d-2}\times (0,T)\to \R$ for $\mu=1,\ldots, 2N$,
\[
\ell^0[\wU_0]( c_\mu N_\mu) = 0 \qquad 
\]
where as usual the repeated $\mu$ index is summed implicitly $\mu=1,\ldots, 2N$.
\end{lemma}

\begin{proof}
Because $\calL$ involves only derivatives with respect to the $y_a$ variables, 
\[
\ell^0_u[U_0](c_\mu N_\mu) 
\overset{\eqref{eq:wSe5} }= \calL[\modu(q)]( c_\mu n_\mu) 
 =  c_\mu \calL[\modu(q)]( n_\mu) \overset{ \eqref{kercalL}}= 0.
\]
Next, for any $\alpha\in 0,3,\ldots, d$, we compute
\[
\begin{aligned}
\ell^0_\alpha[U_0](c_\mu N_\mu)
&= 
(-\Delta_a+|\mphi(q)|^2) (c_\mu \xi_{\mu,\alpha}) + \wpp_{\alpha} \wpp_a ( c_\mu n_{\mu,a})\\
&\hspace{7em}- ( i c_\mu n_{\mu,\phi}, \wD_{\wA_\alpha}\mphi(q)) 
- ( i \mphi(q),  \wD_{\wA_\alpha}(c_\mu n_{\mu,\phi}) .
\end{aligned}
\]
Since $n_\mu(\cdot, q)$ is gauge-orthogonal at $\modu(q)$, 
\[
\begin{aligned}
\wpp_{\alpha} \wpp_a ( c_\mu n_{\mu,a})
&= 
\wpp_{\alpha} [ c_\mu \wpp_a ( n_{\mu,a})]
=
\wpp_\alpha[ c_\mu  (  i\mphi(q), n_{\mu,\phi})] \\
&=
\wpp_\alpha (  i\mphi(q), c_\mu n_{\mu,\phi})\\
&= 
 (  i\wD_\alpha \mphi(q), c_\mu n_{\mu,\phi})
+
 (  i \mphi(q), \wD_\alpha(c_\mu n_{\mu,\phi}))
\end{aligned}
\]
Substituting into the above and simplifying, we find that
\[
\ell^0_\alpha[U_0](c_\mu N_\mu)
= (-\Delta_a+|\mphi(q)|^2) (c_\mu \xi_{\mu,\alpha})  - 2 c_\mu (i n_{\mu,\phi}, \wD_\alpha\mphi(q)) \overset{\eqref{ximualpha}}=0.
\]
\end{proof}

The next two lemmas record some properties of the auxiliary functions introduced above.

\begin{lemma}\label{lem:U0Nmu}
Assume \eqref{q.compact0} and \eqref{q.decay}, and let $j>\frac d2-1$.
    Then 
    \beq\label{U0N1}
    \begin{aligned}
            \| \wU^0 \|_{L^\infty_{T/\ep}W^{L-j, \infty}(\R^d)} +  \kz \|\wpp_0 \wU^0 \|_{L^\infty_{T/\ep}W^{L-j, \infty}(\R^d)} &\le C(\Theta)   \\
            \| N_\mu\|_{L^\infty_{T/\ep}W^{L-j, \infty}(\R^d)} +  \kz  \| \wpp_0N_\mu\|_{L^\infty_{T/\ep}W^{L-j, \infty}(\R^d)} &\le C(\Theta)
    \end{aligned}    
    \eeq
    Moreover, if $v\in H^l(\R^d)$ for $l\le L$, then for any $t\in (0, T/\ep)$,  $\wU^0(t) v, N_\mu(t) v\in H^l(\R^d)$, and
    \beq\label{U0N2}
    \| \wU^0 v \|_{H^l(\R^d)}  +  \| N_\mu v \|_{H^l(\R^d)} \le C(\Theta)  \|  v \|_{H^l(\R^d)} 
    \eeq
    where $\wU^0 v$ or $N_\mu v$ may denote the product of $v$ with any component of $\wU^0(t)$ or $N_\mu(t)$.
    Similarly, for every $t\in (0,T/\ep)$,
    \beq\label{U0N3}
   \left. \begin{aligned}
        \| (\nabla  \wU^0) v \|_{H^l(\R^d)}  +  \kz \|(\wpp_0 \wU^0) v \|_{H^l(\R^d)} &\le C(\Theta)  \|  v \|_{H^l(\R^d)}\\
        \| (\nabla N_\mu) v \|_{H^l(\R^d)}  + \kz \| (\wpp_0 N_\mu) v \|_{H^l(\R^d)} &\le C(\Theta)  \|  v \|_{H^l(\R^d)}
    \end{aligned}  \right\}   \quad\mbox{ for }l \le L-1.
    \eeq
\end{lemma}

\begin{remark}\label{rmk:multiplier}
    if $X$ is a function space, we will say that a function $\wU$ is an {\em $X$-multiplier} if 
\[
\| \wU \psi \|_{X} \le C \| \psi\|_{X} \qquad\mbox{ for all }\psi\in X.
\]
For a vector-valued $\wU$, we understand the left-hand side to mean the product of any component of $\wU$ with $\psi$. In all that follows, the constant for an $X$-multiplier will depend only on $\Theta$ from \eqref{Theta.def0}.
It is easy to see that every $W^{k,\infty}$ function is a $W^{k,\infty}_\gamma$ multiplier, so the lemma implies that for every $t\in (0,T/\ep)$,
\[
\left.
\begin{aligned}
    &(\wU^0, \kz\wpp_0\wU^0, \nabla\wU^0)(\cdot,t) \\
    &(N_\mu, \kz\wpp_0 N_\mu, \nabla N_\mu)(\cdot,t)
    \end{aligned}
    \ \  \right\} \quad \mbox{ are $X^{l,j}_{\gamma}(\R^d)$ -multipliers for $l\le L-1$.}
\]
Thus
\[
\left.
\begin{aligned}
    &(\wU^0, \kz\wpp_0\wU^0, \nabla\wU^0) \\
    &(N_\mu, \kz\wpp_0 N_\mu, \nabla N_\mu)
    \end{aligned}
    \ \  \right\} \quad 
    \mbox{ are $L^\infty_{T/\ep} X^{l,j}_{\gamma}(\R^d)$ -multipliers for $l\le L-1$.}
\]
\end{remark}
\begin{proof}
Note that the $\C\times \R^2$ components of $\wU^0$ and $N_\mu$, {\em i.e.}
\[
\wU^0_u(y)
= \binom{\mphi}{\mA_a}(y_a; q(y_{\bar a}, y_0))
\qquad
N_{\mu,u}(y) = n_\mu(y_a; q(y_{\bar a}, y_0))
\]
all have the form
\beq\label{ac0}
f(q)(y) := f(y_a; q(y_{\bar a}, y_0))
\eeq
for a function $f:\R^2\times M_N\to \C$ or $\R$ such that 
\[
|\pp^r f(y_a,q)| \le C (1+|y_a|)^{-|r|}
\]
where $r$ is a multiindex denoting differentiation with respect to both $y_a\in \R^2$ and $q\in \R^{2N}\cong M_N$.
In Lemma \ref{lem:fq} in Appendix \ref{App:more}, we establish that such functions satisfy estimates \eqref{U0N1}, \eqref{U0N2}, subject to hypotheses \eqref{q.compact0}, \eqref{q.decay}.

The other components of $\wU^0$ and $N_\mu$ have the form
\beq\label{alpha.compoonents}
\pp_\alpha q(y_{\bar a}) f(q)(y), \qquad\alpha = 0,3 ,\ldots, d.
\eeq
for $f:\R^2\times M_N\to \R$ such that $|\pp^r f(y_a, q)| \le C(1+|y_a|)^{-1-|r|}$, when $q\in K$ compact $\subset M_N$.
Indeed, it follows from \eqref{r.ansatz} and \eqref{chiv.def} that 
\[
\wU^0_\alpha (y)= \pp_\alpha q^\lambda(y_{\bar a}) \chi_\lambda(y_a; q(y_{\bar a}, y_0)), \qquad\alpha = 0, 3,\ldots, d
\]
and every $\chi_\lambda$ has the required decay, see Lemma \ref{lem:decayq}.

Similarly, the $\alpha$ components of $N_\mu$ are defined 
as the solutions $\xi_{\mu,\alpha}$ of equation \eqref{ximualpha}. Recall that, due to the choice \eqref{r.ansatz} of $\wA^0_\alpha$, we can rewrite $\wD_{\wA^0_\alpha}\mphi(y)$, which appears on the right-hand side of \eqref{ximualpha},  as $\pp_\alpha q^\nu(y_{\bar a}) \, \tilde n_{\nu,\phi}(y_a; q(y_{\bar a}, y_0))$. From this one deduces that $\xi_{\mu,\alpha}$ can be written as 
\[
\xi_{\mu,\alpha}(y)  = \pp_\alpha q^\nu(y_{\bar a}) \xi_{\mu\nu}(y_a; q(y_{\bar a}, y_0))
\]
where $\xi_{\mu\nu}(\cdot, q)$ solves, for every $q\in M_N$,
\[
(-\Delta_a + |\mphi(q) |^2)\xi_{\mu\nu} = 2(i n_{\mu,\phi}(\cdot, q) , \tilde n_{\nu,\phi}(\cdot; q) ) \qquad\mbox{ in }\R^2 .
\]
Regularity and decay properties of $n_\mu$ imply that the right-hand side is smooth, with
\[
| \pp^r (i n_{\mu,\phi}, \tilde n_{\nu,\phi}) |( y_a, q) \ \le C(K,r) e^{-\gamma|y_a|}\qquad\mbox{ whenever }q\in K
\]
for $\gamma<1$, so it follows from Lemma \ref{scalar.lem1} that 
\[
\| \xi_{\mu\nu}(\cdot, q)\|_{W^{k,\infty}_\gamma(\R^2)} \le C(k,K) \qquad \mbox{ whenever }q\in K.
\]
It is proved in Lemma \ref{lem:fq} that functions of the form\footnote{Lemma \ref{lem:fq} proves \eqref{paqfq} only for $\alpha = \bar a = 3,\ldots, d$, but the
proof also shows that the same inequality holds for $\alpha =0$.}  shown in \eqref{alpha.compoonents} satisfy
\beq\label{paqfq}
\| \pp_\alpha q f(q) v \|_{H^l(\R^d)} \le C \| v \|_{H^l(\R^d)}, \qquad \mbox{ for } l \le L,
\eeq
completing the verification of \eqref{U0N2}. Estimate \eqref{U0N1} for the $\alpha$ components of $\wU^0$ and $N_\mu$
follows easily from conclusion \eqref{fq.Wkinfty} of Lemma \ref{lem:fq} and the embedding $H^L(\R^{d-2})\hookrightarrow W^{L-j,\infty}(\R^{d-2})$ for $j>\frac d2-1$.

The verification of \eqref{U0N3} is very similar. The main point is that every component of 
$\nabla\wU^0 , \wpp_0\wU^0, \nabla N_\mu, \wpp_0N_\mu$ either has the form \eqref{ac0} or\eqref{alpha.compoonents}, or else the form
\[
\wpp_\alpha q(y_{\bar a})\wpp_\beta q(y_{\bar a})f(q)(y) \qquad\mbox{ or }\quad 
\wpp_{\alpha \beta} q(y_{\bar a})f(q)(y).
\]
\end{proof}

\begin{lemma}\label{lem:S1U0}
Assume \eqref{q.compact0} and \eqref{q.decay}, and let $j>\frac d2-1$. Then $\wS^1[\wU^0]$ satisfies
\beq\label{S1U0.est}
    \| \wS^1[\wU^0] \|_{L^\infty_{T/\ep} X^{L-1,j}_{\gamma}(\R^d)} +   \kz  \| \wpp_0\wS^1[\wU^0] \|_{L^\infty_{T/\ep} X^{L-2,j}_{\gamma}(\R^d)}  \le C(\Theta)
    \eeq
    for every $\gamma\in (0,1)$.
\end{lemma}

\begin{proof}
We first consider $\wS^1_u[\wU^0]$. Recall that the choice of $\wA_\alpha$ for $\alpha = 0, 3,\ldots, d$ implies that
\[
\binom{\wD_\alpha\mphi}{\wF_{\alpha a}}(y)  = \pp_\alpha q^\mu (y_{\bar a}, y_0) \, \tilde n_{\mu}(y_a; q(y_{\bar a}, y_0)) = \pp_\alpha q^\mu \, n_\mu(q) (y)
\]
It follows that
\[
\begin{aligned}
\wS^1_u[\wU^0] &= (\kz\pp_0+\ko)(\pp_0 q^\mu \, \tilde n_\mu(q))
- \pp_{\bar a}(\pp_{\bar a}q^\mu \tilde n_\mu(q)) \\
&\hspace{10em}-i(\kz \wA_0 \, \pp_0 q^\mu -
i\wA_{\bar a} \pp_{\bar a} q^\mu) \binom{\tilde n_{\mu,\phi}(q)}0   . 
\end{aligned}
\]
Using \eqref{p00q} and smoothness and decay
properties of $n_\mu(\cdot, q)$, see Lemma \ref{lem:decayq}, it is straightforward to verify that the $u$ components of $\wS^1[\wU^0]$ satisfy \eqref{S1U0.est}, for example by arguing along the lines of the proof of Lemma \ref{lem:25}.

The other components have the form
\beq\label{wFalpha}
\wS^1_\alpha[\wU^0] = (\kz \pp_0+\ko) \wF_{0\alpha} -\pp_{\bar a} \wF_{\bar a \alpha}, \qquad\alpha = 0, 3,\ldots, d.
\eeq
So we need to understand regularity and decay properties of $\wF_{\alpha\beta}$ for $\alpha,\beta\in \{0,3,\ldots, d\}$. From the definitions, 
\beq\label{wF.formula}
\begin{aligned}
    \wF_{\alpha\beta} &= \wpp_\alpha (\pp_\beta q^\mu \chi_\mu(q)) - \wpp_\beta (\pp_\alpha q^\mu \chi_\mu(q)) \\
    &=
    \pp_\beta q^\mu \, \wpp_\alpha q^\nu\, \pp_\nu\chi_\mu(q) - \pp_\alpha q^\mu \, \wpp_\beta q^\nu\, \pp_\nu\chi_\mu(q) \\
    &=
    \pp_\beta q^\mu \, \wpp_\alpha q^\nu\, (\pp_\nu\chi_\mu(q) - \pp_\mu\chi_\nu(q)).
\end{aligned}
\eeq
We claim that
\beq\label{curl.chi}
(\pp_\nu\chi_\mu - \pp_\mu\chi_\nu)(q) \in W^{k,\infty}_\gamma(\R^2)\qquad\mbox{ for every $k\in \N$ and }
\gamma\in (0,1),
\eeq
with uniform estimates for $q$ in a compact subset $K\subset M_N$.
To see this, note from the definitions \eqref{chimu.def}  that  $(\pp_\nu\chi_\mu - \pp_\mu\chi_\nu)(\cdot, q)$ solves
\[
\begin{aligned}
    (-\Delta + |\mphi(q)|^2)(\pp_\nu\chi_\mu - \pp_\mu\chi_\nu)
&=
\pp_\mu|\mphi|^2 \chi_\nu - \pp_\nu|\mphi|^2 \chi_\mu+ \\
&\hspace{5em}\pp_\nu(i\mphi, \pp_\mu\mphi) - \pp_\mu(i\mphi, \pp_\nu \mphi)
\end{aligned}
\]
in $\R^2$. Lemma \ref{scalar.lem1} implies that \eqref{curl.chi} will follow if show that the right-hand side belongs to $W^{k,\infty}_\gamma(\R^2)$ for every $k$. This is  the case for 
$\pp_\mu|\mphi|^2 \chi_\nu - \pp_\nu|\mphi|^2 \chi_\mu$, in view of Lemma \ref{lem:decayq}. The other terms are clearly smooth, so we only need to check exponential decay of all derivatives. To see this, first note that
\[
\pp_\nu(i\mphi, \pp_\mu\mphi) - \pp_\mu(i\mphi, \pp_\nu \mphi)
=
(i\pp_\nu\mphi, \pp_\mu\mphi) - (i\pp_\mu\mphi, \pp_\nu \mphi)
= 2(i\pp_\nu\mphi, \pp_\mu\mphi).
\]
For $|y_a|$ large enough (depending on $q$)
we can write $\mphi(y_a;q) = \rho(y_a;q) e^{i\theta(y_a;q)}$ for $\rho:\R^2\to [0,\infty)$ and $\theta:\R^2\to \R/2\pi\Z$ smooth. One then checks that
\[
(i\pp_\nu\mphi, \pp_\mu\mphi) = \rho \pp_\mu \rho  \pp_\nu \theta - \rho \pp_\nu \rho \pp_\mu\theta
\]
It is easy to see from the form \eqref{eq:TaubesSol1} of $\mphi(q)$
that $\pp_\mu\theta(y_a)\in W^{k,\infty}$ and all its derivatives bounded (in fact decaying algebraically) as $|y_a|\to \infty$,  
so the desired exponential decay again follows from Lemma \ref{lem:decayq}. This completes the proof of
\eqref{curl.chi}.

It follows from \eqref{wF.formula} and \eqref{curl.chi}, along the same lines as the proof of Lemma \ref{lem:25}, that
\[
\| \wF_{\alpha\beta}\|_{ L^\infty H^L(\R^d)} +
\| \wF_{\alpha\beta}\|_{ L^\infty W^{L-j, \infty}_\gamma(\R^d)}  \le C(K,\Theta).
\]
Recalling \eqref{p00q}, this and \eqref{wFalpha} imply the conclusion of the lemma.
\end{proof}

\subsection{First step of the construction} \label{sec:3.3}

With these preliminaries, we now construct approximate solutions by an induction argument. 

We first assume that $k_1 = L-4m > \frac d2$, see
\eqref{parameters}, and 
we seek $\wU^1$ such that  $\wU_{1,\ep} = \wU^0+\ep^2\wU^1$  satisfies
\beq
\| \wS^1_\ep \|_{X^{L-4, j}_{\gamma}} 
+  \kz \|  \wpp_0\wS^1_\ep \|_{X^{L-5, j}_{\gamma}} 
\le C \ep^{4}
\qquad\qquad\mbox{ for }
{\wS^1_\ep} := \wSe[\wU_{1,\ep}]
\label{wS.est1}\eeq
and
\beq
\| \wU^1 \|_{X^{L-3, j}_{\gamma}} 
+ \|  \wpp_0\wU^1 \|_{X^{L-4, j}_{\gamma}} 
\le C
\label{wS.est2}
\eeq
where we recall that $X^{k,j}_{\gamma} = H^k\cap W^{k-j,\infty}_\gamma$, see \eqref{weightednorm} and \eqref{Xkj}, and all 
constants may depend on $K,\Theta$ and $\gamma$.

Using formulas in Section \ref{sec:3.2} above and Lemma \ref{lem:kerS0}, we compute
\begin{equation}\label{k=1.1}
\wS^\ep[\wU_{1,\ep}] =  \ep^2 \left( \ell^0[\wU^0](\wU^1) + \wS^1[\wU^0] \right) + \ep^4 P
\end{equation}
where $P$ is a polynomial in  $\ep, \wZ^0$ and $\wZ^1$, with $\wZ^j := (\wU^j,\kz \wpp_0 \wU^j, \nabla\wU^j)$. Indeed,
from \eqref{eq:wSe3} and \eqref{eq:wSe4} we have
\beq\label{P.formula}
P = p^0_2+\ep^2 p^0_3 + \ep^2 p^1_2 + \ep^4 p^1_3.
\eeq
with $p^m_j$ of order $j$ in $\wZ^1$.

We will choose $\wU^1$ so that the coefficient of $\ep^2$ vanishes and \eqref{wS.est2} holds. It will then be a straightforward matter to verify \eqref{wS.est1}.

To find $\wU^1$, we start by choosing the $\wu^1 = \binom{\Phi^1}{\wA^1_a}$ components so that 
\begin{equation}\label{basecase0}
\ell^0_u[\wU^0](\wU^1) + \wS^1_u[\wU^0]  = 0 . 
\end{equation}
We will solve for the $\wA^1_{\alpha}$ components for $\alpha=0,3,\ldots, d$ at a second stage.
In view of \eqref{eq:wSe5}, we can rewrite \eqref{basecase0} as
\begin{equation}\label{basecase1}
\calL[\wu^0](\wu^1) + \wS^1_{0,u} = 0\qquad\mbox{ for }\wS^1_0 := \wS^1[\wU^0].
\end{equation}

\begin{lemma}\label{lem:J1} $\wS^1_{0,u}(\cdot , y_{\bar a}, y_0)$ is orthogonal to all the zero modes $n_\mu(\cdot; q(y_{\bar a}, y_0))$ for every 
$(y_{\bar a}, y_0)$.
\end{lemma}

\begin{proof}
The lemma follows directly from Lemma \ref{lem:connection} and the assumption \eqref{dwm}
that $q$ solves \eqref{dAHM}. Indeed, Lemma \ref{lem:connection}
implies that for any $\alpha,\beta\in \{0,3,\ldots, d\}$, 
\begin{align*}
\left( \binom{\wD_\alpha\wD_\beta \Phi}{\pp_\alpha \wF_{\beta a}}, \tilde n_\mu \right)_{L^2(\R^2)}
&=
\left( \pp_\alpha q^\nu \pp_\beta q^\lambda  \binom{\calD_\nu \calD_\lambda \Phi}{\pp_\nu \calF_{\lambda a}}, \tilde n_\mu \right)_{L^2(\R^2)}
\\
&= (\pp_\alpha q^\nu \pp_\beta q^\lambda \nabla^g_{\pp_{\nu}}\pp_\lambda, \pp_\mu)_g
= 
(\nabla^g_{\alpha} \pp_\beta q , \pp_\mu)_g\end{align*}
for any $\mu$, where $\pp_\alpha q = \pp_\alpha q^\lambda \pp_\lambda$ and 
$\nabla^g_\alpha = \nabla^g_{\pp_\alpha q} = \pp_\alpha q^\lambda \nabla^g_{\pp_\lambda}$. 

Thus
\begin{align*}
( \wS^1_{0,u} , \tilde n_\mu)_{L^2(\R^2)}
&= 
\kz \left( \binom{\kz \wD_0\wD_0 \Phi}{\wpp_0 \wF_{0a}},  \tilde  n_\mu \right)
+ \ko \left( \binom{\kz \wD_0 \Phi}{ \wF_{0a}},  \tilde n_\mu \right)
-  \left( \binom{\wD_{\bar a}\wD_{\bar a} }{ \wpp_{\bar a} \wF_{\bar a a}},  \tilde n_\mu \right) \\
&=
(\kz \nabla^g_0 \pp_0q +\ko \pp_0 q - \nabla^g_{\bar a}\pp_{\bar a}q , \pp_{\mu})_g \overset{\eqref{dwm}}= 0.
\end{align*}
\end{proof}

\begin{lemma}\label{lem:J2}$\wS^1_{0,u}(\cdot, y_{\bar a}, y_0)$ satisfies the gauge-orthogonality condition  \eqref{eq:go} for every 
$(y_{\bar a}, y_0)$.
\end{lemma}

\begin{proof}
It follows from computations in Section \ref{sec:3.2} that
\[
\wS^\ep[\wU^0]  = \wS^0[\wU^0]+  \ep^2 \wS^1[\wU^0] =   \ep^2 \wS^1[\wU^0] .
\]
It then follows from the algebraic identity in Lemma \ref{lem:identity}, rescaled to the $y$ variables,
that
\[
( \wS^1_\phi[\wU^0], i\Phi) - \wpp_a \wS^1[\wU^0] =  - \ep^2\big((\kz\wpp_0 +\ko)\wS_0^1[\wU^0] - \wpp_{\bar a} \wS_{\bar a}[\wU^0]\big) .
\]
The lemma states that the left-hand side vanishes, which follows by taking the limit as $\ep\to 0$ of the above identity. Alternately, one can directly verify that the right-hand side vanishes for every $\ep>0$.
\end{proof}

In view of Lemma \ref{lem:J1} as well as Lemma \ref{vsolvability1} in Appendix \ref{appendix}, for every $(y_{\bar a}, y_0)$ we can find $\wu^1(\cdot, y_{\bar a}, y_0):\R^2\to \C\times \R^2$
\[
L[\modu(q(y_{\bar a}, y_0))](\wu^1(\cdot, y_{\bar a}, y_0)) + \wS^1_{0,u}(\cdot, y_{\bar a}, y_0)\qquad\mbox{ in }\R^2.
\] 
Lemmas \ref{vsolvability1} and \ref{lem:J2} imply that 
$\wu^1(\cdot, y_{\bar a}, y_0)$ satisfies the gauge orthogonality condition \eqref{eq:go}.
It follows that
$
\calL[\modu(q)](\wu^1)  = 
L[\modu(q)](\wu^1)
$
and hence that $\wu^1$ solves \eqref{basecase1}.

Now it follows from estimates of $\wS^1[\wU^0]$ in Lemma \ref{lem:S1U0} and the elliptic estimates in Proposition \ref{system.pde}
that
\beq\label{wu1.est}
\| \wu^1 \|_{L^\infty_T X^{L-1,j}_{\gamma}(\R^d) } + \kz \|\wpp_0 \wu^1 \|_{L^\infty_T X^{L-2,j}_{\gamma}(\R^d) }
\le C.
\eeq

To complete the construction of $\wU^1$ it now suffices to choose the remaining components such that
\[
\ell^0_\alpha[\wU^0](\wU^1) + \wS^1_{\alpha}[\wU^0]=0\qquad\qquad\mbox{ for }\alpha= 0,3,\ldots, d.
\]
This equation states that
\beq\label{wA1alpha.eqn}
(-\pp_a\pp_a +|\mphi(q)|^2)\wA^1_\alpha = - \pp_\alpha\pp_a \wA^1_a -
 (i\Phi^1, \wD_{\wA^0_\alpha}\mphi(q)) -  (i\mphi(q)), \wD_{\wA^0_\alpha}\Phi^1) -\wS^1_{\alpha}[\wU^0].
\eeq
The right-hand side contains only components of $\wU^0$ and $\wS^1[\wU^0]$, as well as components $\Phi^1, \wA^1_a$ of $\wU^1$ that have
already been found. Writing $rhs$ to denote the right-hand side, 
it follows from Lemma \ref{lem:U0Nmu}, Lemma \ref{lem:S1U0} and \eqref{wu1.est}
that
\beq\label{wA1.rhs}
    \| \mbox{rhs} \|_{L^\infty_T X^{L-3,j}_\gamma(\R^d)} + 
   \kz \| \wpp_0(\mbox{rhs}) \|_{L^\infty_T X^{L-4,j}_\gamma (\R^d)} 
    \le C.
\eeq
It follows from Proposition \ref{scalar.pde} that \eqref{wA1alpha.eqn} has a unique solution that satisfies
\beq\label{wA1.est}
\| \wA^1_\alpha \|_{L^\infty_T X^{L-3,j}_\gamma(\R^d)} +
\kz\| \wpp_0 \wA^1_\alpha \|_{L^\infty_T X^{L-4,j}_\gamma(\R^d)} 
    C(\Theta, K, \gamma).
\eeq
Thus we have completed the construction of $\wU^1$ satisfying  equation \eqref{basecase0} and estimates \eqref{wS.est2}.

Finally, to verify \eqref{wS.est1}, note from \eqref{k=1.1} and \eqref{P.formula} that  \eqref{basecase0} implies that
\beq\label{recallS}
\wS_\ep = \wS^\ep[\wU_{1,\ep}] = \ep^4p^0_2 + \ep^6( p^1_2+p^0_3) + \ep^8 p^1_3
\eeq
where $p^l_j$ are implicitly defined in \eqref{eq:wSe4}. Every term of $p^l_j$ is a product of $j$ components
of $\wZ^1 = (\wU^1, \kz\wpp_0 \wU^1, \nabla \wU^1)$ and at most $3-j$ components of  
$\wZ^0 = (\wU^0, \kz\wpp_0 \wU^0, \nabla \wU^0)$. We readily deduce estimates of norms of $\wZ^1$ from \eqref{wS.est2}.
It follows from Lemma \ref{lem:U0Nmu} that multiplication by $\wZ^0$ preserves membership in $X^{k,j}_\gamma$ for $\frac d2< k \le L-1$, 
so it follows that
\[
\| p^l_j \|_{X^{L-4, j}_\gamma} + \kz\| \wpp_0 p^l_j \|_{X^{L-5, j}_\gamma} \le C( \| \wZ^1 \|_{X^{L-4, j}_\gamma}  + \kz \| \wpp_0\wZ^1 \|_{X^{L-5, j}_\gamma} )\le C.
\]
(Note that we have used \eqref{p00q}.)
This and \eqref{recallS} imply \eqref{wS.est1}.

\medskip

\subsection{completion of proof}\label{sec:3.4}

We now complete the proof by induction on $m$, with $L$ fixed. The base case is provided by section \ref{sec:3.3}.

We will take our approximate solution $\wU_{1,\ep}$ to equal  $\wU^0+ \ep^2 \wU^1$,
and for $m\ge 2$ we will take $\wU_{m,\ep}$ to have the form
\begin{equation}\label{expand.Uke}
\wU_{m,\ep} =
\wU^0+ \ep^2( \wU^1 + c^2_\mu N_\mu) +\cdots + \ep^{2(m-1)}(\wU^{m-1}+ c^m_\mu N_\mu) +\ep^{2m}\wU^m.
\end{equation}

The induction hypothesis is that for $m \ge 1$ such that 
\[
k_m := L-4m>d/2,
\] 
there exists $T_0\in (0,T]$ and
functions
$\wU^m:\R^d \times [0,T_0)\to \C\times \R^{d+1}$ and
$c^m_\mu:\R^{d-2}\times [0,T_0)\to \R$ for 
$\mu=1,\ldots, 2N$, such that
the following estimates hold for $m\ge 2$: 
\beq\label{cmmu.est}
\|c^m_\mu \|_{L^\infty_{T_0} H^{k_m+2}(\R^{d-2})} +
\kz\|\wpp_0 c^m_\mu \|_{L^\infty_{T_0} H^{k_m+1}(\R^{d-2})} \le C, 
\eeq
\beq\label{Um.est}
\|\wU^m \|_{L^\infty_{T_0} X^{ k_m+2,j}_\gamma(\R^{d})} + \kz\|\wpp_0 \wU^m \|_{L^\infty_{T_0} X^{k_m+1,j}_\gamma(\R^{d})}
\le C, 
\eeq
and such that the error of approximation $\wS^\ep[\wU_{\ep,m}]$
admits an expansion 
\begin{equation}\label{ih1}
\wS^\ep[\wU_{\ep,m}] =  \ep^{2(m+1)}\sum_{j=0}^{2m} \ep^{2j}\wS_{m,j}
\end{equation}
where every $\wS_{m,j}(y_a, y_{\bar a}, y_0)$ is a function independent of $\ep$ such that
for every $m,j$
\begin{equation}\label{wSmj.est}
\| \wS_{m,j}\|_{L^\infty_{T_0}X^{k_m,j}_\gamma} + 
\kz\|\wpp_0 \wS_{m,j}\|_{L^\infty_{T_0}X^{k_m-1,j}_\gamma} 
\le C.
\end{equation}
(Note that \eqref{Um.est} does not hold for $m=1$. However, the weaker estimate established in \eqref{wS.est2} suffices for our construction.)

We assume that $k_{m+1}>\frac d2$, see \eqref{parameters}, and we will construct $c_\mu^{m+1}$ and $\wU^{m+1}$
such that the induction hypotheses hold with $m$ replaced by $m+1$.

We first define
\[
\wV_{m,\ep} = \wU_{m,\ep} + c^{m+1}_\mu N_\mu
\]
for functions $c^{m+1}_\mu:\R^{d-2}\times (0,T)\to \R$ that will be chosen to satisfy an orthogonality condition, as  specified in the following lemma.

\begin{lemma}\label{lem:Vek}There exists $T_0\in (0,T)$ such that for every $m\ge 1$ such that $k_m>d/2$ and every $\mu = 1,\ldots, 2N$, there exist $c^{m+1}_\mu\in C(0,T_0, H^{k_m+1}(\R^{d-2}))$ such that $\wS^\ep[\wV_{\ep,m}]$ admits an expansion
\begin{equation}\label{wSe.expand}
\wS^\ep[\wV_{\ep,m}] =  \ep^{2(m+1)}\sum_{j=0}^{2m} \ep^{2j}\wS^m_j
\end{equation}
with the $u$-components of the leading term orthogonal to the zero modes:
\begin{equation}\label{Sk0.orth}
( \wS^m_{0, u} , n_\mu)_{L^2(\R^2)} = 0 \qquad\mbox{ for }\mu=1,\ldots, 2N, \quad \mbox{ for every }(y_{\bar a}, y_0)\ .
\end{equation}
Moreover, $\kz \wpp_0^j c^{m+1}_\mu\in C(0,T_0, H^{k_m+1-j}(\R^{d-2})$ for $j=1,2$,
with
\beq\label{cm+1mu}
\| c^{m+1}_\mu \|_{L^\infty_{T_0} H^{k_m+1}}
+
\kz \| \wpp_0 c^{m+1}_\mu \|_{L^\infty_{T_0} H^{k_m}}
+
\kz \| \wpp_0\wpp_0 c^{m+1}_\mu \|_{L^\infty_{T_0} H^{k_m-1}}\le C,
\eeq
and every $\wS^m_j$ satisfies the estimate
\beq\label{wSkj.est}
\| \wS^m_j \|_{L^\infty_{T_0} X^{k_m,j}_\gamma} +
\kz \|  \wpp_0\wS^m_j \|_{L^\infty_{T_0} X^{k_m-1,j}_\gamma} 
\le C
\eeq\end{lemma}

\begin{proof}
Tp establish the lemma, we will show that the orthogonality condition \eqref{Sk0.orth} reduces to a system of equations for $c^{m+1}_\mu$ for $\mu=1,\ldots, 2N$, see \eqref{orth.pde}. We then prove that this system has a solution satisfying the estimates in \eqref{cm+1mu}, and that \eqref{wSkj.est} follows.

In carrying out this argument, we will  prove a priori estimates for  $c^{m+1}_\mu$ satisfying \eqref{cmmu.est}, for some $T_0>0$ to be determined. These estimates will imply existence of $c^{m+1}_\mu$ with the stated properties, via a fixed-point argument.

Using computations from Section \ref{sec:3.2} above, we find that
\begin{equation}\label{ind1}
\begin{aligned}
\wS^\ep[\wV_{m,\ep}] &
= \wS^\ep[\wU_{m,\ep} + \ep^{2m} c^{m+1}_\mu N_\mu ] \\
&
= \wS^\ep[\wU_{m,\ep}] + \ep^{2m}\ell^0[\wU_{m,\ep}](c^{m+1}_\mu N_\mu) \\
&\hspace{6em}
+ \ep^{2m+2}  \ell^1[\wU_{m,\ep}](c^{m+1}_\mu N_\mu)  + \ep^{4m}P_{m,\ep}
\end{aligned}
\end{equation}
where $P_\ep$ is a polynomial in $\wZ_{m,\ep}, \wtZ^{m+1}$ and $\ep$, for
\begin{equation}
\begin{aligned}
        \wZ_{m,\ep} &= (\wU_{m,\ep}, \kz\wpp_0 \wU_{m,\ep}, \nabla \wU_{m,\ep}), \\
    \wtZ^{m+1} &= (c^{m+1}_\mu N_\mu, \kz \wpp_0(c^{m+1}_\mu N_\mu), \nabla (c^{m+1}_\mu N_\mu)).
\end{aligned}\label{wtZj.def}
\end{equation}
This follows directly from \eqref{eq:wSe2} and \eqref{eq:wSe4}, with  $P_{m,\ep} = p^0_2 + \ep^2 p^1_2 + \ep^{2m}p^0_3 + \ep^{2m+2}p^1_3$.

We now expand the right-hand side of \eqref{ind1} as a polynomial in $\ep$, with coefficients that are
functions of $y$. The expansion of $\wS^\ep[\wU_{m,\ep}]$ is given in \eqref{ih1}. From the definition \eqref{eq:wSe5}, we see that the coefficients of the operator $\ell^0[\wU]$ are polynomials, at most quadratic, in $\wZ = (\wU,\kz\wpp_0\wU, \nabla \wU)$. By expanding, we thus find certain linear operators $\Lambda^0_j$ whose coefficients are polynomials, at most quadratic, in $\wZ^j = (\wU^j,\kz\wpp_0 \wU^j, \nabla \wU^j)$
for $j=0,\ldots, m$ and $\wtZ^j$, defined as in \eqref{wtZj.def}, for $j=2,\ldots, m$, such that
\beq\label{ell0.expand}
\begin{aligned}
    \ell^0[\wU_{\ep,m}](c_\mu^{m+1} N_\mu)
&=
\ell^0[\wU^0](c^{m+1}_\mu N_\mu) + \sum_{j=1}^{2m}\ep^{2j}\Lambda^0_j(c^{m+1}_\mu N_\mu)
\\&= 
\sum_{j=1}^{2m}\ep^{2j}\Lambda^0_j(c^{m+1}_\mu N_\mu) 
    \end{aligned}
    \eeq
by Lemma \ref{lem:kerell0}. 

We claim that $\Lambda^0_j(c^{m+1}_\mu N_\mu)(\cdot, t) \in  X^{k_m,j}_\gamma = H^{k_m}\cap W^{k_m-j,\infty}_\gamma(\R^{d})$ for every $t\in (0,T_0)$, and
\beq\label{Lambda0j.est}
\| \Lambda^0_j(c^{m+1}_\mu N_\mu)(\cdot, t)\|_{X^{k_m,j}_\gamma(\R^{d})} \le C \Theta_1(t)
\eeq
for
\beq\label{Theta1.def1}
\Theta_1(t) := \| c^{m+1}_\mu(\cdot, t) \|_{ H^{k_m+1}(\R^{d-2})}  + \kz \| \wpp_0 c^{m+1}_\mu(\cdot, t) \|_{ H^{k_m}(\R^{d-2})}
\eeq
for every $t\in (0,T_0)$.
This is clear, since $\Lambda^0_j$ is a first-order differential operator with coefficients that are products of functions that either belong to $X^{k_m+1,j}_\gamma\subset X^{k_m,j}_\gamma$, which is an algebra, or are $X^{k_m,j}_\gamma$ multipliers. This follows from the induction hypothesis, Lemma \ref{lem:U0Nmu} and Remark \ref{rmk:multiplier}.

Similarly, we can expand
\[
 \ell^1[\wU_{m,\ep}](c^{m+1}_\mu N_\mu) = \ell^1[\wU^0](c^{m+1}_\mu N_\mu) + \sum_{j=1}^{2m}\ep^{2j}\Lambda^1_j(c^{m+1}_\mu N_\mu)
\]
for first-order linear operators $\Lambda^1_j$, and
\[
P_{m,\ep} =  \sum_{j=0}^{m}\ep^{2j}P_{m,j}.
\]
Recalling that the polynomial terms are at most cubic in $c^{m+1}_\mu N_\mu$, we find by the same considerations that lead to \eqref{Lambda0j.est} that 
\beq
\| \Lambda^1_j(c^{m+1}_\mu N_\mu(\cdot, t))\|_{X^{k_m,j}_\gamma(\R^{d})} 
\le C\Theta_1(t)
\label{Lambda1j.est}\eeq
and
\beq\label{Pmj.est}
\begin{aligned}
\| P_{m,j}(\cdot, t)\|_{X^{k_m,j}_\gamma(\R^{d})} &\le 
C\left(
\Theta_1(t) +\Theta_1(t)^3\right)
\end{aligned}\eeq
for $t\in (0,T_0)$.
Combining these, keeping track of the powers of $\ep$ that appear in the expansion, and noting that $4m = 2(m+1)$ when $m=1$, we find that \eqref{wSe.expand} holds, with
\begin{equation}\label{Sk0}
\wS^m_0 = \begin{cases}
 \ell^1[\wU^0](c^{2}_\mu N_\mu) 
+ \Lambda^0_1(c^{2}_\mu N_\mu) + \wS_{1,0} + P_{1,0}
&\mbox{ if }m = 1  \\
\ell^1[\wU^0](c^{m+1}_\mu N_\mu)
+ \Lambda^0_1(c^{m+1}_\mu N_\mu) + \wS_{m, 0} 
&\mbox{ if }m\ge  2
\end{cases}
\end{equation}
and for $j\ge 1$,
\beq\label{Smj}
\wS^m_j = \mbox{a sum of certain }S_{m,k_0}, \  \Lambda^0_{k_1}(c^{m+1}_\mu N_\mu), \ \Lambda^1_{k_2}(c^{m+1}_\mu N_\mu),\  P_{m,k_3}  
\eeq
for  $k_0, \ldots,k_3$ depending on $m,j$ in a way that will not matter for us.

We next consider the orthogonality condition \eqref{Sk0.orth}.
Using the definitions \eqref{eq:wSe6}, \eqref{eq:wSe7} of $\ell^1$,
we check that
\beq\label{ell1phi}
\ell^1_\phi[\wU^0](c^{m+1}_\mu N_\mu) =  [(\kz \wpp_0\wpp_0 +\ko\wpp_0 - \wpp_{\bar a}\wpp_{\bar a})c^{m+1}_
\mu] n_{\mu,\phi} + \mbox{Tri} 
\eeq
where $\mbox{\em Tri}$  denotes a trilinear function of $(\wZ^0,\wZ^0, \wtZ^{m+1})$. Arguing as above,
we see that
\beq
\begin{aligned}
\|  \mbox{Tri}\|_{X^{k_m,j}_\gamma(\R^{d})} &\le 
C \Theta_1(t)\label{TriZ.est}
\end{aligned}
\eeq
Similarly, recalling that $\wpp_a c^{m+1}_\mu=0$,
\beq
\begin{aligned}
    \ell^1_a[\wU^0](c^{m+1}_\mu N_\mu) 
    &= (\kz\wpp_0+\ko)\left(\wpp_0 (c^{m+1}_\mu n_{\mu,a}) - \wpp_a( c^{m+1}_\mu N_{\mu,0})\right) \\
&\hspace{5em}-\wpp_{\bar a} \left( \wpp_{\bar a} (c^{m+1}_\mu n_{\mu,a}) - \wpp_a( c^{m+1}_\mu N_{\mu,\bar a}) \right)\\
&= [(\kz \wpp_0\wpp_0 +\ko\wpp_0 - \wpp_{\bar a}\wpp_{\bar a})c^{m+1}_
\mu] n_{\mu,a} + \mbox{Bi}
\end{aligned}
\eeq
where $\mbox{\em Bi}$ denotes a bilinear function of $(c^{m+1}_\mu,\kz\wpp_0 c^{m+1}_\mu,  \wpp_{\bar a}c^{m+1}_\mu)$ and terms
drawn from
\[
N_\mu,\   \nabla  N_\mu, \,   \wpp_0 N_\mu,\ 
\kz \wpp_0\nabla N_\mu, \ \kz \wpp_0\wpp_0 N_\mu,
\wpp_{\bar a}\nabla N_\mu . 
\]
In addition, $\wpp_0 N_\mu$ appears with a coefficient of $\kz$ except when multiplying $\kz \wpp_0 c^{m+1}_\mu$.
As above\footnote{The constant diverges as $\kz\searrow0$ but it finite for $\kz=0$. This is due to the presence of an unfavorable term $\kz \wpp_0 c^{m+1}_\mu \cdot \wpp_0 n_{\mu,a}$, which we can only say is bounded by $\kz^{-1}\Theta (\kz \|c^{m+1}_\mu\|_{H^{k_m+2}})$, but which vanishes when $\kz=0$.}
\beq
\begin{aligned}
\| \mbox{Bi}\|_{X^{k_m,j}_\gamma(\R^{d})} &\le C\Theta_1(t)
\label{Bi.est}
\end{aligned}
\eeq
We thus obtain
\[
\ell^1_u[\wU^0](c^{m+1}_\mu N_\mu)
=
[(\kz \wpp_0\wpp_0 +\ko\wpp_0 - \wpp_{\bar a}\wpp_{\bar a})c^{m+1}_
\mu] n_\mu + \mbox{Bi} + \mbox{Tri},
\]
with \eqref{Bi.est}, \eqref{TriZ.est} holding.

Recalling that $\{ n_\mu(\cdot;q)\}_{\mu=1}^{2N}$ are orthonormal in $L^2(\R^2)$ for every $q\in M_N$,
we deduce from the above that the orthogonality condition \eqref{Sk0.orth} holds if and only if
\beq\label{orth.pde}
(\kz \wpp_0\wpp_0 +\ko\wpp_0 - \wpp_{\bar a}\wpp_{\bar a})c^{m+1}_\nu= rhs_\nu, \qquad\nu=1,\ldots, 2N
\eeq
for
\[
rhs_\nu(y_{\bar a}, y_0) :=
\big( (RHS), n_\nu(q))_{L^2(\R^2)(y_{\bar a}, y_0)}
\]
with 
\[
RHS = 
\begin{cases}
- S_{1,0,u} - \Lambda^0_{1,u}(c^{2}_\mu N_\mu) - Bi - Tri - P_{1,0,u}
&\mbox{if }m=1\\
- S_{m,0,u} - \Lambda^0_{1,u}(c^{m+1}_\mu N_\mu) - Bi - Tri
&\mbox{if }m\ge2.   
\end{cases}
\]
Note in addition that for $m\ge 2$,  $rhs_\nu$ is a linear function of $(c^{m+1}_\mu)_{\mu=1}^{2N}$ and its first derivatives, and for $m=1$ it has polynomial nonlinearities that are at most cubic.

Since $X^{k_m,j}_\gamma\subset H^{k_m}$, we see from 
\eqref{wSmj.est}, \eqref{Lambda0j.est}, \eqref{Lambda1j.est}, and \eqref{Pmj.est}
that 
\[
\| RHS(\cdot, t) \|_{H^{k_m}(\R^d)} \le C (1+\Theta_1(t) + \delta_{m1}\Theta_1(t)^3).
\]
where $\delta_{m1}$ is a Kronecker delta. Since $n_\mu$ is a $H^{k_m}$ multiplier, see Lemma \ref{lem:U0Nmu}, we know that $\| RHS \cdot n_\mu(q)(\cdot, t)\|_{H^{k_m}} \le C \|RHS(\cdot, t)\|_{H^{k_m}}$. Lemma \ref{lem:last} in the appendix implies that
\begin{align}
\| rhs(\cdot, t) \|_{H^{k_m}(\R^{d-2})} &\le C 
\|RHS(\cdot, t)\|_{H^{k_m}(\R^d)} \nonumber\\
&\le C (1+\Theta_1(t) + \delta_{m1}\Theta_1(t)^3).
\label{rhs.est}
\end{align}
For initial conditions $c^{m+1}_\nu = \kz\wpp_0 c^{m+1}_\nu = 0$, standard theory guarantees that  \eqref{orth.pde}, \eqref{rhs.est} has a unique solution on $\R^{d-2}\times [0, T_0)$ for some $T_0 \in (0,T]$, satisfying
\[
\| c^{m+1}_\mu \|_{L^\infty_{T_0} H^{k_m+1}}
+
\kz \| \wpp_0 c^{m+1}_\mu \|_{L^\infty_{T_0} H^{k_m}} \le C.
\] 
Moreover, $T_0$ admits a lower bound that depends only on $C$. 
We review this theory for the nonlinear case $m=1$ in Lemma \ref{lem:dhs}, in Appendix \ref{app:hyp}. The linear case is very similar. Note that it is only for $m=1$ that the equation \eqref{orth.pde} is nonlinear and we may have to decrease the interval of existence $T_0$.

It follows from \eqref{orth.pde} that $\kz \wpp_0\wpp_0 c^{m+1}_\mu$ has the same regularity as $\wpp_{\bar a}\wpp_{\bar a}c^{m+1}_\mu$, so \eqref{cm+1mu} follows by
taking $T_0$ smaller if necessary. With control over $\kz\wpp_0\wpp_0 c^{m+1}_\mu$ and over $\Theta_1(t), 0<t<T_0$, we can go back and, by the same reasoning used to 
deduce \eqref{Lambda0j.est}, \eqref{Lambda1j.est}, \eqref{Pmj.est},
conclude that
\[
\left.
\begin{aligned}
\| \Lambda^0(c^{m+1}_\mu N_\mu)\|_{L^\infty_{T_0}X^{k_m,j}_\gamma}
+
\kz \| \wpp_0 \Lambda^0(c^{m+1}_\mu N_\mu)\|_{L^\infty_{T_0}X^{k_m-1,j}_\gamma}
\\
\| \Lambda^1(c^{m+1}_\mu N_\mu)\|_{L^\infty_{T_0}X^{k_m,j}_\gamma}
+
\kz \| \wpp_0 \Lambda^1(c^{m+1}_\mu N_\mu)\|_{L^\infty_{T_0}X^{k_m-1,j}_\gamma}
\\
\| P_{m,j}\|_{L^\infty_{T_0}X^{k_m,j}_\gamma}
+
\kz \| \wpp_0 P_{m,j}\|_{L^\infty_{T_0}X^{k_m-1,j}_\gamma}
\end{aligned}\right\} \le C.
\]
Then \eqref{wSkj.est} follows immediately from these estimates, \eqref{wSmj.est} and \eqref{Smj}.

\end{proof}

Having fixed $c^{m+1}_\mu$ and hence $\wV_{\ep,m}$ and $\wS^m_j, j=0,\ldots, 2m$, we now establish another 
important property of $\wS^m_0$.

\begin{lemma}\label{lem:Sk0.go}
$\wS^m_{0,u} (\cdot, y_{\bar a}, y_0)= \binom{\wS^m_{0,\phi}}{\wS^m_{0,a}}(\cdot, y_{\bar a}, y_0)$ satisfies the gauge-orthogonality condition \eqref{eq:go} for every $(y_{\bar a}, y_0)$.
\end{lemma}

\begin{proof}As with Lemma \ref{lem:J2}, we start from the algebraic identity in Lemma \ref{lem:identity}, which after rescaling to the $y$ variables implies that
\[
0= (\wS^\ep_\phi [\wV_{\ep,m}] , i \Phi_{\ep, m}) 
- \pp_{ a}\wS^\ep_{a}[\wV_{\ep,m}] 
 +\ep^2 \left(- (\kz\wpp_0 + \ko)\wS^\ep_0 [\wV_{\ep,m}]
- \pp_{\bar a}\wS^\ep_{\bar a}[\wV_{\ep,m}]  \right)
\]
Now we use \eqref{wSe.expand} and \eqref{expand.Uke} to expand this identity in powers of $\ep$.
This yields
\[
0 = \ep^{2m+2}\left( (S^m_{0, \phi}, i\mphi(q)) - \pp_a \wS^m_{0,a}
\right)
+ \ep^{2m+4} \big( \ \cdots \ \big) + \cdots
\]
Since this identity holds for every $\ep$ and the coefficient functions are independent of $\ep$, it follows that
\[
(S^m_{0, \phi}, i\mphi(q)) - \pp_a \wS^m_{0,a}
=0
\]
which is what we need to prove.
\end{proof}

We are now ready to complete the induction argument. We must find $\wU^{m+1}$ such that
\[
\wU_{m+1,\ep} =  \wU_{m,\ep} + \ep^{2m}c_\mu^{m+1}N_\mu + \ep^{2m+2}\wU^{m+1} = \wV_{\ep,m} + \ep^{2m+2}\wU^{m+1}
\]
satisfies the conclusions of the proposition.
Using \eqref{ih1} and computations from Section \ref{sec:3.2} above, we compute
\beq\label{Sep.m+1}
\wS^\ep[\wU_{m+1,\ep}] =  \ep^{2m+2}\left( \wS^m_0 +  \ell^0[\wU^0](\wU^{m+1})\right) + \ep^{2m+4}\sum_{j=0}^{2(m+1)}\ep^{2j}\wS_{m+1,j}.
\eeq
The sum on the right-hand side 
gathers all the terms arising from
the polynomial terms $p^m_j$ for $m=0,1$ and $j=2,3$, appearing in \eqref{eq:wSe4}, together with
\beq\label{S.ingredients}
\begin{aligned}
    &\ep^{2m+2}\sum_{j=1}^{2m} \ep^{2j}\wS^m_j \qquad\mbox{ from \eqref{wSe.expand}}\\
    &\ep^{2m+4}\ell^1[\wV_{\ep,m}](\wU^{m+1}),\\
& \ep^{2m+2}\left(\ell^0[\wV_{\ep,m}](\wU^{m+1}) -  \ell^0[\wU^0](\wU^{m+1})\right).
\end{aligned}   
\eeq
For the last of these, notice that by expanding $\ell^0[\wV_{\ep,m}]$ as
in \eqref{ell0.expand}, we obtain a polynomial in $\ep^2$ with coefficient functions independent of $\ep$, whose leading term is of order $\ep^{2m+4}$ as required.

We must choose $\wU^{m+1}$ so that
\begin{equation}\label{Uk+1}
 \ell^0[\wU^0](\wU^{m+1}) + \wS^m_0 = 0
\end{equation}
This can be done exactly as in Step 3 above: we first consider the components $\wU^{m+1}_u =: \wu^{m+1} = \binom {\Phi^{m+1}}{\wA_a^{m+1}}$. As above, we take $\wu^{m+1}$ to solve the system of equations
\[
L[\modu(q)](\wu^{m+1}) =- \wS^m_{0,u} \qquad\mbox{ on }\R^d\times (0,T_0).
\]
The orthogonality condition established in Lemma \ref{lem:Vek} provides exactly the condition for solvability of this system, see Proposition \ref{system.pde} in Appendix \ref{app:system}. It follows from the proposition and the regularity \eqref{wSkj.est} of the right-hand side that there exists a unique solution satisfying \eqref{system.gauge1}, with the estimates
\beq\label{um+1.est}
\| \wu^{m+1} \|_{L^\infty_{T_0}X^{k_m,j}_\gamma} + \kz \| \wpp_0\wu^{m+1} \|_{L^\infty_{T_0}X^{k_m-1,j}_\gamma} \le C.
\eeq
It follows from  Lemmas \ref{lem:Sk0.go} and  \ref{vsolvability1} that $\wu^{m+1}(\cdot, y_{\bar a}, y_0)$ is gauge-orthogonal for every $(y_{\bar a}, y_0)$ and hence that 
\[
\ell^0_u[\wU^0](\wU^{m+1}) \overset{\eqref{eq:wSe5}}= \calL[\modu(q)](\wu^{m+1}) \overset{\eqref{L=calL}}= L[\modu(q)](\wu^{m+1}) = -\wS^m_{0,u}
\]
in $\R^d\times (0,T_0)$. Thus the $u$-components of \eqref{Uk+1} are satisfied.

To complete the construction of $\wU^{m+1}$ it now suffices to choose $\wA^{m+1}_0, \wA^{m+1}_{\bar a}$ such that
\[
\ell^0_\alpha[\wU^0](\wU^{m+1}) = - \wS^m_{0,\alpha} \quad\mbox{ on }(0,T_0)\times \R^d\qquad\qquad\mbox{ for }\alpha= 0,3,\ldots, d.
\]
This equation states that
\[
\begin{aligned}
    (-\pp_a\pp_a+|\mphi(q)|^2)\wA^{m+1}_\alpha &= - \pp_\alpha\pp_a \wA^{m+1}_a -
 (i\Phi^{m+1}, \wD_{\wA^0_\alpha}\mphi(q))
 \\
 &\hspace{5em}-  (i\mphi(q)), \wD_{\wA^0_\alpha}\Phi^{m+1}) - \wS^m_{0,\alpha} .
\end{aligned}
\]
The right-hand side contains only components of $\wU^0$ and components $\Phi^{m+1}, \wA^{m+1}_a$ of $\wU^{m+1}$ that have already been found.
The first term on the right-hand side is the least regular, and 
from \eqref{um+1.est} it satisfies
\[
\| \wpp_\alpha\wpp_a \wA^{m+1}_a \|_{L^\infty_{T_0} X^{k_m-2, j}_\gamma }
+
\kz \| \wpp_0\wpp_\alpha\wpp_a \wA^{m+1}_a \|_{L^\infty_{T_0} X^{k_m-3, j}_\gamma .
}\le C.
\]
It therefore follows from Proposition \ref{scalar.pde} that for $\alpha=0,3,\ldots, d$,
this equation has a unique solution that satisfies
\[
\| \wA^{m+1}_\alpha \|_{L^\infty_{T_0} X^{k_m-2, j}_\gamma }
+
\kz \| \wpp_0 \wA^{m+1}_\alpha \|_{L^\infty_{T_0} X^{k_m-3, j}_\gamma .
}\le C.
\]
Since $\wU^{m+1} = (\wu^{m+1}, \wA^{m+1}_{\bar a} , \wA^{m+1}_{0})$ is as regular as the least regular of its components, it follows that $\wU^{m+1}$ has the regularity required by the induction hypothesis \eqref{Um.est}

To complete the induction proof, we consider $\wS^\ep[\wU_{m+1,\ep}]$.
In view of \eqref{Sep.m+1} and our choice of $\wU^{m+1}$, we see that
$\wS^\ep[\wU_{m+1,\ep}]$ consists of a finite number of terms of order $\ep^{2m+4}$ or smaller, as required by \eqref{ih1}, so it only remains to 
show that they have the regularity required by the induction hypothesis \eqref{wSmj.est}.  For the terms $\wS^m_j$ appearing in \eqref{S.ingredients}, this has been established in \eqref{wSkj.est}. Every other term in the \eqref{S.ingredients} is a product of functions involving at most two derivatives of $\wU^{m+1}$
multiplied by functions arising from earlier stages in the construction that are more regular or are multipliers in more regular spaces, see Remark \ref{rmk:multiplier}. 
It follows that $\wS^\ep[\wU_{m+1,\ep}]$ has the regularity of $\ell^1[\wV_{\ep,m}](\wU^{m+1})$, the worst term appearing in the expansion. This and our estimates of $\wU^{m+1}$ imply that 
\[
\| \wS^\ep[\wU_{m+1,\ep}] \|_{L^\infty_{T_0} X^{k_m-4, j}_\gamma }
+
\kz \| \wS^\ep[\wU_{m+1,\ep}] \|_{L^\infty_{T_0} X^{k_m-5, j}_\gamma } 
\le C .
\]
Since $k_m-4 = k_{m+1}$, this completes the induction step.

\section{the modified system}\label{sec:auxsys}

We henceforth fix a solution $q:\R^{d-2}\times (0,T)\to M_N$ of \eqref{dwm} satisfying our standing assumptions \eqref{q.compact0}, \eqref{q.decay}, and with $\Theta<\infty$, see \eqref{Theta.def0}. For the choice of $L$ in \eqref{q.decay} and \eqref{Theta.def0},
we fix an integer $\ell$ such that $2\ell>\frac d2+1$, and we assume that the parameters in Proposition \ref{prop:approx.sol} satisfy
\beq\label{fix.parameters}
m > \frac d4+1, \qquad  j =  \lfloor d/2\rfloor>\frac d2-1, \qquad L \ge 4m+2\ell+j-1 .
\eeq
For example we can take $m =\lfloor d/4\rfloor +2$ and $\ell =\lfloor (d+2)/4\rfloor +1$. For these values,
the choice 
$L = 2d+11$ from  Theorems \ref{thm:1} and \ref{thm:2} satisfies the requirement from \eqref{fix.parameters}, but $L$ can be taken a bit smaller if $d= 1,2$, or $3 \mod 4$.


Given \eqref{fix.parameters}, the approximate solution $U_{m,\ep} = U^\ep_0 + \ep^2 U^\ep_m$ provided by the proposition satisfies
\begin{align}
    \| S[U_{m,\ep}] \|_{L^\infty_{T_0}H^{k_m}} &\le C\ep^{9/2} \qquad\mbox{ with }k_m \ge d+1, \label{good1}\\
    \| U_{m,\ep}\|_{L^\infty_{T_0}W^{2\ell+1, \infty}} + \ep^{-1} \kz  \| \pp_0 U_{m,\ep}\|_{L^\infty_{T_0}W^{2\ell, \infty}} &\le C, \label{good2}
\\
     \| U^\ep_m \|_{L^\infty_{T_0}W^{2\ell+1, \infty}}+ \ep^{-1} \kz \|\pp_0 U^\ep_m \|_{L^\infty_{T_0}W^{2\ell, \infty}} 
    &\le C. \label{good3}
\end{align}

We now set $\aU := U_{m,\ep}$ for the above choice of $m$.
We will prove the main theorem by finding $\tU = \binom \tPhi \tA$ that is small in suitable norms, and such that
\beq\label{tU.eqn1}
S[\aU+\tU] = 0.
\eeq
This is now a system of equations for $\tU$. In this section we will rewrite \eqref{tU.eqn1} in a way that will facilitate our later analysis.

These equations inherit a gauge symmetry from \eqref{dAHM}. A key point in our rewriting of \eqref{tU.eqn1} will be to eliminate this unwanted degree of freedom
by looking for solutions that satisfy the gauge condition
\beq\label{dynamic.gauge}
G[\aU](\tU) :=
-(\kz\pp_0+\ep\ko)\tA_0 + \pp_k \tA_k
- ( i\aPhi, \tPhi) = 0
\eeq

We will use the abbreviations
\begin{equation}\label{abbrev}
q_\ep = q_\ep (x_{\bar a}, t) = q(\ep x_{\bar a}, \ep t), 
\qquad
L[q_\ep] := L[\modu(q_\ep)],
\end{equation}
and we introduce the norm $\vvvert\,\cdot\,\vvvert$ defined by
\begin{equation}\label{nnnorm} 
\vvvert U \vvvert_k := \sup_{0<t<T_0/\ep} \left( \kz \| \pp_t U(\cdot, t)\|_{H^{k}(\R^d)} +
\| U(\cdot, t)\|_{H^{k+1}(\R^d)} \right)
\end{equation}
This same norm will appear later in a fixed-point argument used to prove our main existence theorems.

We will prove
\begin{proposition}\label{prop:auxsystem}
There exist constant $c_2, \ep_0>0$, linear differential operators $\calM, P$ and a nonlinear operator $\calN$, all acting on functions    
$\R^d\times [0,\frac{T_0}\ep)\to \C\times \R^{d+1}$, such that if $\tU$ is a solution of 
\begin{equation}\label{aux.system}
(\kz \pp_0^2+ \ep\ko\pp_0)\tU + \calM\tU + P\tU = -\calN(\tU) - S[\aU]    
\end{equation}
such that $\vvvert \tU\vvvert_{2\ell}<c$, 
with initial data such that
\beq\label{modified.data}
G[\aU](\tU) =0  \ \ \mbox{ and } \ \  S_0[\aU+\tU] = 0
\qquad\mbox{ at }t=0,
\eeq
and if $0<\ep<\ep_0$, then
\beq\label{prop.mod}
G[\aU](\tU) = 0  \ \ \mbox{ and } \ \  S[\aU+\tU] = 0 
\qquad\mbox{ in }\R^d\times [0, \frac {T_0}\ep).
\eeq
Moreover, 
\begin{equation}\label{M.def}
\calM\tU (x_a, x_{\bar a}, t) =  \begin{pmatrix}
(-\Delta_{x_{\bar a}} + L[q_\ep])\tu \\ 
(-\Delta +|\mphi(q_\ep)|^2)\tA_{\bar a}\\
(-\Delta +|\mphi(q_\ep)|^2)\tA_{0}
\end{pmatrix}
\end{equation}
and 
$P = P_1+ P_2$ where
\begin{equation}\label{P1.def}
\begin{aligned}
P_1\tU 
&= \begin{pmatrix}
    P_{1,\phi} \tU\\
    P_{1, a} \tU\\
    P_{1,\bar a}\tU\\
     P_{1,0}\tU
\end{pmatrix}
= 
\begin{pmatrix}
    2i \left[ - \kz \tA_{0}\De_{0}\aPhi +\tA_{\bar a}\De_{\bar a}\aPhi  - \kz\aAzero \pp_0\tPhi + \aAbara \pp_{\bar a}\tPhi \right]\\
    0\\
    -2(i\tPhi, \De_{\bar a}\aPhi)\\
    -2(i\tPhi, \De_{0}\aPhi)
\end{pmatrix}
    \end{aligned}
\end{equation}
and $P_2$ is a first-order linear differential operator with smooth coefficients, say $P_{2,\alpha\beta}$, satisfying
\beq\label{P2.est}
\| P_{2,\alpha\beta} \|_{L^\infty_{T_0/\ep}W^{2\ell,\infty}( \R^d)} \le C \ep^2.
\eeq
Finally, 
if $k>d/2$ and $\vvvert \tU\vvvert_k, \vvvert \tV\vvvert_k \le 1$, then
\begin{equation}\label{aux.props2}
\begin{aligned}
     \| \calN (\tU)(\cdot, t)\|_{L^\infty_{T_0/\ep} H^k(\R^d)} &\le C \vvvert \tU \vvvert_{k}^2,\\
     \|\calN (\tU)(\cdot, t) - \calN(\tV) (\cdot, t)\|_{L^\infty_{T_0/\ep} H^k(\R^d)} &\le 
     C \left( \vvvert \tU\vvvert_k + \vvvert \tV\vvvert_k \right) \left( \vvvert \tU -  \tV\vvvert_k \right) \end{aligned}
\end{equation}
where $C$ depends on $d, k$ and $\vvvert \aU \vvvert_{k}$. \end{proposition}

The equation $S_0[\aU+\tU]=0$ is an elliptic equation and if $\kz>0$ imposes a constraint on the initial data, as is standard in Yang-Mills-Higgs theories. If $\kz=0$ it holds automatically. 

From the explicit form of $P_1$ and estimates \eqref{Uep.size} of the approximate solution, as well as the choice \eqref{fix.parameters} of $L$,  we see that the coefficients of $P_1$, say $P_{1,\alpha\beta}$, satisfy
\beq\label{P1.est}
\| P_{1,\alpha\beta} \|_{L^\infty_{T_0/\ep}W^{2\ell,\infty}( \R^d)} \le C \ep.
\eeq
Similarly, the coefficients of  $\calM$, say $\calM_{\alpha\beta}$, satisfy
\beq\label{calM.est}
\| \calM_{\alpha\beta} \|_{L^\infty_{T_0/\ep}W^{2\ell,\infty}( \R^d)} \le C.
\eeq

\begin{proof}
{\bf Step 1}.  For simplicity we will write $G(\tU)$ instead of $G[\aU](\tU)$.  We will enforce the gauge condition \eqref{dynamic.gauge} by solving
the system
\begin{equation}\label{modify1}
\begin{aligned}
    S_\phi[\aU +\tU ] &= i(\aPhi +\tPhi) G(\tU)\\
    S_j[\aU +\tU ] &= \pp_j G(\tU), \qquad\qquad j=1,\ldots, d \\
    S_0[\aU +\tU ] &= \pp_0 G(\tU).
\end{aligned}
\end{equation}
We first claim that if $\tU$ solves \eqref{modify1} with initial data as described above, then $G(\tU) = 0$ and hence $\tU$ solves \eqref{tU.eqn1}. To see this, note that our choice of initial data immediately implies that $G(\tU)=0$ at $t=0$. Moreover, in the hyperbolic case $\kz>0$, it follows from \eqref{modified.data} and \eqref{modify1} that $\pp_0 G(\tU) \overset{\eqref{modify1}}= S_0[\aU+\tU] \overset{\eqref{modified.data}}= 0$ at $t=0$. We wish to apply Lemma \ref{lem:Z} to conclude that $G(tU)$ vanishes identically. To do this we must verify some assumptions. Regularity condition \eqref{suff.smooth1} on $\aU$ holds due to \eqref{good2}, and condition \eqref{suff.smooth2} on $\tU$ is exactly the assumption that $\vvvert \tU\vvvert_{2\ell}<\infty$, which also is easily seen to imply assumption \eqref{suff.smooth3}.
We must also check that there exists some $c>0$ such that
\[
\int_{\R^d} |\nabla f(x)|^2 + |\aPhi+\tPhi(x,t)|^2 f^2(x)\, dx \ge c \int_{\R^d} f^2(x)\, dx 
\]
for all $f\in H^1(\R^d)$.
To do this, note that $\aPhi = \mphi(q_\ep) + \ep^2\Phi^\ep_m$, elementary inequalities
imply that
\[
|\aPhi+\tPhi|^2 = |\mphi(q_\ep) + \ep^2\Phi^\ep_m +\tPhi|^2 \ge \frac 12 |\mphi(q_\ep)|^2 - |\ep^2\Phi^\ep_m +\tPhi|^2.
\]
For any fixed $t$, since we continue to assume \eqref{q.compact0}, we can use Lemma \ref{basic.elliptic} to compute
\begin{align*}
    \int_{\R^d} |\nabla f|^2 + |\aPhi+\tPhi|^2 f^2\, 
    &
    \ge \frac 12 \int_{\R^{d-2}} \int_{\R^2} |\nabla_a f|^2 + |\mphi(q_\ep)|^2 f^2 \, dx_a \, dx_{\bar a}\\
    &\hspace{10em}
    - \| \Phi^\ep_m +\tPhi|^2_{L^\infty} \|f \|_{L^2}^2\\
    &\ge \left(\frac c2 -  \| \ep^2\Phi^\ep_m +\tPhi|^2_{L^\infty}\right) \|f\|_{L^2}^2
    \end{align*}
where the constant $c$ comes from \eqref{basic.elliptic}. Since $\| \Phi^\ep_m\|_{L^\infty} \le C$, and because
$\| \tPhi\|_{L^\infty}\le C \vvvert \tPhi \vvvert_{2\ell} < C c_2$ by assumption, we can fix $\ep_0$ and $c_2$ so small that 
 \[
 \int_{\R^d} |\nabla f|^2 + |\aPhi+\tPhi|^2 f^2\, 
    \ge \frac c4 \|f\|_{L^2}^2 
 \]   
 whenever $\vvvert \tU\vvvert_{2\ell} < c_2$ and $0<\ep <\ep_0$. Thus the hypotheses of Lemma \ref{lem:Z} are satisfied, and it follows that $G(\tU)=0$.

{\bf Step 2}. 
Our goal now is to rewrite \eqref{modify1} in the form given in the statement of the proposition.
Toward this end, we define
\begin{align}
\label{Sprime.def}
S'[\aU](\tU) &:= \frac d{ds} S[\aU+s\tU]\Big|_{s=0},\\
\label{calN0.def}
\calN^0[\aU](\tU) &:= S[\aU+\tU] - S[\aU] - S'[\aU](\tU)
\end{align}
These are operators that act on sufficiently smooth functions $\tU: [0, \frac{T_0}\ep)\times \R^d\to \C\times \R^{d+1}$. 
With this notation, \eqref{modify1} can be rewritten
\begin{equation}\label{modify2}
\begin{aligned}
      S_\phi[\aU] + S_\phi'[\aU](\tU ) - i\aPhi G(\tU)&= - \calN^0_\phi[\aU](\tU) + i\tPhi G(\tU)\\
     S_j[\aU] + S_j'[\aU](\tU )  - \pp_j G(\tU)&= - \calN^0_j[\aU](\tU) \\
     S_0[\aU] + S_0'[\aU](\tU )  -\pp_0 G(\tU) &= - \calN^0_0[\aU](\tU) .
\end{aligned}
\end{equation}

In view of \eqref{modify2}, to complete the proof, it suffices to show that 
\beq\label{aux.suffices}
\left(\begin{aligned}
    S'_\phi[\aU](\tU) - i\aPhi  G(\tU) \\
    S'_j[\aU](\tU) - \pp_j G(\tU) \\
    S'_0[\aU](\tU)- \pp_0 G(\tU) \\
\end{aligned}\right) =(\kz \pp_0^2+ \ep\ko\pp_0)\tU + \calM\tU + P\tU
\eeq
for $\calM$ and $P$ as described in the statement of the Proposition, and that
\begin{equation}\label{props2bis}
\calN(\tU) 
 := \begin{pmatrix}
    \calN^0_\phi[\aU](\tU) - i \tPhi G(\tU)\\
    \calN_j^0[\aU](\tU)\\
    \calN_0^0[\aU](\tU)
\end{pmatrix}_{j=1,\ldots, d} 
\quad\mbox{ satisfies \eqref{aux.props2}.
}
\end{equation}

{\bf Step 3: proof of \eqref{props2bis}}
First note that $S[U]$ is a (constant coefficient) cubic polynomial function of $U$, $\nabla_{t,x} U$ and $\nabla^2_{t,x}U$, where
$\nabla_{t,x}^k\tU$ denotes the collection of all partial derivatives in all the $t,x$ variables of order $k$. Moreover, all second derivatives appear linearly in $S[U]$. It then follows from the definition \eqref{calN0.def} that $\calN^0[\aU](\tU)$ is a cubic polynomial in $\tU$ and $\nabla_{t,x}\tU$ containing only terms that are quadratic or cubic, and with coefficients that are linear functions\footnote{that is, polynomials of degree at most $1$.} of $\aU$. The same is clearly true of $i\tPhi G(\tU)$, so it remains true of $\calN(\tU)$. That is, $\calN(\tU)$ has the form
\begin{equation}\label{calN.form}
    \calN(\tU) = 
\sum_{\alpha_{i}, \alpha_j} L_{\alpha_{i}, \alpha_j}(\aU) Z_{\alpha_i} Z_{\alpha_j} 
+
\sum_{\alpha_{i}, \alpha_j, \alpha_k} C_{\alpha_{i}, \alpha_j,\alpha_k} Z_{\alpha_i} Z_{\alpha_j} Z_ {\alpha_k},
\end{equation}
where $Z_{\alpha}$ denotes a generic component of $\tU$ or $\nabla_{t,x}\tU$, every
$L_{\alpha_i \alpha_j}(\aU)$ is a linear function of $\aU$, and every $C_{\alpha_{i}, \alpha_j,\alpha_k}$ is a constant.

 When $\kz>0$, the first inequality in \eqref{aux.props2} follows directly from the standard fact that 
\[
\| u v \|_{H^s(\R^d)} \le C_{s,d} \| u  \|_{H^s(\R^d)} \|  v \|_{H^s(\R^d)}
\]
(which of course implies the corresponding inequality for products of more terms).
The second inequality follows by similar considerations, since
\[
x_i x_j - y_iy_j = (x_i-y_i)x_j + y_i (x_j-y_j)
\]
and 
\[
x_i x_j x_k - y_iy_j y_k = (x_i-y_i)x_jx_k + y_i (x_j-y_j) x_k + y_i y_j(x_k-y_k).
\]

When $\kz=0$, it is straightforward to check that $\calN(\tU)$ as written in \eqref{calN.form} does not contain any terms involving components of  $\pp_t\tU$, so both
inequalities of \eqref{aux.props2} continue to hold, by the same arguments.

{\bf Step 4: proof of \eqref{aux.suffices}}

Introducing the notation
\[
\De_{\alpha} := \pp_\alpha - i \aAalpha, \qquad
\tF_{\alpha\beta} := \pp_\alpha \tA_\beta - \pp_\beta \tA_\alpha,\qquad
F^{ap}_{\ep,\alpha\beta} := \pp_\alpha \aAbeta - \pp_\beta \aAalpha,
\]
we find after a calculation that
\beq\label{Sprime}
\begin{aligned}
S'_\phi[\aU](\tU)
&= 
(\kz \De_0\De_0 +\ep\ko \De_0)\tPhi-\De_j \De_j  \tPhi  + 2 i(-\kz \tA_0 \De_0\aPhi + \tA_j \De_j\aPhi) \\
&\qquad + i \aPhi( -(\kz\pp_0+\ep\ko)\tA_0 +\pp_j\tA_j)
+ \frac 12 (|\aPhi|^2-1)\tPhi + ( \aPhi, \tPhi) \aPhi
\\
S'_j[\aU](\tU)
&=
(\kz\pp_0+\ep\ko)\tF_{0j} -\pp_k \tF_{kj}
-( i\aPhi, \De_j \tPhi) 
-( i\tPhi, \De_j \aPhi) 
+ \tA_j |\aPhi|^2\\
S_0'[\aU](\tU)
&=
-\pp_k \tF_{k0}
-( i\aPhi, \De_0\tPhi) 
-( i\tPhi, \De_0 \aPhi) 
+ \tA_0 |\aPhi|^2
\end{aligned}
\eeq
(Recall that repeated indices $j,k$ are summed from $1$ to $d$.)
The terms $S'_\phi[\aU](\tU)$ and $S'_a[\aU](\tU), a=1,2$ contain the degenerate operator
$\calL[\aU]$ defined in \eqref{calL.def}. 
The gauge condition \eqref{dynamic.gauge}
will effectively replace $\calL[\au]$ by the coercive operator $L[\au]$ from \eqref{eq33.14}.
To start, we compute
\begin{equation}\label{linearized.ops}
\begin{aligned}
S'_\phi[\aU](\tU) - i\aPhi G(\tU)
&= 
(\kz \De_0\De_0 +\ep\ko \De_0)\tPhi-\De_j \De_j  \tPhi  + \frac 12 (3|\aPhi|^2-1)\tPhi
\\
&\qquad\qquad
+ 2 i(-\kz \tA_0 \De_0\aPhi + \tA_j \De_j\aPhi)
\\
S'_j[\aU](\tU) - \pp_j G(\tU)&=
(\kz\pp_0\pp_0 +\ep\ko\pp_0)\tA_j - \pp_k\pp_k \tA_j
-2( i\tPhi, \De_j \aPhi) 
+ \tA_j |\aPhi|^2\\
S'_0[\aU](\tU)- \pp_0 G(\tU)&=
(\kz\pp_0\pp_0 +\ep\ko\pp_0)\tA_0-\pp_k\pp_k \tA_0
-2( i\tPhi, \De_0 \aPhi) 
+ \tA_0 |\aPhi|^2.
\end{aligned}
\end{equation}
We now consider various terms on the right-hand side, with the aim of eventually rewriting everything as in \eqref{aux.suffices}.

\bigskip

{\bf Step 4a}. 
First,
\[
\De_0 \tPhi = \pp_0\tPhi - i \aAzero \tPhi
\]
and 
\[
\De_0\De_0 \tPhi = \pp_0\pp_0 \tPhi -2i \aAzero\pp_0\tPhi - i (\pp_0 \aAzero) \tPhi - (\aAzero)^2\tPhi.
\]
Using these we find that
\begin{equation}
\label{step3a}
\begin{aligned}
(\kz \De_0\De_0 +\ep\ko \De_0)\tPhi - 2 i\kz \tA_0 \De_0\aPhi
&= 
(\kz \pp_0^2 + \ep\ko \pp_0) \tPhi 
\\
&\qquad - 2i\kz(\tA_0 \De_0\aPhi+ \aAzero \pp_0\tPhi) \\
&\qquad 
+ P_{2a,\phi} \tPhi 
\end{aligned}
\end{equation}
where $P_{2a,\phi}= -\kz [i\pp_0 \aAzero + \aAzero^2] + i\ep\ko \aAzero$. Recalling that $\aU = U_{j,\ep}$, it then follows from \eqref{good2} that
\[
\| P_{2a,\phi}\|_{L^\infty_{T_0} W^{2\ell,\infty}(\R^d)} \le C \ep^2.
\]
Indeed, the choices in \eqref{fix.parameters} are made exactly to guarantee this and similar inequalities below.

{\bf Step 4b}.
Next, splitting the remaining terms in
$S'_\phi[\aU](\tU) - i\aPhi G(\tU)$ into sums over $a = 1,2$ and $\bar a = 3, \ldots, d$ and recalling the definition of $L[\au]$ from \eqref{eq33.14}, we find that
\[
    -\De_j \De_j  \tPhi  + \frac 12 (3|\aPhi|^2-1)\tPhi
    +2i \tA_j \De_j\aPhi
    =
    L_\phi[\au]\tu 
    -\De_{\bar a } \De_{\bar a} \tPhi + 2i  \tA_{\bar a} \De_{\bar a}\aPhi.
\]
Computing as in Step 3a, we find that
\[
 -\De_{\bar a } \De_{\bar a} \tPhi + 2i  \tA_{\bar a} \De_{\bar a}\aPhi
 = -\pp_{\bar a}\pp_{\bar a}\tPhi +2i(
 \tA_{\bar a} \De_{\bar a}\aPhi+ \aAbara \pp_{\bar a}\tPhi) + P_{2b,\phi}\tPhi 
\]
where $P_{2b,\phi} = - [ i \pp_{\bar a} \aAbara + \aAbara\aAbara]$. It again follows from \eqref{good2}
that 
\[
\| P_{2b,\phi}\|_{L^\infty_{T_0} W^{2\ell,\infty}(\R^d)}
\le C \ep^2.
\]

Next, we claim that 
\[
L_\phi[\au](\tu) = L_\phi[q_\ep](\tu) + P_{2c,\phi} \tu + P_{2d,\phi} \nabla \tu
\]
where 
\[
\| P_{2c,\phi} \|_{L^\infty_{T_0} W^{2\ell,\infty}} \le C\ep^2, \qquad
\| P_{2d,\phi} \|_{L^\infty_{T_0} W^{2\ell,\infty}} \le
C\ep^2
\]
Indeed, recalling from Proposition \ref{prop:approx.sol} that 
\[
\au = u_{m,\ep} = \modu(q_\ep) + \ep^2 u^\ep_m \quad\mbox{ with $u^\ep_m$ satisfying \eqref{good3}}
\]
the above estimate holds because the coefficients of the first-order operator $L_\phi[u]$ depend smoothly, indeed polynomially, on components of $u$.

Combining this with \eqref{step3a}, we find that
\begin{equation}\label{step3b}
\begin{aligned}
    &S'_\phi[\aU](\tU) - i\aPhi G(\tU) \\
    &\qquad    =(\kz \pp_0\pp_0 +\ep\ko \pp_0)\tPhi  - \pp_{\bar a }\pp_{\bar a}\tPhi + L_\phi[q_\ep]\tu\\
    &\qquad \qquad\qquad+ 2i[ ( \tA_{\bar a} \De_{\bar a}\aPhi+ \aAbara \pp_{\bar a}\tPhi)
    - \kz(\tA_0 \De_0\aPhi+ \aAzero \pp_0\tPhi)] - i \ep\ko \aAzero\tPhi\\
    &\qquad\qquad\qquad  P_{2,\phi} \tilde u
    \end{aligned}
\end{equation}
where
\[
P_{2,\phi}\tU := ( P_{2a,\phi}+P_{2b,\phi}+P_{2c,\phi})\tilde u 
+ P_{2d,\phi}\nabla\tilde u.
\]
This verifies \eqref{P2.est} for $P_{2,\phi}$.

{\bf Step 4c}. Similar but easier arguments show that for $\alpha=0, 3,\ldots, d$ 
\begin{equation}\label{step3c1}
\begin{aligned}
        S'_\alpha[\aU](\tU) - \pp_\alpha G(\tU)&=
    (\kz\pp_0\pp_0 + \ep\ko \pp_0)\tA_\alpha - \pp_{\bar a}\pp_{\bar a}\tA_\alpha + (-\pp_a\pp_a +|\mphi(q_\ep)|^2)\tA_\alpha \\
    &\qquad\qquad  - 2(i\tPhi, \De_\alpha \aPhi) + P_{2,\alpha} \tU
\end{aligned}
\end{equation}
where
\[
P_{2,\alpha}\tU := (|\aPhi|^2 - |\mphi(q_\ep)|^2)\tA_\alpha, \qquad
\|P_{2,\alpha}\|_{L^\infty_{T_0}W^{2\ell,\infty}} \le C\ep^2
\]
For $a = 1,2$, 
we start as in \eqref{step3c1}, then
use the definition \eqref{eq33.14} of $L[u]$
to obtain
\begin{equation}\label{step3c2} 
    S'_a[\aU](\tU) - \pp_a G(\tU) =
    (\kz\pp_0\pp_0 + \ep\ko \pp_0)\tA_a - \pp_{\bar a}\pp_{\bar a}\tA_a- L_a[q_\ep]\tu + P_{2,a}\tU.
\end{equation}
where
\[
P_{2,a}\tU = (|\aPhi|^2 - \mphi(q_\ep)|^2)\tA_a - 2(i\tPhi, \De_a \aPhi - (\pp_a-i\mA(q_\ep)_a)\mphi(q_\ep)). 
\]
And again it follows from \eqref{Uep.size} that
\[
\| P_{2,a} \|_{L^\infty_{T_0} W^{2\ell,\infty}(\R^d)}
\le
C\ep^2.
\]
The remaining conclusions of the proposition follow by comparing
\eqref{step3b}, \eqref{step3c1}, and \eqref{step3c2} with \eqref{aux.system} -- \eqref{P2.est}.
\end{proof}

\section{linear estimates for the modified system}\label{sec:linests}

In this section, $q: \R^{d-2}\times [0,T) \to M_N$ continues to denote a fixed solution of \eqref{dwm}
satisfying \eqref{q.compact0}, \eqref{q.decay} and with $\Theta<\infty$, see \eqref{Theta.def0}.
We continue to assume that \eqref{fix.parameters} holds with some $\ell$ such that $2\ell>\frac d2+2$.

We consider the linear system
\begin{equation}\label{lin.system}
\kz \pp_0^2\tU + \ep\ko \pp_0\tU + \calM\tU + P\tU = \eta  
\end{equation}
where the operators $\calM$ and $P$, encoding properties of $q$, are as described in Proposition \ref{prop:auxsystem},
and $\eta: [0, \frac {T_0}\ep)\times \R^d\to \C\times \R^{d+1}$ is a given function,
with components that we will as usual write as
\[
\eta =\begin{pmatrix} \eta_u \\ \eta_{\bar a}\\ \eta_0 \end{pmatrix} \quad\mbox{ where }\eta_u = \binom {\eta_\phi}{\eta_a}.
\]
Equation \eqref{lin.system} is defined for functions $\tU:\R^d\times [0,T_0/\ep)\to \C\times \R^{d+1}$, where $T_0\le T$ is the time from the construction
of the approximate solution.
Existence of solutions of \eqref{lin.system} is standard.

For $\tU = (\tu, \tA_{\bar a}, \tA_0)^T$ solving \eqref{lin.system},
we define
\beq\label{cmu.def}
c_\mu(x_{\bar a}, t) 
:= ( \tu , n_\mu )_{L^2(\R^2_{(x_{\bar a},t)})}
:= \int_{\R^2}  \tu(x_a, x_{\bar a}, t) \cdot n_\mu(x ; q_\ep(x_{\bar a}, t))\,dx_a
\eeq
where $\R^2_{(x_{\bar a},t)} := \R^2\times \{ (x_{\bar a},t)\}$ and  $\{n_\mu\}_{\mu=1}^{2N}$ is the orthonormal basis introduced in Definition \eqref{def:nmu}. 

We will prove

\begin{theorem}\label{thm:linearest}
Assume that
\[
\eta \in L^2_{T_0/\ep}( H^{2\ell}(\R^d ; \C\times \R^{d+1})),
\]
and let $\tilde U:\R^d\times (0,T_0/\ep)\to \C\times \R^{d+1}$ solves \eqref{lin.system}
with initial data such that
\beq\label{cmu.initialdata}
c_\mu(x_{\bar a}, 0) = \kz \pp_t c_\mu(x_{\bar a}, 0) = 0.
\eeq
Then there exists a constant $C>0$, independent of $\ep$ and $\eta$, such that 
\begin{align*}
    \kz \| \pp_t \tU(\cdot, t)\|_{H^{2\ell}}^2 + \|\tU(\cdot, t)\|_{H^{2\ell+1}}^2
&\le Ce^{Ct/\ep} \left( \kz \| \pp_t \tU(\cdot, 0)\|_{H^{2\ell}}^2 + \|\tU(\cdot, 0)\|_{H^{2\ell+1}}^2\right) \\
&\hspace{4em}+ \frac C{\ep^3} \int_0^t e^{C \ep (t-s)}\| \eta(s, \cdot)\|_{H^{2\ell}}^2 \, ds
\end{align*}
for every $t\in (0,T_0/\ep)$.
\end{theorem}

The Theorem will follow rather directly from 3 propositions that we will state below, after introducing some notation.

Next, for $\ell\ge 0$ we define
\begin{align*}
    Q_\ell^\perp(t) &:=  \int_{\R^d}(\calM^\ell\tU(\cdot, t))\cdot (\calM^{\ell+1}\tU (\cdot, t)) \\
    Q_\ell^\top (t) &:= \int_{\R^{d-2}} \sum_\mu |\Delta^\ell c_\mu(\cdot, t)|^2 \\
    Q_\ell(t) &:= \kz \|\pp_t (\calM^\ell\tU) (\cdot, t)\|_{L^2(\R^d)}^2 + Q_\ell^\perp(t) + Q_\ell^\top(t)
\end{align*}
and 
\begin{equation}\label{calQ.def}
\calQ_\ell(t) :=   Q_\ell(t) + Q_0(t).    
\end{equation}

Several times in the proof of the theorem we will establish estimates by differentiating \eqref{lin.system} up to $2\ell$ times
with respect to the spatial variables. These computations are all jutified, since the coefficients appearing in the definitions of $\calM$ and $P$ belong to $W^{2\ell,\infty}$, see \eqref{P2.est}, \eqref{P1.est} and \eqref{calM.est}.

The first proposition, whose proof is deferred to Section \ref{sec:Qell}, implies that
it suffices to estimate $\calQ_\ell(t)$ for $t\in [0, T/\ep).$

\begin{proposition}\label{prop:ellenergy}
There exist positive constants $c<C$ such that
\begin{equation}\label{ell.energy}
c \calQ_\ell(t) \le C\left( \kz \| \pp_t \tU(\cdot, t) \|_{H^{2\ell}(\R^{d})}^2
+ \| \tU(\cdot, t)\|_{H^{2\ell+1}(\R^{d-2})}^2 \right) \le C\calQ_\ell(t).
\end{equation}
for every $t\in [0, T_0/\ep)$.    
\end{proposition}

We will first control the part of $Q_\ell$ that, roughly speaking, is orthogonal to the kernel of $\calM$.

\begin{proposition}
        \label{prop:Qperp}
    There exists a constant $C$ such that 
    \begin{equation}\label{Qellperp.est}
            \frac 12 \frac d{dt}\left( \kz \|\pp_t \calM^\ell \tU\|_{L^2}^2 + Q_\ell^\perp(t) \right) \le C \ep \calQ_\ell(t)  + \frac C  \ep \| \eta(\cdot, t)\|_{H^{2\ell}(\R^d)}^2 
        \end{equation}
    for all $t\in [0,T_0/\ep)$.
\end{proposition}

This will be proved in Section \ref{sec:Qperp}.

The other ingredient in the proof is provided by the next Proposition, which is the most subtle part of the proof of Theorem \ref{thm:linearest}. It is proved in Section \ref{sec:Qtop}.

\begin{proposition}
    \label{prop:Qtop}
    There exists a constant $C$ such that
    \begin{equation}\label{Qelltop.est}
    Q^\top_\ell(t) \le C \kz^2\calQ_\ell(0)+
    C \ep \int_0^t \calQ_\ell(s)\, ds + \frac C {\ep^3} \int_0^t  \|\eta (s, \cdot)\|_{H^{2\ell}(\R^d)}^2\, ds 
    \end{equation}
    for all $t\in [0,T_0/\ep)$. 
    \end{proposition}

We now show that as remarked above, these three propositions easily imply the conclusion of Theorem \ref{thm:linearest}.

\begin{proof}[Proof of Theorem \ref{thm:linearest}]
We integrate \eqref{Qellperp.est} from $0$ to $t$, add to \eqref{Qelltop.est}, and add the result to the same inequality for $\ell=0$ to obtain
\[
    \calQ_\ell(t) \le C\calQ_\ell(0) + 
    C \ep \int_0^t \calQ_\ell(s)\, ds + \frac C {\ep^3} \int_0^t  \|\eta (s, \cdot)\|_{H^{2\ell}(\R^d)}^2\, ds .       
\] 
Then Gr\"onwall's inequality implies that
    \[
    \calQ_\ell(t) \le e^{C\ep t}\calQ_\ell(0) + 
    \frac C {\ep^3}\int_0^t e^{C \ep (t-s)}\| \eta(s, \cdot)\|_{H^{2\ell}(\R^d)}^2 \, ds.
    \]
The Theorem now follows from Proposition \ref{prop:ellenergy}.
\end{proof}

\subsection{Proof of Proposition \ref{prop:Qperp}}\label{sec:Qperp}

We start with a lemma that is exactly the $\ell=0$ case of Proposition \ref{prop:Qperp}.

\begin{lemma}\label{lem:Qperp0}
        There exists a constant $C$,  such that
    \begin{align*}
    \frac 12 \frac d{dt}\left( \kz \|\pp_t \tU\|_{L^2}^2 + Q_0^\perp(t) \right) 
    &\le C \ep (\kz  \|\pp_t\tU\|_{L^2}^2 + \|  \tU(\cdot, t) \|_{H^1}^2) + \frac C{\ep} \|\eta(\cdot, t)\|_{L^2}^2
    \\
    &\le C \ep Q_0(t)  + \frac C  \ep \| \eta(\cdot, t)\|_{L^2(\R^d)}^2
    \end{align*}
    for all $t\in [0,T/\ep)$.
\end{lemma}

\begin{proof}
We compute
\[
\begin{aligned}
   \frac 12 \frac d{dt}\left( \kz \|\pp_t \tU\|_{L^2}^2 + Q_0^\perp(t) \right) 
    &= \int _{\R^d} \kz \tU_t \cdot \tU_{tt} + \tU_t \cdot \calM\tU +  \frac 12 \tU \cdot \calM_t \tU \\
    &= \int_{\R^d}\tU_t\cdot (-\ep \ko \tU_t - P\tU - \eta)  + \tU \cdot \calM_t\tU
\end{aligned}
\]
where $\calM_t$ denotes the operator obtained by differentiating the coefficients of $\calM$ with respect to $t$.
The definition \eqref{M.def} of $\calM$ implies that
\begin{equation}\label{Mt.est}
\int \tU \cdot \calM_t \tU \le C \ep \| \tU\|_{H^1}^2.
\end{equation}
If $\ko>0$ we have the pointwise inequality
\[
\tU_t\cdot (-\ep \ko \tU_t - P\tU - \eta) \le \frac 1{2\ep \ko} (|P\tU|^2 + |\eta|^2) .
\]
Integrating and using the estimate  $\|P\tU\|\le C\ep \|\tU\|_{H^1}$ from Proposition \ref{prop:auxsystem}, as well as the $\ell =0$ case of \eqref{ell.energy},
we find that
\[
\begin{aligned}
\frac d{dt}\left( \kz \|\pp_t \tU\|_{L^2}^2 + Q_0^\perp(t) \right) 
&\le C \ep \|  \tU(\cdot, t) \|_{H^1}^2 + \frac C{\ep} \|\eta(\cdot, t)\|_{L^2}^2 \\
&\overset{\eqref{ell.energy}}\le C \ep \calQ_0(t) +\frac C{\ep} \|\eta(\cdot, t)\|_{L^2}^2 .    
\end{aligned}
\]
On the other hand, if $\ko=0$ then $\kz>0$, so we may compute
\[
\begin{aligned}
    \frac d{dt}\left( \kz \|\pp_t \tU\|_{L^2}^2 + Q_0^\perp(t) \right)   
    &\overset{\eqref{Mt.est}}\le \int_{\R^d} -\tU_t\cdot  ( P\tU + \eta)  + C \ep \| \tU\|_{H^1}^2 \\ 
    &\le \int \frac \ep 2 |\tU_t|^2 + \frac 1{2\ep} (|P\tU|^2  + |\eta|^2) + C \ep \| \tU\|_{H^1}^2\\
    &\le C \ep ( \|\tU_t\|_{L^2}^2 + \|  \tU(\cdot, t) \|_{H^1}^2) + \frac C{\ep} \|\eta(\cdot, t)\|_{L^2}^2\\ &\overset{\eqref{ell.energy}}\le C \ep Q(t)+ \frac C{\ep} \|\eta(\cdot, t)\|_{L^2}^2 .
\end{aligned}
\]
\end{proof}

We can now complete the

\begin{proof}[Proof of Proposition \ref{prop:Qperp}]
By applying the operator $\calM^\ell$ to equation \eqref{lin.system}, we find that $\tU_\ell := \calM^\ell \tU$ satisfies
    \begin{equation}\label{tUell.eqn}
    \kz \pp_0^2\tU_\ell + \ep\ko \pp_0\tU_\ell + \calM\tU_\ell + P\tU_\ell = \eta_\ell  
    \end{equation}
    where
    \[
    \eta_\ell = \kz [\pp_0^2, \calM^\ell]\tU + \ep\ko [\pp_0, \calM^\ell]\tU + [P,\calM^\ell]\tU + \calM^\ell \eta
    \]
    and $[A,B]$ denotes the commutator $AB-BA$. These computations and similar ones below are justified because
    all coefficients in \eqref{lin.system} are bounded in $L^\infty_{T_0/\ep}W^{2\ell, \infty}(\R^d)$, see 
    \eqref{P2.est}, \eqref{P1.est}, \eqref{calM.est}.
    
    In view of \eqref{tUell.eqn} we can apply the previous lemma to $\tU_\ell$ to find that
    \begin{align*}
        &\frac d{dt} \left( \kz \|\pp_t \calM^\ell \tU\|_{L^2}^2 + \int \calM^{\ell+1}\tU\cdot\calM^\ell\tU \right)\\
        &\hspace{5em}
        \le C\ep \left(\kz  \| \pp_t\calM^\ell \tU\|_{L^2}^2 + \|\calM^\ell \tU\|_{H^1}^2\right) + \frac C \ep \|\eta_\ell \|_{L^2}^2.
    \end{align*}
    Thus Proposition \ref{prop:ellenergy} and the definitions implies that
    \[
    \frac d{dt} \left( \kz \|\pp_t \calM^\ell \tU\|_{L^2}^2 + Q_\ell^\perp(t)\right)\le C\ep \calQ_\ell(t)  + \frac C \ep \|\eta_\ell \|_{L^2}^2 \qquad\mbox{ at every }t\in (0,T_0/\ep).
    \]
    To complete the proof, it therefore suffices to prove that
    \begin{equation}\label{etaell}
    \| \eta_\ell\|_{L^2}^2 \le C \| \eta\|_{H^{2\ell}}^2 + C \ep^2 \calQ_\ell(t) 
    \end{equation}
    We consider each of the terms in $\eta_\ell$ in turn. Due to the appearance of time derivatives in the estimates, we will sometimes write $\calM(t)$ to explicitly record the dependence of $\calM$ on $t$.
    
    First, from the definition \eqref{M.def} of $\calM$, we see that
    \[
    \calM(t) = -\Delta + L_{q_\ep(t), 1}
    \]
    where $L_{q_\ep(t), 1}$ is a first-order linear differential operator whose coefficients depend on $q_\ep(\cdot, t) = q(\ep \cdot ,\ep t )$. 
    Thus 
    \[
    \calM(t)^\ell = (-\Delta)^\ell + L_{q_\ep(t), \ell}
    \]
    where $L_{q_\ep, \ell}$ is a linear differential operator of order at most $2\ell -1$ whose coefficients depend smoothly on $q_\ep$. Since $(-\Delta)^\ell$ clearly commutes with $\pp_t^2$, 
    \[
    [\pp_0^2, \calM(t)^\ell]\tU = [\pp_0^2, L_{q_\ep(t), \ell}]\tU
    = (\pp_t^2  L_{q_\ep(t), \ell})\tU + 2  (\pp_t L_{q_\ep(t), \ell})\pp_t\tU
    \]
    where $(\pp_t^k L_{q_\ep(t), \ell})$ denotes the differential operator, still of order at most $2\ell-1$, obtained by differentiating the coefficients of $L_{q_\ep(t), \ell}$ as indicated. From the fact that $\pp_t^k q_\ep(x_{\bar a}, t) = \ep^k (\pp_t^k q)\ep x_{\bar a}, \ep t)$ as well as the smooth dependence of the coefficients of $L_{q_\ep(t), \ell}$ on $q_\ep$, we infer that
    \[
     \left\|\kz [\pp_0^2, \calM(t)^\ell]\tU \right\|_{L^2}^2 \le C\kz\left( \ep^2\| \pp_t\tU\|_{H^{2\ell -1}}^2
     + \ep^4 \| \tU\|_{H^{2\ell -1}}^2\right).
    \]
    Very similar arguments show that 
    \[
    \left\|\ep\ko [\pp_0, \calM(t)^\ell]\tU \right\|_{L^2}^2 \le C \ep^4 \| \tU\|_{H^{2\ell -1}}^2.
    \]
It follows from \eqref{P2.est}, \eqref{P1.est} that
    \[
    \left \| [P, \calM^\ell] \tU \right\|_{L^2}^2  \le  C \ep^2 \| \tU\|_{H^{2\ell+1}}^2.
    \]
    Finally, it is clear that
    \[
    \left\| \calM^\ell \eta \right\|_{L^2}^2 \le C \| \eta \|_{H^{2\ell}}^2\ .
    \]
    Claim \eqref{etaell} follows by adding up the previous inequalities and using \eqref{ell.energy} to bound the right-hand sides in terms of $\calQ_\ell(t)$.

\end{proof}

\subsection{Proof of Proposition \ref{prop:Qtop}} \label{sec:Qtop}

For this proof we will need to write the PDE satisfied by $c_\mu$ in a way that carefully exploits the structure of \eqref{lin.system}, in particular that of the linear operator $P$. The following decomposition of $\tA_\alpha$ for $\alpha = 0, 3, \ldots, d$ will play a crucial role in our arguments.

\begin{lemma}\label{lem:tA_split}
Assume that
\[
 \| \tU\|_{L^\infty_{T_0}H^{2\ell+1} } + \kz \| \pp_t \tU\|_{L^\infty_{T_0}H^{2\ell} } <\infty.
\]
    For $\alpha = 0,3,\ldots, d$ exist functions $f_\alpha$ and $\chi_\alpha$ such that
    \begin{equation}\label{tA.split}
        \tA_\alpha = f_\alpha + \pp_0\chi_\alpha
    \end{equation}
    and
    \begin{equation}\label{chi.and.f}
    \begin{aligned}
       \|\chi_\alpha(\cdot, t) \|_{H^{2\ell+1}} &\le C\kz \|\pp_0 \tA_\alpha(\cdot, t)\|_{H^{2\ell}} + C \ep \ko\| \tA_\alpha(\cdot, t)\|_{H^{2\ell}} 
       \\ 
       \|f_\alpha(\cdot, t) \|_{H^{2\ell+1}} &\le C\ep \left( \| \tU(\cdot, t)\|_{H^{2\ell+1}}  + \kz \|\pp_0 \tA_\alpha(\cdot, t)\|_{H^{2\ell}} \right)+ \|\eta(\cdot, t)\|_{H^{2\ell}}
    \end{aligned}
        \end{equation}
    for all $t\in [0,T/\ep)$.
\end{lemma}

\begin{proof}
We first note that if $\psi\in H^1(\R^d)$ is a weak solution of 
\begin{equation}\label{psig}
(-\Delta + |\mphi(q_\ep)|^2)\psi = g\in H^{k}(\R^d),
\end{equation}
then $\psi\in H^{k+1}(\R^d)$ and
\begin{equation}\label{standard}
\| \psi\|_{H^{k+1}(\R^d)} \le C \| g \|_{H^{k}(\R^d)}.
\end{equation}
Since $|\mphi(q_\ep)|^2$ is smooth, standard elliptic regularity implies that $\psi\in H^{k+2}_{loc}$. The
$k=0$ case of \eqref{standard} then follows directly from Lemma \ref{basic.elliptic}. The general case follows by induction.
Indeed, if $\psi$ solves \eqref{psig}, then for any $i\in \{1,\ldots, d\}$ we see by differentiating \eqref{psig} that $\psi_i := \pp_i \psi$ solves
\[
(-\Delta + |\mphi(q_\ep)|^2)\psi_i = \pp_i g - (\pp_i |\mphi(q_\ep)|^2) \psi  \in H^{k-1}(\R^d).
\]
So the induction hypothesis implies that 
\[
\| \pp_i\psi\|_{H^{k}}\le C(\|\pp_i g\|_{H^{k-1}} + \|\psi\|_{H^{k-1}}) \le C \| g\|_{H^k}
\]
for every $i$, which immediately implies \eqref{standard}.

Next, we define $\chi_\alpha$ as the solution of the equation
\[
(-\Delta + |\mphi(q_\ep)|^2)\chi_\alpha(\cdot, t) = -\kz \pp_0 \tA_\alpha(\cdot, t) - \ep \ko \tA_\alpha(\cdot, t).
\]
Then the first estimate in \eqref{chi.and.f} follows directly from \eqref{standard}.

We are forced by \eqref{tA.split} to define
$f_\alpha = \tA_\alpha - \pp_0\chi_\alpha$.
Writing $\mphi$ for $\mphi(q_\ep)$  and recalling the definition \eqref{M.def} of the operator $\calM$ in \eqref{lin.system}, we find that $f_\alpha$ solves
\begin{align}
    (-\Delta + |\mphi(q_\ep)|^2) f_\alpha 
    &= (-\Delta +  |\mphi|^2) (\tA_\alpha - \pp_0\chi_\alpha) \nonumber\\
    &= (-\Delta +  |\mphi|^2) \tA_\alpha - \pp_0[ (-\Delta +  |\mphi|^2)\chi_\alpha)] + \pp_0 |\mphi|^2 \chi_\alpha \nonumber\\
    &= (\kz\pp_0^2 + \ep \ko \pp_0 - \Delta +|\mphi|^2) \tA_\alpha+ \pp_0 |\mphi|^2 \chi_\alpha
    \nonumber\\
    &\overset{\eqref{lin.system}}=- P_\alpha\tU +\eta_\alpha + \pp_0 |\mphi|^2 \chi_\alpha.
    \label{falpha.eqn}
\end{align}
Recall from \eqref{good2} that 
\[
\kz \| \pp_0 |\mphi|^2\|_{W^{2\ell, \infty}(\R^d)} \le C \ep
\]
Note also that $H^{2\ell}(\R^d)$ is an algebra, as we have assumed that $2\ell>\frac d2$. It thus follows from \eqref{standard} and properties of the operator $P$ in Proposition \ref{prop:auxsystem} and our earlier estimate of $\chi_\alpha$ that
\[
\begin{aligned}    
\|f_\alpha\|_{H^{2\ell+1}(\R^d)} 
&\le
\| P_\alpha \tU\|_{H^{2\ell}} + \|\eta_\alpha\|_{H^{2\ell}} + \| \pp_0|\mphi(q_\ep)|^2 \chi_\alpha\|_{H^{2\ell}}
\\
&\le C \ep \left( \| \tU\|_{H^{2\ell+1}} + \kz \| \pp_t\tU \|_{H^{2\ell}}\right) + \|\eta_\alpha\|_{H^{2\ell}}
\end{aligned}
\]
at time $t$.
\end{proof}

\begin{lemma}
    \label{lem:cmupde2}
    $c_\mu$ satisfies the initial value problem
    \begin{equation}
        \label{cmu.pde2}
        \begin{aligned}
        (\kz \pp_{tt} +\ep \ko \pp_t - \Delta_{\bar a}) c_\mu &=
        \pp_0 h_0 + \nabla_{\bar a} \cdot h_{\bar a} + h, \\
        c_\mu = \kz \pp_t c_\mu = 0\qquad&\mbox{ at }t=0
        \end{aligned}
        \end{equation}
    where the functions on the right-hand side satisfy the estimate 
    \begin{equation}
        \label{h.estimates}
        \begin{aligned}    
        \| h_0(\cdot, t)\|_{H^{2\ell}(\R^{d-2})} &\le  C \ep \kz \left( \| \tU(\cdot, t)\|_{H^{2\ell+1}(\R^d)} + \kz\| \pp_t \tU(\cdot, t)\|_{H^{2\ell}(\R^{d})} \right)
        \\
        \| h_{\bar a}(\cdot, t)\|_{H^{2\ell}(\R^{d-2})} &\le C \ep\| \tU(\cdot, t)\|_{H^{2\ell}(\R^d)}
        \\
        \| h(\cdot, t)\|_{H^{2\ell}(\R^{d-2})} &\le C \ep^2 \left (\| \tU(\cdot, t)\|_{H^{2\ell+1}(\R^d)} + \kz \| \pp_t\tU(\cdot, t)\|_{H^{2\ell}(\R^{d})} \right) \\
        &\hspace{8em} + C\|\eta(\cdot, t)\|_{H^{2\ell}(\R^{d})}.
        \end{aligned}
    \end{equation}
    \end{lemma}

\begin{proof}
We will write $( \, \cdot\, , \, \cdot\, )$ as an abbreviation for $( \, \cdot\, , \, \cdot\, )_{L^2(\R^2)}$.

The initial conditions for $c_\mu$ have been assumed in \eqref{cmu.initialdata}.

Next, note that
    \[
    \begin{aligned}
    \pp_t c_\mu &= ( \pp_t \tu , n_\mu ) + ( \tu, \pp_t n_\mu)\\
    \pp_{tt} c_\mu &=  ( \pp_{tt} \tu , n_\mu ) + 2 \pp_t( \tu, \pp_t n_\mu) - ( \tu, \pp_{tt} n_\mu).
    \end{aligned}
    \]
    and similarly $\pp_{\bar a \bar a} c_\mu$ for $\bar a = 3,\ldots, d$.
    Next, note that because $L[\modu(q)]$ is self-adjoint and $L[\modu(q)] n_\mu(q) = 0$ (as follows from Theorem \ref{thm:hessian} and Definition \ref{def:nmu}),
    \[
    ( \calM_u\tU, n_\mu) = ( -\Delta_{\bar a} \tu + L[\modu(q)]\tu, n_\mu) 
    = -\Delta_{\bar a} c_\mu +   2\nabla_{\bar a}\cdot( \tu, \nabla_{\bar a}n_\mu  ) - ( \tu, \Delta_{\bar a}n_\mu  ) 
    \]
    where we write $\calM_u\tU$ to denote the relevant rows of $\calM\tU$ (and similarly $P_u\tU$ below). 
    We now take the $L^2(\R^2)$ inner product of the relevant components of \eqref{lin.system} with $n_\mu$ and use the above computations to rewrite, leading to
    \[
    \begin{aligned}
    (\kz\pp_{tt} + \ep\ko \pp_t - \Delta_{\bar a})c_\mu &= 
    2\kz\pp_t (  \tu, \pp_t n_\mu) - 
    2\pp_{\bar a}(  \tu, \pp_{\bar a}  n_\mu) \\
    & \qquad -  (  \tu, (\kz \pp_t^2 - \ep\ko \pp_t -\Delta_{\bar a })n_\mu) 
    -( P_u \tU, n_\mu)+ ( \eta_u, n_\mu) . 
    \end{aligned}
    \]
    We next write $P = P_1 + P_2$ as in Proposition \ref{prop:auxsystem} (and thus $P_u = P_{1u} + P_{2u}$) and we set 
    \[
    h_{0,i} := 2\kz ( \tu, \pp_tn_\mu ), \qquad
    h_{\bar a,i} := -2( \tu, \pp_{\bar a}n_\mu ), \qquad
    \]
    and
    \[
    h_{,i} := -  (  \tu, (\kz \pp_t^2 - \ep\ko \pp_t -\Delta_{\bar a })n_\mu)   -( P_{2u} \tU, n_\mu) + ( \eta_u, n_\mu),
    \]
    so that
    \[
    (\kz\pp_{tt} + \ep\ko \pp_t - \Delta_{\bar a})c_\mu = 
    \pp_0 h_{0,i} + \pp_{\bar a}h_{\bar a, i} + h_{,i} -( P_{1u}\tU, n_\mu)\ .
    \]
    Below we will look more carefully at $( P_{1u}\tU, n_\mu).$ First, however, we note that
    \begin{equation}\label{hi.estimates}
    \begin{aligned}
    \|h_{0, i}(\cdot, t)\|_{H^{2\ell}(\R^{d-2})} &\le C\ep  \kz \|\tu(\cdot, t)\|_{H^{2\ell}(\R^d)}  \\ 
    \|h_{\bar a, i}(\cdot, t)\|_{H^{2\ell}(\R^{d-2})} &\le C\ep \|\tu(\cdot, t)\|_{H^{2\ell}(\R^d)}\qquad\qquad \mbox{ for }\bar a=3,\ldots, d, \\ 
    \|h_{,i}(\cdot, t)\|_{H^{2\ell}(\R^{d-2})} &\le C \ep^2 \|\tu(\cdot, t)\|_{H^{2\ell+1}(\R^d)} + C \|\eta(\cdot, t)\|_{H^{2\ell}(\R^d)}.
    \end{aligned}
    \end{equation}
    Indeed, these follow from the estimate for $P_2$ in Proposition \ref{prop:auxsystem}, see \eqref{P2.est}, and the fact that $n_\mu = n_\mu(x_a; q_\ep(x_{\bar a}, t)) = n_\mu(x_a ; q( \ep x_{\bar a}, \ep t))$.
    
    Writing $n_\mu = (n_{\mu, \phi}, n_{\mu,a})^T$, we use the definition \eqref{P1.def} of $P_1$ and  substitute $\tA_\alpha = f_\alpha+\pp_0\chi_\alpha$ for $\alpha = 0,3, \ldots, d$ to find
    that
    \[
    \begin{aligned}
        ( P_{1u}\tU, n_\mu) &= 
        -\kz(f_0( 2i \De_0\aPhi, n_{\mu,\phi})  + \pp_0\chi_0( 2i \De_0\aPhi, n_{\mu,\phi})) +
        f_{\bar a}( 2i\De_{\bar a}\aPhi, n_{\mu,\phi}) \\
        &\qquad\qquad\qquad  + \pp_0\chi_{\bar a}( 2i\De_{\bar a}\aPhi, n_{\mu,\phi})
        -\kz( 2i\aAzero\pp_0\tPhi,n_{\mu,\phi}) + ( 2i\aAbara\pp_{\bar a}\tPhi,n_{\mu,\phi})  \\
        &=\pp_0 h_{0,ii} + \pp_{\bar a} h_{\bar a, ii} + h_{, ii}
    \end{aligned}
    \]
    for
    \begin{align*}
        h_{0, ii} &:=  -\kz\chi_0( 2i\De_0\aPhi, n_{\mu,\phi}) + 
        \chi_{\bar a}( 2i\De_{\bar a}\aPhi, n_{\mu,\phi})
         -\kz( 2i\aAzero\tPhi,n_{\mu,\phi})\\
        h_{\bar a, ii} &:= ( 2i\aAbara\tPhi,n_{\mu,\phi}) \\
        h_{,ii} &:= 
         -\kz f_0( 2i \De_0\aPhi, n_{\mu,\phi})  + \kz\chi_0\pp_0( 2i \De_0\aPhi, n_{\mu,\phi}) +
        f_{\bar a}( 2i\De_{\bar a}\aPhi, n_{\mu,\phi}) \\
        &\qquad\qquad  -\chi_{\bar a}\pp_0( 2i\De_{\bar a}\aPhi, n_{\mu,\phi})
        +\kz( 2i\pp_0(\aAzero)\tPhi,n_{\mu,\phi}) +( 2i\aAzero\tPhi,\pp_0n_{\mu,\phi}) \\
        &\qquad\qquad- ( 2i\pp_{\bar a}(\aAbara)\tPhi,n_{\mu,\phi}) 
        - ( 2i\aAbara\tPhi,\pp_{\bar a}n_{\mu,\phi})
    \end{align*}
    Then Lemma \ref{lem:tA_split} and properties of the approximate solution imply that for every $t$,
    \begin{equation}
        \label{hii.estimates}
        \begin{aligned}
        \| h_{0,ii}(\cdot, t)\|_{H^{2\ell}(\R^{d-2})} &\le  C \kz\ep \left( \| \tU(\cdot, t)\|_{H^{2\ell+1}(\R^d)} + \kz \| \pp_t \tU(\cdot, t)\|_{H^{2\ell}(\R^d)} \right)
        \\
        \| h_{\bar a,ii}(\cdot, t)\|_{H^{2\ell}(\R^{d-2})} &\le C \ep\| \tU\|_{H^{2\ell}(\R^d)}
        \\
        \| h_{,ii}(\cdot, t)\|_{H^{2\ell}(\R^{d-2})} &\le C \ep^2 \left (\| \tU(\cdot, t)\|_{H^{2\ell+1}(\R^d)} + \kz \| \pp_t\tU(\cdot, t)\|_{H^{2\ell}(\R^d)} \right)\\
        &\hspace{9em}+ C \ep \|\eta(\cdot, t)\|_{H^{2\ell}(\R^d)} \ .
        \end{aligned}
    \end{equation}
It follows from this and \eqref{hi.estimates} that conclusions \eqref{h.estimates}  hold if we define
    \[
    h = h_{,i} + h_{,ii} \ \  \qquad\mbox{and}\ \ \qquad h_\alpha = h_{\alpha,i} + h_{\alpha,ii} \ \ \mbox{ for }\alpha = 0,3,\ldots, d .
    \]
    
    \end{proof}

We now present the

\begin{proof}[proof of Proposition \ref{prop:Qtop}] The desired estimate follows directly from Lemma \ref{lem:cmupde2} and basic estimates for hyperbolic and parabolic equations, recalled in Section  \ref{auxlin} below. We first check this in the case $\kz > 0$. Then, since $c_\mu(\cdot, 0) = \pp_t c_\mu(\cdot, 0) = 0$ by assumption,
Lemma \ref{lem:dwest} implies that for $0< t < T/\ep$, 
\[
\begin{aligned}
Q_\ell^\top(t) 
&\le \| c_\mu(\cdot, t)\|_{L^2(\R^{d-2)}}^2
\\
&\le 
C t^3\int_0^t  \| h(s, \cdot)\|^2_{H^{2\ell}}\,ds  +  t\int_0^t  \left(\| h_{0}(s, \cdot)\|^2_{H^{2\ell}} +  \sum_{\bar a=3}^d
\| h_{\bar a}(s, \cdot)\|^2_{H^{2\ell}} \,\right) ds \\
&\hspace{12em} + C t^2 \| h_0(\cdot, 0)\|_{H^{2\ell}}^2 .
\end{aligned}
\]
Using Lemma \ref{lem:cmupde2} to estimate the right-hand side and Proposition \eqref{prop:ellenergy} to express the result in terms of $\calQ_\ell$ yields \eqref{Qelltop.est}.

When $\kz =0$, the proof is essentially the same, but using Lemma \ref{lem:heatest} instead of Lemma \ref{lem:dwest}
\end{proof}

\subsection{auxiliary linear hyperbolic and parabolic estimates}\label{auxlin}
 
We have used the following basic estimates in the proof of Proposition \ref{prop:Qtop} above.

\begin{lemma}\label{lem:dwest}
Assume that $f:\R^{1+n}\to \R$ solves
\begin{equation}\label{f.eqn2}
	\kz\pp_{tt}f +\ep \ko\pp_t f - \Delta f = \pp_0 h_0 + \cdots + \pp_n h_n  + h
	\end{equation}
    with $\kz>0$ and $\ko \ge 0$, 
	for $h_\alpha, h:\R^{1+n}\to \R$ such that 
	\[
	h\in L^2_t H^{2\ell} , \qquad h_0\in H^1_t H^{2\ell},\qquad  h_j\in L^2_t H^{2\ell+1}\quad\mbox{ for }j=1,\ldots, n.
	\]
	and with initial data $(f, \pp_t f)|_{t=0}\in H^{2\ell+1}\times H^{2\ell}$.
Then there exists a constant $C$ that may depend on $\kz$ but is independent of $\ko\ep >0$ such that 
\begin{equation}\label{wave.est2}
\begin{aligned}
\| f(\cdot, t)\|_{H^{2\ell}}^2  &\le
C t^3\int_0^t  \| h(s, \cdot)\|^2_{H^{2\ell}}\,ds  +  t\int_0^t  \sum_{\alpha=0}^n\| h_\alpha(s, \cdot)\|^2_{H^{2\ell}} \, ds   \\ 
&
\hspace{1em} 
+ C\| f(\cdot, 0)\|_{H^{2\ell}}^2 +C t^2 \left( \| \nabla f(\cdot, 0)\|_{H^{2\ell}}^2 + \| \pp_t f(\cdot, 0)\|_{H^{2\ell}}^2  + \| h_0(\cdot, 0)\|_{H^{2\ell}}^2 \right).
\end{aligned}
\end{equation}
\end{lemma}

\begin{proof}
It suffices to prove the lemma for $\ell=0$, since the general case then follows by differentiating the equation and applying the $\ell=0$ case.

Several times during the proof we will use the fact that for any function $f:\R^{1+n}\to \R$
and  $t>0$,
\begin{equation}\label{basicL22}
\| f(\cdot, t)\|_{L^2}^2 \le  2\left(\| f(\cdot, 0)\|_{L^2}^2 + t \int_0^t \| f_t(s, \cdot) \|_{L^2}^2\, ds\right).
\end{equation}
as long as the right-hand side makes sense  and is bounded (that is, $f(\cdot, 0)\in L^2(\R^n)$, and $f(\cdot, x)\in H^1((0,\infty))$ for {\em a.e.} $x\in \R^n$). Indeed, for almost every $x$ we have
\begin{align*}
f(x,t)^2  =  \left(f(0,x) + \int_0^t f_t(s,x)\, ds \right)^2
&\le 2 f(0,x)^2 + 2\left(\int_0^t f_t(x,s)\,ds\right)^2 
\\ &\le 2 f(0,x)^2 + 2t\int_0^t f_t(x,s)^2\, ds.
\end{align*}
We obtain \eqref{basicL22} by integrating both sides of this inequality.

\medskip

We now consider several cases.

\smallskip

{\bf Case 1}: $h_0 = h_1 = \cdots = h_n = h = 0$. Then basic estimates (that is, \eqref{basic.ee} below, with $h=0$) imply that 
\[
\kz \| f_t(\cdot, t) \|_{L^2}^2 + \| \nabla f(\cdot, t) \|_{L^2}^2  \le 
\kz \| f_t(\cdot, 0) \|_{L^2}^2 + \| \nabla f(\cdot, 0) \|_{L^2}^2 
\]
for all $t$. We deduce \eqref{wave.est2} from this and \eqref{basicL22}.

\smallskip

{\bf Case 2}: $h_0 = \cdots h_n = 0$, and $(f, f_t)|_{t=0} = 0$. Then the equation becomes
\[
\kz \pp_{tt}f + \ep \ko \pp_t f - \Delta f = h, \qquad (f, f_t)|_{t=0} = 0.
\]
We recall the standard argument. Basic energy estimates imply that 
\begin{equation}\label{basic.ee}
\ep \ko  \int f_t^2\,dx + \frac d{dt} \frac 12 \int (\kz f_t^2+|\nabla f|^2) \, dx = \int h f_t \le \left( \int h^2 \ dx \ \int f_t^2\,dx\right)^{1/2}.
\end{equation}
Defining $\eta(t) := \sqrt{ \int \kz f_t(\cdot, t)^2+|\nabla f(\cdot, t)|^2 \, dx}$ and $H(t) := \| h(\cdot, t)\|_{L^2}$, we deduce that
\[
\frac 12 \frac d{dt}\eta^2(t) \le  \frac 1 {\sqrt \kz} \eta(t)H(t)\qquad\mbox { and thus }\eta'(t) \le \frac 1{\sqrt \kz}  H(t).
\]
The assumptions of Case 2 imply that $\eta(0)=0$, so
\begin{equation}\label{case2extra2}
\int (\kz f_t^2+|\nabla f|^2 )\, dx\Big|_t = \eta^2(t) \le\left( \frac 1{\sqrt \kz}\int_0^t \| h(s,\cdot)\|_{L^2}\, ds \right)^2 \le \frac t \kz \int_0^t \|h(s,\cdot)\|_{L^2}^2\, ds.
\end{equation}
Then since we have assumed that $f=0$ when $t=0$,  we can combine this with \eqref{basicL22} to find that
\[
\| f(\cdot, t) \|_{L^2}^2 \le Ct\int_0^t s\int_0^s \|h(r,\cdot)\|_{L^2}^2 \, dr \, ds \le C t^3 \int_0^t \|h(s,\cdot)\|_{L^2}^2 \, ds
\]
where the constant depends on $\kz$ and blows up as $\kz \searrow 0$. Thus \eqref{wave.est2} holds in this case.

\smallskip

{\bf Case 3}: There exists some $j\in \{0,\ldots, n\}$ such that $h_k = 0$ if $k\ne j$, and in addition $h= 0$,  and $(f, f_t)|_{t=0} = 0$. 

Then the equation becomes
\begin{equation}\label{case32}
\kz \pp_{tt}f +\ep \ko \pp_t f - \Delta f = \pp_j h_j , \qquad \qquad (f, f_t)|_{t=0} = 0.
\end{equation}
where here we do {\em not} implicitly sum over $j$.

{\bf Case 3a}. $j=0$.  Then we define
\[
F_0(x,t) = \int_0^t f(s,x)\,ds, \qquad
H_0(x,t) = \int_0^t \pp_t h_0(s,x)\,ds = h_0(x,t) - h_0(0,x).
\]
Clearly
\[
\kz \pp_{tt}F_0 +\ep \ko \pp_t F_0 - \Delta F_0 = H_0 , \qquad \qquad (F_0, \pp_t F_0 )|_{t=0} = (0,0).
\]
So we deduce from \eqref{case2extra2} that
\begin{align*}
\| f(\cdot, t)\|_{L^2}^2  = \| \pp_t F_0(\cdot, t)\|_{L^2}^2 &\le \frac t\kz  \int_0^t \|H_0(s, \cdot)\|_{L^2}^2\, ds \\
& \le \frac{2t}\kz \int_0^t \|h_0(s, \cdot)\|_{L^2}^2\, ds 
+ \frac{2 t^2}{\kz} \| h_0(\cdot, 0)\|_{L^2}^2.
\end{align*}
This establishes \eqref{wave.est2} in this case.

{\bf Case 3b}: $j\in \{1, \ldots, n\}$. We may assume without loss of generality that $j=1$. We may then argue as in Case 3a, defining
\[
F_1(t,x_1,\ldots, x_n) = \int_{-\infty}^{x_1} f(t,s, x_2, \ldots x_n)\, ds, 
\]
and noting that  $\| f(\cdot, t) |_{L^2}^2 = \| \pp_1 F_1(\cdot, t) \|_{L^2}^2 \le \| \nabla F_1(\cdot, t) \|_{L^2}^2$, with  $F_1$ solving $\pp_{tt}F_1 - \Delta F_1 = h_1$ and $(F_1,\pp_t F_1)|_{t=0}=(0,0)$. 

\smallskip

Since \eqref{f.eqn2} is linear, the lemma follows from Cases 1, 2 and 3.
\end{proof}

Our next lemma establishes a similar estimate to that of Lemma \ref{lem:dwest} when $\kz=0$, that is, in the case of a pure heat flow.

\begin{lemma}\label{lem:heatest}
Assume that $f:\R^{1+n}\to \R$ solves
\begin{equation}\label{f.eqn2a}
	\ep \ko\pp_t f - \Delta f = \pp_1 h_1 + \cdots + \pp_n h_n  + h
	\end{equation}
    with $\ko > 0$, 
	for $h_j, h:\R^{1+n}\to \R$ such that 
	\[
	h\in L^2_t H^{2\ell} , \qquad   h_j\in L^2_t H^{2\ell+1}\quad\mbox{ for }j=1,\ldots, n.
	\]
	and with initial data $f(\cdot, t) \in H^{2\ell}$.
Then there exists a constant $C$ that may depend on $\ko$ but is independent of $\ep>0$, such that 
\begin{equation}\label{heat.est2}
\begin{aligned}
\| f(\cdot, t)\|_{H^{2\ell}}^2  &\le
C \frac t{\ep^2}\int_0^t  \| h(s, \cdot)\|^2_{H^{2\ell}}\,ds  +  \frac C \ep\int_0^t  \sum_{j=1}^n\| h_j(s, \cdot)\|^2_{H^{2\ell}} \, ds   \\ 
&
\hspace{10em} 
+ C\| f(\cdot, 0)\|_{H^{2\ell}}^2.
\end{aligned}
\end{equation}
\end{lemma}

\begin{proof}
It suffices to prove the lemma for $\ell=0$, since the general case then follows by differentiating the equation and applying the $\ell=0$ case.

    We consider several cases.

\smallskip

{\bf Case 1}: $ h_1 = \cdots = h_n = h = 0$. Then basic estimates (that is, \eqref{basic.heat} below, with $h=0$) imply that 
\[
\| f(\cdot, t) \|_{L^2}^2  \le 
\| f(\cdot, 0) \|_{L^2}^2  
\]
for all $t$, immediately implying \eqref{heat.est2}.

\smallskip

{\bf Case 2}: $h_0 = \cdots =  h_n = 0$, and $f|_{t=0} = 0$. Then the equation becomes
\[
 \ep \ko \pp_t f - \Delta f = h, \qquad f|_{t=0} = 0.
\]
Multiplying the equation by $f$ and integrating by parts, we obtain
\begin{equation}\label{basic.heat}
\frac 12 \ep \ko \frac d{dt}\int f^2\,dx +  \int |\nabla f|^2\,dx  = \int fh \,dx \le \ \left( \int h^2\,dx \int f^2\,dx
\right)^{1/2}.
\end{equation}
Arguing as in Lemma \ref{lem:dwest}, we find that 
\[
\| f(\cdot, t) \|_{L^2}^2  \le \frac{t}{(\ep \ko)^2}\int_0^t \| h(s, \cdot)\|_{L^2}^2\, ds.
\]

\smallskip

{\bf Case 3}: There exists some $j\in \{1,\ldots, n\}$ such that $h_k = 0$ if $k\ne j$, and in addition $h= 0$,  and $f|_{t=0} = 0$. 

We may assume without loss of generality that $j=1$. Then the equation becomes
\begin{equation}\label{case32b}
\ep \ko \pp_t f - \Delta f = \pp_1 h_1 , \qquad \qquad f|_{t=0} = 0.
\end{equation}

Consider the equation obtained by formally integrating \eqref{case32b} with respect to the $x_1$ variable:
\[
\ep\ko \pp_t F_1 - \Delta F_1 = h_1, \qquad\qquad \left. F_1\right|_{t=0}=0 .
\]
This equation is clearly solvable, and  $\pp_1 F_1$ solves \eqref{case32b} and hence equals $f$.
We multiply both sides by $\pp_t F_1$ and integrate by parts to find that
\[
\begin{aligned}
    \ep \ko \int(\pp_t F_1)^2\,dx  + \frac 12 \frac d{dt}\int|\nabla F_1|^2\,dx 
    &= \int  h_1 \pp_t F_1\, dx \\
    &\le \frac 1{4\ep \ko} \int h_1^2\, dx +  \int\ep \ko (\pp_t F_1)^2\,dx.
\end{aligned}
\]
Discarding $\ep \ko \int(\pp_t F_1)^2\,dx $ on both sides and integrating with respect to $t$,
we obtain
\[
\left. \int f^2 \,dx \right|_{t} \le \left. \int|\nabla F_1|^2\,dx \right|_{t} \le \frac C\ep \int_0^t \| h_1(s, \cdot)\|^2\,ds.
\]

\smallskip

Since \eqref{f.eqn2a} is linear, the lemma follows from Cases 1, 2 and 3.
\end{proof}

\subsection{Proof of Proposition \ref{prop:ellenergy}} \label{sec:Qell}

We next present the proof of Proposition \ref{prop:ellenergy}.
We start with the following basic elliptic estimate.

\begin{lemma}\label{basic.elliptic} For every compact $K\subset M_N$, there exists $c<C$ such that if  $q\in K$ 
\[
c\| f\|_{H^1(\R^2)}^2 
\le \int_{\R^2} f (-\Delta f + |\mphi(q)|^2 f)  \le 
\| f\|_{H^1(\R^2)}^2 
\]
for every $f\in H^1(\R^2)$.
\end{lemma}

We recall the standard proof. 

\begin{proof}
    Since $|\mphi(\cdot;q)|\le 1$ everywhere on $\R^2$ for every $q$, the inequality on the right is obvious.
    
    For the other inequality, since $K$ is compact and $|\mphi(x;q)|\to 1$ as $|x|\to \infty$ for every 
    $q$, we can fix a function $V:\R^2\to \R$ and a number $R>0$ such that
    \[
    V(x)\ge \frac 12\mbox{ if }|x|\ge R, \qquad 0 \le V(x) \le |\mphi(x;q)|^2 \ \ \mbox{ for every $q\in K$ and $x\in \R^2$}.
    \]
    Then it suffices to show that there exists $c>0$ such that 
    \[
   c\| f\|_{L^2(\R^2)}^2 \le \int_{\R^2} (|\nabla f|^2+ V f^2)
    \] 
    as the desired inequalities then follow easily.
    This can be proved by an argument by contradiction. Indeed, if we fix functions $f_k$ such that
    \[
    1 = \| f_k\|_{L^2}^2 \ \ge  \ k\int_{\R^2} (|\nabla f_k|^2+ V f_k^2) \ \ge \ \frac k2\int_{\{ x : |x|\ge R\}}  f_k^2
    \]
    then it is easy to check that $f_k$ converges weakly in $H^1(\R^2)$ and strongly in $L^2(\R^2)$, after passing to a subsequence, to a nonzero limit, say $f\in H^1(\R^2)$, supported on $\{ x : |x|\le R\}$. It follows that
    \[
    0< \int_{\R^2} (|\nabla f|^2+ V f^2) \le \liminf_k \int_{\R^2} (|\nabla f_k|^2+ V f_k^2) =0,
    \]   
    a contradiction.
\end{proof}

\begin{remark}
    In fact the previous Lemma holds with a constant independent of $q\in M_N$. This follows from the above considerations combined with a symmetrization argument (replacing $V$ by its radial increasing symmetrization and $f$ by its radial decreasing symmetrization), the Hardy-Littlewood and Polya-Szego inequalities, and the observation that for any $q\in M_N$,
    \[
    \frac 18 \int_{\R^2} (1-|\mphi(q)|^2)^2 \le E(\modu(q)) = \pi N,
    \]
    and thus $\{ x\in \R^2 : |\mphi(x;q)|^2 \le \frac 12\}$ has measure bounded by some constant $C_N$ depending only on $N$. 
    \end{remark}

Let $q:\R^{d-2}\to M_N$ temporarily denote a map satisfying 
\[
\mbox{Image}(q)\subset K \mbox{ compact }\subset M_N,  \mbox{ with  }\| \pp_{\bar a} q\|_{H^L}\le \Theta<\infty
\]
and let
    \[
    \calM_q\tU (t,x_a,x_{\bar a}) =  \begin{pmatrix}
(-\Delta_{\bar a} + L[q_\ep])\tu \\ 
(-\Delta +|\mphi(q_\ep)|^2)\tA_{\bar a}\\
(-\Delta +|\mphi(q_\ep)|^2)\tA_{0}
\end{pmatrix}, \qquad q_\ep(x_{\bar a}) := q(\ep x_{\bar a}).
    \]
    Thus $\calM_q$ denotes the operator $\calM$ based around a generic map $q:\R^{d-2}\to M_N$, independent of $t$.
    Given $\tU:\R^d\to \C\times \R^{d+1}$ of the form
    \[
    \tU = \begin{pmatrix} \tu  \\ \tA_{\bar a} \\ \tA_0\end{pmatrix} \qquad\mbox { with }\qquad \tu = \binom \tPhi {\tA_a}
    \]
    where as usual $a=1,2$ and $\bar a = 3,\ldots, d$, we will write
    \[
    c_{q,\mu}(x_{\bar a}) = ( \tu(\cdot , x_{\bar a}), n_\mu(\cdot; q_\ep(x_{\bar a}))_{L^2(\R^2)}
    \]
    where $q$ now denotes a map as above. We will also write 
    \[
    Q_{q,\ell}^\perp(\tU) := \int_{\R^d} \calM_q^{\ell+1}\tU \cdot \calM_q^\ell\tU, \qquad\qquad
    Q_{q,\ell}^\top(\tU) :=  \int_{\R^{d-2}} \sum_\mu |(-\Delta_{\bar a})^\ell c_{q,\mu}|^2
    \]

The $\ell=0$ case of \eqref{ell.energy} is an immediate consequence of the next lemma.

\begin{lemma}\label{lem:ell=1}
There exist positive constants $c<C$ such that for any $\tU:\R^d\to \C\times \R^{d+1}$,
    \[
    c \| \tU \|_{H^1(\R^d)}^2  \le 
    Q_{q,0}^\perp(\tU) + Q_{q,0}^\top(\tU) 
    \le C \| \tU \|_{H^1(\R^d)}^2 .
    \]
\end{lemma}

\begin{proof}
    The inequality on the right is rather clear and we omit the details of its verification.

    For the inequality on the left, we start from the definition \eqref{M.def} of $\calM_q$ and write 
    $-\Delta = -\Delta_{\bar a} - \Delta_a$,
    to find that
    \[
    \begin{aligned}
    &\int_{\R^d} \calM_q \tU \cdot \tU + \int_{\R^{d-2}} \sum_\mu |c_{q,\mu}|^2   \ = \  \int_{\R^d} \tU \cdot (-\Delta_{\bar a} \tU)\\
    &\hspace{5em}
    +  \int_{\R^{d-2}} \left(\sum_{\alpha= 0,3,\ldots, d}\int_{\R^2} \tA_\alpha(-\Delta_a + |\mphi(q_\ep)|^2) \tA_\alpha\, dx_a\right)dx_{\bar a}
    \\
    &\hspace{5em} + \int_{\R^{d-2}} \left( \int_{\R^2}  \tu \cdot L[q_\ep]\tu \ dx_a\,  + \sum_\mu |c_{q,\mu}|^2\right) dx_{\bar a}         .
    \end{aligned}
    \]
    The first integral on the right-hand side becomes $\| \nabla_{\bar a}\tU\|_{L^2(\R^d)}^2$ after integrating by parts. 
    For every $\alpha$,  Lemma \ref{basic.elliptic} implies that
    \[
    \int_{\R^{d-2}} \left(\int_{\R^2} \tA_\alpha(-\Delta_a + |\mphi(q_\ep)|^2) \tA_\alpha\, dx_a\right)dx_{\bar a}
    \ge c \int_{\R^{d-2}} \int_{\R^2}( |\nabla_a \tA_\alpha|^2 + \tA_\alpha^2) \, dx_a\, dx_{\bar a}.
    \]
    Similarly, Stuart's coercivity result, Theorem \ref{thm:hessian}, implies that
    \[
    \int_{\R^{d-2}} \left( \int_{\R^2}  \tu \cdot  L[q_\ep]\tu dx_a\,  + \sum_\mu |c_{q,\mu}|^2\right) dx_{\bar a} 
    \ge c \int_{\R^{d-2}} \int_{\R^2}( |\nabla_a \tu|^2 + |\tu|^2) \, dx_a\, dx_{\bar a}.
    \]
    Adding up these estimates, we obtain
    \[
    \int_{\R^d} \calM_q \tU \cdot \tU + \int_{\R^{d-2}} \sum_\mu |c_{q,\mu}|^2  \ge c\int_{\R^d} |\nabla_{\bar a}\tU|^2 + |\nabla_a \tU|^2 + |\tU|^2  \ = \ c\|\tU\|_{H^1(\R^d)}^2.
    \]
\end{proof}

\begin{lemma}\label{lem:higherest}
    for any $\ell \ge 1$, there exist positive constants $c<C$ such that for any $\tU\in H^{2\ell+1}(\R^d; \C\times \R^{d+1})$,
    \[
    \begin{aligned}
    c \| \tU \|_{H^{2\ell+1}(\R^d)}^2  \le Q_{q,\ell}^\perp(\tU) + Q_{q,\ell}^\top(\tU)  + \| \tU \|_{L^2}^2
    \le C \| \tU \|_{H^{2\ell+1}(\R^d)}^2 .   
    \end{aligned}
    \]
\end{lemma}

\begin{proof}Again we will prove only the estimate on the left, since the other is straightforward.

Define
\[
c_{q,\mu}^\ell := ( (\calM_{q,u}^\ell(\cdot, x_{\bar a}) \tu, n_\mu(\cdot, x_{\bar a}))_{L^2(\R^2)},
\]
where
\[
\calM_{q,u}\tu :=
(-\Delta_{\bar a} + L[q_\ep])\tu \ = \ \mbox{ the $u$ components of }\calM_q\tU.
\]
Then applying the previous lemma with $\tU$ replaced by $\calM_q^\ell \tU$ implies that
\begin{equation}\label{higherest0}
   c \| \calM_q^\ell \tU \|_{H^1(\R^d)}^2  \le \int_{\R^d} \calM_q^{\ell+1} \tU \cdot \calM_q^\ell\tU + 
   \sum_\mu \|c_{q,\mu}^\ell\|_{L^2}^2 . 
\end{equation}
To prove the Lemma, it therefore suffices to prove the following two inequalities. 
First, there exists $C>0$ such that 
\begin{equation}\label{higherest1}
  \| \tU\|_{H^{2\ell+1}}^2 \le C \| \calM_q^\ell \tU \|_{H^1}^2 +  C \| \tU\|_{L^2}^2
\end{equation}
and second, for any $\delta>0$ and any $\mu=1,\ldots, 2N$, there exists $C$  such that 
\begin{equation}\label{higherest2}
\|c_{q,\mu}^\ell \|_{L^2}^2 \le  \delta   \| \tU\|_{H^{2\ell+1}}^2 + 2\| (-\Delta_{\bar a})^\ell c_{q,\mu} \|_{L^2}^2 + C \| \tU\|_{L^2}^2.
\end{equation}
Indeed, for a suitably small choice of $\delta$ in \eqref{higherest2}, these imply that
\begin{align*}
     \| \tU\|_{H^{2\ell+1}}^2 &\overset{\eqref{higherest1}, \eqref{higherest0}}\le
     C \int_{\R^d} \calM_q^{\ell+1} \tU \cdot \calM_q^\ell\tU + C\sum_\mu \|c_{q,\mu}^\ell\|^2_{L^2} + C \| \tU\|_{L^2}^2
     \\
     &\overset{\eqref{higherest2}}\le 
     C \int_{\R^d} \calM_q^{\ell+1} \tU \cdot \calM_q^\ell\tU + C\sum_\mu \|(-\Delta)^\ell c_{q,\mu}\|^2_{L^2} + C \| \tU\|_{L^2}^2 + \frac 12  \| \tU\|_{H^{2\ell+1}}^2
\end{align*}
from which the desired conclusion follows.

{\em Proof of \eqref{higherest1}}:
From the definition \eqref{M.def} of $\calM_q$, we see that
\[
\calM_q \tU = - \Delta \tU + L_{q,1}\tU
\]
where 
$L_{q,1}$ is a first-order differential operator with smooth coefficients. Thus
\begin{equation}\label{leadingorder}
\calM_q^\ell \tU = (-\Delta)^\ell \tU + L_{q,\ell}\tU
\end{equation}
where $L_{q,\ell}$ is a differential operator of order at most $2\ell-1$ with smooth coefficients. 
Thus
\begin{align}
    \| \tU\|_{H^{2\ell+1}}^2 
    &\le C\left( \|\nabla(-\Delta)^\ell \tU\|_{L^2}^2 + \|  \tU\|_{L^2}^2 \right)  \nonumber\\
    &\le  C\left(\|\nabla \calM_q ^\ell \tU\|_{L^2}^2 +  \|\nabla  L_{q,\ell} \tU\|_{L^2}^2 +  \| \tU\|_{L^2}^2\right)\nonumber \\
    &\le  C\left(\|\calM_q ^\ell \tU\|_{H^1}^2 +  \|\tU\|_{H^{2\ell}}^2 +  \| \tU\|_{L^2}^2\right).\label{higherest1a}
\end{align}
We now use the fact that for every $\delta>0$ there exists $C=C_\delta$ such that 
\[
\|\tU\|_{H^{2\ell}}^2 \le \delta \| \tU\|_{H^{2\ell+1}}^2  + C\|\tU\|_{L^2}^2.
\]
We deduce \eqref{higherest1} by inserting this into \eqref{higherest1a}, choosing $\delta$ small enough that $C\delta<1/2$, then rearranging.

{\em Proof of \eqref{higherest2}}:
Recall that for every $q\in M_N$, $L[q]$ is self-adjoint and $L[q]n_\mu=0$. Thus for any $\tu\in H^2(\R^2; \C\times \R^2)$ and any $\mu\in \{1,\ldots, 2N\}$,
\begin{align*}
    ( \calM_{q,u} \tu, n_\mu )_{L^2(\R^2)}
& =
( (-\Delta_{\bar a} + L[q_\ep]) \tu, n_\mu )_{L^2(\R^2)}
\\
&
=
( -\Delta_{\bar a} \tu, n_\mu )_{L^2(\R^2)}
\\
&
=
-\Delta_{\bar a}(  \tu, n_\mu )_{L^2(\R^2)}
+ 2( \nabla_{\bar a} \tu , \nabla_{\bar a} n_\mu )_{L^2(\R^2)} +
( \tu , \Delta_{\bar a} n_\mu )_{L^2(\R^2)}.
\end{align*}
Since $\nabla_{\bar a}^k n_\mu = O(\ep^k)$  for every $k$, we deduce that
\begin{equation}\label{cmu.ell.1}
\left\|    ( \calM_{q,u} \tu, n_\mu )_{L^2(\R^2)}
- (-\Delta_{\bar a}) (  \tu, n_\mu )_{L^2(\R^2)} \right\|_{H^k(\R^{d-2})}^2 \le C \e^2 \| \tu \|_{H^{k+1}(\R^d)}^2
\end{equation}
for every $k$.
Next, we write a telescoping sum
\begin{align*}
c_{q,\mu}^\ell - (-\Delta_{\bar a})^\ell c_{q,\mu}
&= ( \calM_{q,u}^\ell \tu, n_\mu )_{L^2(\R^2)}
- (-\Delta_{\bar a})^\ell (  \tu, n_\mu )_{L^2(\R^2)} 
\\&
= 
\sum_{j=0}^{\ell-1}
(-\Delta_{\bar a})^j \left( (  \calM_{q,u}^{\ell-j}  \tu, n_\mu )_{L^2(\R^2)}
-  (-\Delta_{\bar a}) ( \calM_{q,u}^{\ell - j - 1}  \tu, n_\mu )_{L^2(\R^2)}\right)
\end{align*}
and apply \eqref{cmu.ell.1} to each term in the sum (with $\tu$ replaced by $\calM_{q,u}^{\ell-j-1}\tu$) to deduce that
\begin{align*}
\|c_{q,\mu}^\ell - (-\Delta_{\bar a})^\ell c_{q,\mu}\|_{L^2(\R^{d-2})}
&\le\sum_{j=0}^{\ell-1}
C \ep^2 \| \calM_{q,u}^{\ell-j-1}\tu \|_{H^{2j+1}(\R^{d-2})}
\ \le \ C\ep^2  \| \tU \|_{H^{2\ell -1}}.
\end{align*}
This immediately implies \eqref{higherest2} if $\ep$ is sufficiently small relative to $\delta$. If not, we can argue as at the end of the proof of \eqref{higherest1} to conclude.
\end{proof}

We are now ready to present the 

\begin{proof}[Proof of Proposition \ref{prop:ellenergy}]
We will only prove the inequality 
\begin{equation}\label{ellest.restate}
\kz \| \pp_t\tU(\cdot, t)\|_{H^{2\ell}}^2 + \| \tU(\cdot, t)\|_{H^{2\ell+1}}^2 \le C \calQ(t)
\end{equation}
as the other inequality is straightforward, and we will sometimes omit the argument $(\cdot, t)$.

For $t\in [0,T/\ep)$ and we $q_\ep(t)$ to denote the map $x_{\bar a}\in \R^{d-2}\to q( \ep x_{\bar a}, \ep t)$ where $q$ is the given solution of \eqref{dAHM}. Thus we can write
\[
Q_\ell^\perp(t) = Q_{q_\ep(t), \ell}^\perp(\tU(\cdot, t)), \qquad
Q_\ell^\top(t) = Q_{q_\ep(t), \ell}^\top(\tU(\cdot, t)).
\]
It follows directly from Lemmas \ref{lem:ell=1} and \ref{lem:higherest} that
\begin{equation}\label{H2ell+1}
\| \tU(\cdot, t) \|_{H^{2\ell +1}}^2 \le
C \left( Q_\ell^\perp(t) +Q_\ell^\top(t) + Q_0^\perp(t) + Q_0^\top(t) \right)
\end{equation}

We will also write $\calM_{q_\ep(t)}^\ell = (-\Delta)^\ell + L_{q_\ep(t),\ell}$ where $L_{q_\ep(t),\ell}$ is a differential operator of order at most $2\ell -1$ with coefficients depending on $q_\ep$.
With this notation,
\begin{align*}
\pp_t(\calM_{q_\ep(t)}^\ell \tU )
&= \pp_t (-\Delta)^\ell \tU + \pp_t(L_{q_\ep(t),\ell}\tU) \\
&= \pp_t (-\Delta)^\ell \tU + (\pp_t L_{q_\ep(t),\ell})\tU + L_{q_\ep(t),\ell}\pp_t\tU.
\end{align*}
It follows that
\[
\|(-\Delta)^\ell \pp_t \tU\|_{L^2(\R^d)}^2
\le
C \left( \| \pp_t\calM^\ell \tU\|_{L^2}^2 + \| \pp_t \tU\|_{H^{2\ell -1}}^2 + \| \tU\|_{H^{2\ell -1}}^2\right). 
\]
We combine this with the standard inequalities
\[
\| \pp_t \tU\|_{H^{2\ell}}^2 \le \| (-\Delta)^\ell \pp_t\tU\|_{L^2}^2 
+ C \| \pp_t\tU\|_{L^2}^2
\]
and, for any $\delta>0$,
\[
\| \pp_t \tU\|_{H^{2\ell -1}}^2 \le \delta \| \pp_t \tU\|_{H^{2\ell}}^2 + C_\delta \| \tU\|_{L^2}^2.
\]
Choosing $\delta>0$ small enough and rearranging, we deduce that
\[
\| \pp_t \tU\|_{H^{2\ell}}^2 \le
C \left( \| \pp_t\calM^\ell \tU\|_{L^2}^2 + \| \pp_t \tU\|_{L^2}^2 + \| \tU\|_{H^{2\ell -1}}^2\right). 
\]
Thus
\[
\kz \| \pp_t\tU\|_{H^{2\ell}}^2 + \| \tU\|_{H^{2\ell+1}}^2 \le
C\kz \left( \| \pp_t\calM^\ell \tU\|_{L^2}^2 + \| \pp_t \tU\|_{L^2}^2\right)
+ C\| \tU\|_{H^{2\ell +1}}^2. 
\]
The conclusion \eqref{ellest.restate} follows by combining this with \eqref{H2ell+1} and recalling the definition \eqref{calQ.def} of $\calQ_\ell(t)$.
\end{proof}

\section{proof of Theorems \ref{thm:1} and \ref{thm:2}} \label{sec:mainproofs}

In this section we prove our main theorems. We first restate Theorem \ref{thm:1}.

\begin{thm1full}
Let $q:\R^{d-2}\times (0,T)\to M_N$ be a solution of \eqref{dwm} satisfying \eqref{q.compact0} and \eqref{q.decay} for $L=2d+11$,
and assume that $\Theta<\infty$, where $\Theta$ is defined in \eqref{Theta.def0}.

Let $\aU:(0,T_0/\ep)\times \R^d\to \C\times \R^{d+1}$ denote the approximate solution constructed in  Proposition \ref{prop:approx.sol}.

Then there exist constants $\ep_0>0$ and $0<c_0<c_10$ such that, setting $\ell = \lfloor d/4\rfloor+2$,  if  $0<\ep<\ep_0$ and $(\tU^0_0, \tU^0_1)\in H^{2l+1}\times H^{2l}(\R^d; \C\times \R^{d+1})$ satisfies
\beq\label{final.data}
\| \tU^0_0\|_{H^{2l+1}(\R^d)}  + \kz \|\tU^0_1\|_{H^{2l}(\R^d)} < c_0\ep^2
\eeq
as well as the compatibility and gauge conditions appearing in \eqref{modified.data} as well as the orthogonality condition appearing in \eqref{cmu.initialdata},
then there exists a map $\tU:(0,T_0/\ep)\times \R^d\to \C\times \R^{d+1}$ such that 
\[
\aU+\tU\mbox{ solves \eqref{dAHM}} , \quad \qquad \ (\tU, \kz\pp_0\tU)_{t=0} = (\tU^0_0, \kz\tU^0_1)
\]
and
\beq\label{final.tU}
\| \tU\|_{L^\infty_{T_0/\ep} H^{2\ell+1}} + \kz \|\pp_0 \tU \|_{L^\infty_{T_0/\ep} H^{2\ell}} < c_1\ep^2.
\eeq
The solution $U_\ep := \aU +\tU$ of \eqref{dAHM} thereby constructed satisfies
\begin{equation}\label{finalerror}
\begin{aligned}
\| \Phi_\ep - \mphi(q_\ep)\|_{L^\infty_{T_0/\ep}W^{2,\infty}(\R^d)} 
+\kz \|\pp_t( \Phi_\ep - \mphi(q_\ep))\|_{L^\infty_{T_0/\ep}W^{1,\infty}(\R^d)} 
&\le C \ep^2, \\
\| A_{\ep a} - \mA_a(q_\ep)\|_{L^\infty_{T_0/\ep}W^{2,\infty}(\R^d)} 
+ \kz \| \pp_t(A_{\ep a} - \mA_a(q_\ep))\|_{L^\infty_{T_0/\ep}W^{1,\infty}(\R^d)} 
&\le C \ep^2\\
\|  A_{\ep \alpha} - \chi_{\pp_\alpha q_\ep}(q_\ep) \|_{L^\infty_{T_0/\ep}W^{2,\infty}(\R^d)}
+\kz \| \pp_t (A_{\ep \alpha} - \chi_{\pp_\alpha q_\ep}(q_\ep)) \|_{L^\infty_{T_0/\ep}W^{1,\infty}(\R^d)}
&\le C \ep^2
\end{aligned}    
\end{equation}
for $a=1,2$ and $\alpha = 0,3,\ldots, d$.


\end{thm1full}

\begin{remark}It is easy to see that there exist initial data satisfying conditions \eqref{final.data}, \eqref{modified.data} and \eqref{cmu.initialdata}.

Indeed, if $\kz=0$ we can simply take $\tU^0_0=0$.

In the hyperbolic case $\kz>0$, we must prescribe $\tU^0_0$ and $\tU^0_1$. We can choose
\beq\label{tdata1}
\tPhi^0_0=\tPhi^0_1 =0, \qquad \tA^0_i = \pp_0\tA^0_i = 0\ \ \mbox{ for }i=1,\ldots, d. 
\eeq
Then the requirement that $S_0[\aU+\tU]=0$ forces us to choose $\tA_0$ satisfying
\beq\label{tdata2}
0 = S_0[\aU +\tU]\Big|_{t=0} = (-\Delta + |\aPhi(\cdot, 0)|^2)\tA^0_0 + S_0[\aU](\cdot, 0).
\eeq
We fix $\pp_0\tA_0|_{t=0} = \tA^1_0$ by specifying that
\beq\label{tdata3}
G[\aU](\tU)\Big|_{t=0} \ = \ 
\kz \tA^1_0 + \ep\ko\tA^0_0 = 0.
\eeq
Next, recall that condition \eqref{cmu.initialdata} states that $c_\mu = \kz \pp_t c_\mu = 0$ when $t=0$, where
\[
c_\mu(x_{\bar a},t) = ( \tu , n_\mu )_{L^2(\R^2_{(x_{\bar a},t)})}
\]
This is clear for the above choice of initial data, which implies that $\tilde u = \kz \pp_t\tilde u = 0$ at $t=0$.

Finally\eqref{tdata1}, \eqref{tdata2}, our estimate \eqref{good1} of $\|S_0[\aU]\|_{H^{2\ell+1}}$, and elliptic estimates imply that
\[
\| \tU^0_0\|_{H^{2\ell+1}}^2 = 
\| \tA_0^0\|_{H^{2\ell+1}}^2 \le C \|S[\aU](\cdot, 0)\|_{H^{2\ell}}^2 \le C\ep^9
\]
and then we see from \eqref{tdata3} that
\[
\kz \|\tU^0_1\|_{H^{2\ell}}^2 = \kz \|\tA^0_1\|_{H^{2\ell}}^2 = \frac{(\ep \ko)^2}{\kz} \|\tA^0_0 \|_{H^{2\ell}}^2   \le C \ep^{11}.
\]

\end{remark}

\begin{proof}
For $\kz\ge 0$ let $X_{\kz}$ denote the space
\[
X_{\kz} := 
\begin{cases}
    C^0([0,T_0/\ep) ; H^{2\ell+1}) &\mbox{ if }\kz=0\\
    C^0([0,T_0/\ep) ; H^{2\ell+1})
    \cap     C^1([0,T_0/\ep) ; H^{2\ell}) &\mbox{ if }\kz>0
\end{cases}
\]
where $H^s$ denotes $H^s(\R^d;\C\times \R^{d+1})$.
This is a Banach space with the norm $\vvvert \, \cdot\, \vvvert_{2\ell}$ defined in \eqref{nnnorm}:
\[
\vvvert U \vvvert_{2\ell} := \sup_{0<t<T_0/\ep} \left( \kz \| \pp_t U(\cdot, t)\|_{H^{2\ell}(\R^d)} +
\| U(\cdot, t)\|_{H^{2\ell+1}(\R^d)} \right).
\]
For $\tV\in X_{\kz}$ we define $F(\tV) := \tU$, where
$\tU$ solves the modified system \eqref{aux.system}, with $\calN(\tU)$ replaced by $\calN(\tV)$, that is:
\[
(\kz \pp_{tt} + \ep \ko \pp_t)\tU + \calM\tU + P\tU = - S[\aU] - \calN(\tV),
\]
with initial data $(\tU,  \kz\pp_t\tU)|_{t=0} = (\tU^0_0, \kz\tU^0_1)$.
It follows from Theorem \ref{thm:linearest} and our assumed bounds \eqref{final.data} on the initial data that $F$ maps $X_{\kz}$ to $X_{\kz}$, with the estimate
\begin{align*}
\vvvert F(\tV) \vvvert_{2\ell}^2 = \vvvert \tU \vvvert_{2\ell}^2
\le C c_0^2\ep^4+\frac C {\ep^4}\sup_{0<t<T_0/\ep} \| S[\aU](\cdot, t)+ \calN(\tV)(\cdot, t) \|_{H^{2\ell}}^2.    
\end{align*}
Recall from  \eqref{good1} that 
\[
\sup_{0<t<T_0/\ep} \| S[\aU](\cdot, t) \|_{H^{2\ell}}^2 \le C \ep^9.
\]
Putting these together and recalling \eqref{aux.props2}, we find that
\[
\vvvert F(\tV)\vvvert_{2\ell}^2
\le
C c_0^2\ep^4 + C \ep^{5} + C\ep^{-4} \vvvert \tV \vvvert_{2\ell}^4
\]
if $\vvvert \tV \vvvert_{2\ell}\le 1$.

It follows that there $c_0, c_1, \ep_0>0$ such that if we define
\[
\calB := \{ \tU\in X_{\kz} : \vvvert \tU\vvvert_{2\ell} \le c_1\ep^2 \}
\]
then $F$ maps $\calB$ into $\calB$ whenever $0<\ep<\ep_0$.

We claim that in addition, $F:\calB\to \calB$ is a contraction mapping with respect to the natural $\vvvert \,\cdot\,\vvvert_{2\ell}$ norm, and after possibly shrinking $c_0, c_1, \ep_0$. To verify this, we fix
$\tV_1, \tV_2\in \calB$, and let $\tW := F(\tV_1) - F(\tV_2)$.
Then $\tW$ solves
\[
(\kz \pp_{tt} + \ep \ko \pp_t)\tW + \calM\tW + P\tW = \calN(\tV_2) - \calN(\tV),
\]
with $\tW(\cdot, 0)= 0$ and (if $\kz>0$) $\pp_t\tW(\cdot, 0)=0$. So Theorem \ref{thm:linearest} implies that
\[
\vvvert \tW \vvvert_{2\ell}^2 \le C\ep^{-4}\sup_{0<t<T_0/\ep} \| \calN(\tV_2)(\cdot, t) - \calN(\tV)(\cdot, t)\|_{H^{2\ell}}^2.
\]
Appealing to \eqref{aux.props2}, we deduce that
\[
\vvvert \tW \vvvert_{2\ell} \le C \ep^{-4} ( \vvvert \tV_1\vvvert_{2\ell} + \vvvert \tV_1\vvvert_{2\ell})^2 \vvvert \tV_1 - \tV_2\vvvert_{2\ell}^2
\]
Since $\tV_1,\tV_2\in \calB$, we conclude that (after possibly shrinking $c_1$, then adjusting $c_0,\ep_0$ as necessary)
\[
\vvvert F(\tV_1) - F(\tV_2) \vvvert_{2\ell}^2 = \vvvert \tW \vvvert_{2\ell}^2 \le \frac 12\vvvert \tV_1-\tV_2 \vvvert_{2\ell}^2.
\]
Thus the Contraction Mapping Principle implies that there is a unique $\tU\in \calB$ such that $F(\tU) = \tU$. This says exactly that $\tU$ solves \eqref{aux.system} with the required initial data, and hence that $\aU + \tU$ solves  \eqref{dAHM}, with
$\vvvert \tU \vvvert_{2\ell} \le c_0\ep^2$, which is exactly \eqref{final.tU}.

Finally, to prove \eqref{finalerror}, recall that the approximate solution 
$\aU = U_{m,\ep}$ has the form $\aU=U^\ep_0 +\ep^2 U^\ep_m$, where $U^\ep_0$ is defined in \eqref{U0.form}, so the actual solution $U = \aU+ \tU$ satisfies
\beq\label{addlabel}
\| U - U^\ep_0 \|_{L^\infty_{T_0}W^{2,\infty}(\R^d)} \le  \ep^2 \| U^\ep_m\| _{L^\infty_{T_0}W^{2,\infty}(\R^d)} + \| \tU\|_{L^\infty_{T_0}W^{2,\infty}(\R^d)} .
\eeq
We have chosen $\ell$ such that $2\ell+1> \frac d2+2$, so $H^{2\ell+1}(\R^d)\hookrightarrow W^{2,\infty}(\R^d)$, and  \eqref{finalerror} follows from  \eqref{good3} and \eqref{final.tU}. The estimates of $\| \pp_t(U - U^\ep_0 )\|_{L^\infty_{T_0}W^{1,\infty}(\R^d)} $ follow in exactly the same way.
\end{proof}

Theorem \ref{thm:2} has a full version, exactly along the same lines as Theorem \ref{thm:1} above. We now present its proof.

\begin{proof}[Proof of Theorem \ref{thm:2}]

The present theorem differs from Theorem \ref{thm:1} in two ways.


 First, we consider \eqref{dwmp} on $\T^{d-2}_R\times (0,T)$, and similarly we study \eqref{dAHM} on $\R^2\times \T^{d-2}_{R/\ep}\times (0,T_0/\ep)$. This has essentialy no impact on our arguments.
At various points in the proof of Theorem \ref{thm:1} we use PDE estimates 
for functions whose spatial domain is $\R^{d-2}$ or $\R^d$. The introduction of $y_{\bar a}$-periodicity in Theorem \ref{thm:2} requires versions of these estimates for which the spatial domain is changed to $\T^{d-2}_R$ or $\R^2\times \T^{d-2}_R$. 
These hold, with the proofs essentially unchanged from the non-periodic case. (There are some places in which arguments could be simplified in the $y_{\bar a}$-periodic setting, as a consequence of the fact that $L^\infty\subset L^2$ on $\T^{d-2}_R$.)

The other new feature compared to Theorem \ref{thm:1},  the extension to the near-critical case $\lambda = 1+\kt\ep^2$,  introduces a few changes. To construct approximate solutions, we start by changing variables as in Section \ref{sec:3.1}.  Defining $\wS^\ep[\wU]$ as in \eqref{wSe.def}, it remains the case that
$\wS^\ep = \wS^1 + \ep^2\wS^1$ as in \eqref{eq:wSe1}, and $\wS^0$ does not change at all, see \eqref{eq:wSe2}. The nonzero $\kt$ appears only in $\wS^1$, which becomes 
\begin{equation}\label{eq:wSe3kt}
\wS^1[\wU] =
\begin{pmatrix}
(\kz \wD_0\wD_0 + \ko \wD_0 - \wD_{\bar a}\wD_{\bar a})\Phi + \frac {\kt}{2}(|\Phi|^2-1)\Phi \\
(\kz\wpp_0+\ko)\wF_{0a} - \wpp_{\bar a} \wF_{\bar a a} \\
(\kz\wpp_{0}+\ko) \wF_{0  \bar a}  - \wpp_{\bar b}  \wF_{\bar b \bar a} \\
-\wpp_{\bar b}  \wF_{\bar b 0 }
\end{pmatrix}.
\end{equation}
It follows that $\ell^1_\phi[\wU]$, see \eqref{eq:wSe6}, \eqref{eq:wSe7}, changes to 
\begin{equation}\label{eq:wSe7a}
\begin{aligned}
\ell^1_\phi[\wU](\wtU) &= (\kz \wD_{A_0}\wD_{A_0} + \ko \wD_{A_0} - \wD_{A_{\bar a}}\wD_{A_{\bar a}})\tPhi  + \frac \kt 2 ((|\Phi|^2-1)\tPhi +
2 (\Phi, \tPhi) \Phi)
\\
&\hspace{1em}\ \  
-
 (\kz\wpp_0 \wtA_0 + \ko \wtA_0 -  \wpp_{\bar a}\wtA_{\bar a} )i\Phi -2i(\kz\wtA_0 \wD_0\Phi  - \wtA_{\bar a}\wD_{\bar a}\Phi) .
\end{aligned}
\end{equation}
Neither $\ell^0[\wU]$, see \eqref{eq:wSe5},  nor the other components of $\ell^1[\wU]$ are affected.

Lemma \ref{lem:S1U0} is the point at which the proof of Theorem \ref{thm:1} breaks down when $\kt\ne0$. The lemma establishes estimates for  $\wS^1[\wU^0]$, including in $L^\infty_T H^{L-1}(\R^d)$. These are simply not true when $\kt \ne 0$, since $\frac \kt 2(|\Phi|^2-1)\Phi$ does not decay as $|y_{\bar a}|\to \infty$.  However, the proof carries over with no change to show that in the $y_{\bar a}$-periodic setting, the needed estimate holds:
\beq\label{S1U0.est.kt}
    \| \wS^1[\wU^0] \|_{L^\infty_T X^{L-1,j}_{\gamma}(\R^2\times \T^{d-2}_R)} +   \kz  \| \wpp_0\wS^1[\wU^0] \|_{L^\infty_T X^{L-2,j}_{\gamma}(\R^2\times \T^{d-2}_R))}  \le C(\Theta).
    \eeq
 for every $\gamma\in (0,1)$.

The necessity of adding a potential to the geometric evolution equation \eqref{dwmp} in the near-critical case arises when we revisit Lemma \ref{lem:J1}. Noting that $\wU^0_\phi = \mphi(q)$, 
it is convenient to write
\[
\wS^1_u[\wU^0] =  
\wS^1_{u}[\wU^0]\Big|_{\kt=0} + 
 \frac {\kt}{2}
 \begin{pmatrix}
(|\mphi(q)|^2-1)\mphi(q) \\
0
\end{pmatrix}
\]
Thus
\[
(\wS^1_u[\wU^0] , \tilde n_\mu)_{L^2(\R^2)} = (\wS^1_{u}[\wU^0]\Big|_{\kt=0} , \tilde n_\mu)_{L^2(\R^2)}
+
\frac {\kt}{2}
( \begin{pmatrix}
(|\mphi(q)|^2-1)\mphi(q)  \\
0
\end{pmatrix}, \tilde n_\mu )_{L^2(\R^2)}
\]
The content of Lemma \ref{lem:J1} is exactly that
\[
(\wS^1_u[\wU^0] \Big|_{\kt=0} , \tilde n_\mu)_{L^2(\R^2)} = (\kz \nabla_0 \pp_0q +\ko \pp_0 q - \nabla_{\bar a}\pp_{\bar a}q , \pp_{\mu})_g.
\]
And writing $V_0(q) := \frac 18(|\mphi(q)|^2-1)^2$, 
\begin{align*}
\left( \begin{pmatrix}
(|\mphi(q)|^2-1)\mphi(q)  \\0
\end{pmatrix}, \tilde n_\mu \right)_{L^2(\R^2)}
&= ((|\mphi(q)|^2-1)\mphi(q), \calD_\mu\mphi(q))_{L^2(\R^2)} \\
&= 2 \pp_\mu V_0(q) \ = \ 2 (\nabla_g V_0(q) , \pp_\mu)_g
\end{align*}
Thus 
\[
(\wS^1_u[\wU^0], \tilde n_\mu)_{L^2(\R^2)} = (\kz \nabla_0 \pp_0q +\ko \pp_0 q - \nabla_{\bar a}\pp_{\bar a}q + \kt \nabla_g V_0(q) , \pp_{\mu})_g.
\]
It follows that the orthogonality condition holds if and only if the modified geometric evolution equation \eqref{dwmp} is satisfied.

The first step in the construction of the approximate solution as carried out in Section \eqref{sec:3.3} is otherwise unchanged, noting that the algebraic identity in Lemma \ref{lem:identity} holds for all values of the coupling constant $\lambda$.

Similarly, the induction argument in Section \ref{sec:3.4} is essentially unchanged. In the proof of Lemma \ref{lem:Vek}, the new term in
$\ell^1[\wU]$ (compare \eqref{eq:wSe7a} and \eqref{eq:wSe7}) leads us to change \eqref{ell1phi} to
\beq\label{ell1phi.2}
\ell^1_\phi[\wU^0](c^{m+1}_\mu N_\mu) =  [(\kz \wpp_0\wpp_0 +\ko\wpp_0 - \wpp_{\bar a}\wpp_{\bar a} - \frac \kt 2)c^{m+1}_
\mu] n_{\mu,\phi} + \mbox{Tri} 
\eeq
with modifications to the trilinear term that do not affect its character or the estimate \eqref{TriZ.est}. Similarly the details of the linear operators $\Lambda^1_j$ change in a way that does not fundamentally alter the estimate \eqref{Lambda1j.est}. Thus Lemma \ref{lem:Vek}, suitably modified,  continues to hold. The rest of the construction of approximate solutions may be carried out exactly as in the final step of the proof of Proposition \ref{prop:approx.sol}

Having constructed an approximate solution, we proceed to formulate the modified system 
\[
(\kz \pp_0^2+ \ep\ko\pp_0)\tU + \calM\tU + P\tU = -\calN(\tU) - S[\aU]    ,.
\]
where we later decompose $P = P_1+P_2$.
We impose the same gauge condition \eqref{dynamic.gauge} as in the critical case. 
Since 
\[
S_\phi(U)\Big|_{\kt\ne 0} = S_\phi(U)\Big|_{\kt = 0} + \frac{\ep^2 \kt} 2 (|\Phi|^2-1)\Phi,
\]
the new terms that arise when $\kt\ne 0$ show up only in the nonlinear terms $\calN$, which however continue to satisfy \eqref{aux.props2}, and in the linear operator $P_2$, which becomes a little more complicated but continues to satisfy \eqref{P2.est}.

Finally, while the linear estimates in Theorem \ref{thm:linearest} rely strongly on the explicit form of $P_1$, see \eqref{P1.def}, the only properties they use of $P_2$ and $\calN$ are the estimates \eqref{P2.est} and \eqref{aux.props2}, which as noted remain true when $\kt\ne 0$. 

With these ingredients in place, the proof of Theorem \ref{thm:2} can be concluded by exactly the same argument, using the contraction mapping principle, as the proof of Theorem \ref{thm:1}.
\end{proof}

\section{reconnection}\label{sec:reconnection}

In this section we discuss  vortex filament collisions and reconnection in $d=3$
dimensions, leading up to the proof of Theorem \ref{thm:3}. For simplicity we will restrict our discussion to the case of $N=2$ filaments, in which our main theorem provides solutions of \eqref{dAHM} starting from a solution $q:I\times \R \to M_2$ of the geometric evolution equation \eqref{dwm}, where $I\subset \R$ is an interval.

Recall from Theorem \ref{thm1.new} that the $M_2 \cong \C^2$, with the bijections
\[
\begin{aligned}
(q_1, q_2)\in \C^2 
&\quad \longleftrightarrow\quad
\mbox{ the polynomial }p(z;q) := z^2+ q_1 z + q_2 \\
&\quad \longleftrightarrow\quad\mbox{ the roots (repeated according to multiplicity) of $p(\cdot; q)$}\\
&\quad\longleftrightarrow\quad
u(\cdot; q)=
\binom \mphi \mA\in M_2 
\mbox{ such that }|\mphi(z;q)| \approx \frac {|p(z;q)|}{\sqrt {1+|p(z;q)|^2}}.
\end{aligned}
\]
The last condition says that $\mphi(\cdot; q)$ vanishes at the precisely same points, and to the same order, as $p(\cdot;q)$.
It is customary to refer to the zeroes of $\mphi(\cdot;s)$, {\em i.e.} the roots of $p(\cdot;q)$, as the {\em vortex centers} for $u(\cdot; q)$.

With this in mind, given a map $q:\R\to \C^2 \cong M_2$ and $U^0:\R^3\to \C\times \R^3$ as appearing in the leading-order term in our approximate solution, that is,
\begin{equation}\label{Uzero.def}
U^0(y_a, y_3) =
\begin{pmatrix}
 \modu(y_a ; q(y_{3}))\\
\chi_{q'(y_{3})}(y_a) 
\end{pmatrix}, \qquad\quad\mbox{ with } \quad\modu(y_a ; q(y_{3})) = \binom{\mphi}{\mA}(y_a, q(y_3))
\end{equation}
we define 
\[
\mbox{ the {\em vortex centerlines} for $U^0$ are } \ \big\{ (z, y_3)\in \C\times \R : z\mbox{ is a root of } p(\cdot; q(y_3))  \big\}
\]
This is exactly the zero locus of the Higgs field $\mphi$ of the leading-order approximate solution $U^0$.

More generally, if $U = (\Phi, A_a, A_{3}):\R^3\to \C\times \R^3$ is a map such that $0$ is a regular value of $\Phi$, 
and for every $y_3$, 
\beq\label{centerline.need}
\Phi(\cdot, y_3)\in \calE_2,\mbox{ and $\Phi(\cdot, y_3)$ has exactly two zeros, each of degree $1$}
\eeq
then we define
\[
\mbox{ the {\em vortex centerlines} of $U$ are the set }\{ y\in \R^3 : \Phi(y)=0\}.
\]
Under our assumptions, this set consists of two curves, both graphs over the vertical $y_3$ axis.
We will not need to define vortex centerlines for more general $U$, since the solutions we have constructed above generally satisfy the above hypotheses. 

Informally, reconnection is a change in the connectivity of vortex centerlines. 
We illustrate this with a couple of examples, before presenting the proof of Theorem \ref{thm:3}.

     \begin{example}\label{example1}
 Let $q = (q_1, q_2): (0,T)\times \R \to \C^2\cong M_2$ be a solution of \eqref{dAHM} with $\kz>0$,
 with initial conditions\footnote{
 The assumption that $\kz>0$ allows us to prescribe the initial velocity in this example.}
 \begin{align*}
 q|_{t=0} 
 &= (q_1, q_2)|_{t=0} = (q^0_1, q^0_2)\in H^k(\R; \C^2), 
 \\
 \pp_{y_0}q|_{t=0} 
 &= (\pp_{y_0}q_1, \pp_{y_0}q_2)|_{t=0} = (q^1_1, q^1_2)\in H^{k-1}(\R; \C^2), 
\end{align*}
 such that
 \begin{align*}
&q^0_1(y_3) = q^1_1(y_3) = 0\mbox{ for all }y_3, 
\\
&q^0_2(y_3)=a y_3, \qquad
\qquad 
q^1_2(y_3) = ib \qquad\mbox{ with }a, b>0 \qquad\mbox{ for $y_3$ near $0$.}
 \end{align*}
One can check\footnote{This is most easily seen in new coordinates $(P,Q) := ( q_1, q_1^2-4q_2)$ on the moduli space. Then the vortex centers corresponding to $(P,Q)$ are the points $\frac 12(P\pm \sqrt{Q})$. It follows from the invariance of the Abelian Higgs action with respect to translations on $\R^2$ that the geometry of $M_2$ is invariant with respect to translations $(P,Q) \mapsto (P+v, Q)$ and moreover, see \cite{Samols}, that in these coordinates the metric takes the form
$g = (dP_1^2 + dP_2^2) + f(Q) (dQ_1^2+dQ_2^2)$.  Using this to compute the Christoffel symbols, it follows that when $q$ is written in these coordinates, the $P$ and $Q$ components decouple, and the $P$ components solve the linear equation
\[
(\kz \pp_{y_0}^2+\ko \pp_{y_0} - \pp_{y_3}^2)P = 0
\]
with initial conditions $(P, P_t)=(0,0)$.
This implies the claim.}
that the initial conditions for $q_1$ imply that $q_1(t,y_3)=0$ everywhere.
We can also solve the equation backwards in time to obtain a wave map $q:(-T,T)\times \R\to M_2$.
By a Taylor expansion around $t=0$, we find that 
this solution satisfies
\[
q_1(t,y_3)=0, \qquad q_2(t,y_3) = a y_3 + ibt + O(|t|^2+ |y_3|^2)
\]
and if $(y_3, t)$ is small then 
\beq\label{Im.change}
\Im(q_2(y_3,t))<0\mbox{ if }t<0, \qquad \qquad
\Im(q_2(y_3,t))>0\mbox{ if }t>0,
\eeq
Thus $q_2(\cdot, t)$ is qualitatively as pictured in Figure \ref{fig:two} below for small $t$, shown for concreteness with $a=b=1$. The plots show the curve in the modlui space for $-0.1<y_3<0.1$.

\begin{figure}[ht!]
\centering
\includegraphics[scale=1.5]{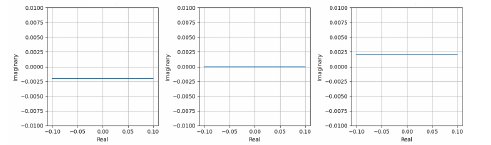}
\caption{Curves in the moduli space for different values of $t$: On left, $t=-0.002$; center, $t=0$; right, $t=0.002$.}
\label{fig:two}
\end{figure}

Since $q_1=0$, the roots of the $p(z; q(t,y_3)) = z^2+ q_2(t,y_3)$ are just $\pm \sqrt{- q_2(t, y_3)}$. Thus each curve in the moduli space yields a pair of vortex curves in physical space, as illustrated in Figure \ref{fig:three} below

\begin{figure}[ht!]
\centering
\includegraphics[scale=1.5]{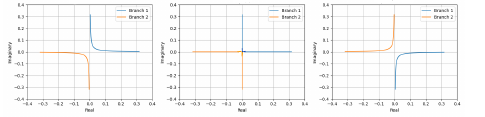}
\caption{Vortex curves for the leading-order approximate solution $U^0$, see \eqref{Uzero.def},corresponding to the moduli space curves shown above: On left, $t=-0.002$; center, $t=0$; right, $t=0.002$.}
\label{fig:three}
\end{figure}

Note that in this example, if $\delta>0$ fixed and $0< t\ll \delta$, then $q_2(t, -\delta)$ is  very near the negative real axis, and
$q_2(t,\delta )$ is very near the positive real axis. Thus, the corresponding vortex centers $\pm\sqrt{-q_2}$ are close to the real and imaginary axes respectively. The vortex curves connect the centers at $y_3 = -\delta$ to those at $y_3=\delta$, and the way this connection occurs changes at $t=0$. This reflects the change of sign \eqref{Im.change} of the imaginary part of the moduli space curve at $t=0$: if $\Im(q_2(t,y_3))<0$, then $\sqrt{-q_2(t,y_3)}$ lies in the first or third quadrant, and if $\Im(q_2(t,y_3))>0$, then $\sqrt{-q_2(t,y_3)}$ lies in the second or fourth quadrant.

The curves shown in Figure \ref{fig:three} are the projections onto the horizontal space of the 3d vortex centerlines as defined above. The corresponding vortex centertlines are pictured in Figure \ref{fig:four} below.

Informally, this change in connectivity is what is meant by ``reconnection". Although the  map $q$ will not be exactly as pictured, the occurence of reconnection in this example is a consequence only of \eqref{Im.change}, which holds for the $q$.

\begin{figure}[ht!]
\centering
\includegraphics[scale=1.5]{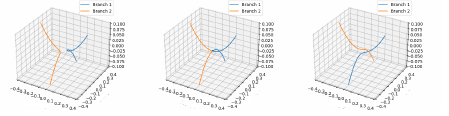}
\caption{3d vortex centerlines for $U^0$ corresponding to the moduli space curves shown above: On left, $t=-0.002$; center, $t=0$; right, $t=0.002$.}
\label{fig:four}
\end{figure}
\end{example}

In the above example, for $|t|$ small we can (arbitrarily) choose to say that $z_1(\cdot, t)$ is the vortex centerline (that is, the branch of $\sqrt{-q_2(\cdot, t)}$\,) starting with negative real part 
at $(t, -\delta)$. If $\delta$ is small enough, this is well-defined for all $0< |t| \le \delta$ and $|y_3|<\delta$, and
\beq\label{rct0}
\lim_{t\nearrow 0} z_1(t, y_3)\ne \lim_{t\searrow 0}z_1(t, y_3) \quad \mbox{ for } y_3>0.
\eeq
Indeed, it is clear from Figure \ref{fig:three} that these limits lie on the negative and positive imaginary axes, respectively. Property \eqref{rct0} is characteristic of reconnection at time $t_0$ and height $y_3$.

It is convenient to reformulate condition \eqref{rct0}. For small $\delta>0$, assume that
\begin{itemize}
\item $I = [\tau_0,\tau_1]$ is an interval, and
$\gamma: I\to \{(t, y_3): |t|\le\delta, |y_3|\le \delta\}$ is continuous
\item $\gamma(\tau) \ne (0,0)$ for $\tau\in I$, and
\item $z_1$ is a root of $p(\cdot; q(\gamma(\tau_0)))$
\end{itemize}
Then there is a unique continuous map $\zeta[\gamma, z_1]: I\to \C$ such that 
\begin{equation}\label{zeta.def}
\zeta[\gamma, z_1](\tau_0)
= z_1,\qquad \mbox{and }\quad \zeta[\gamma, z_1](\tau) \mbox{ is a root of $p(\cdot; q(\gamma(\tau)))$ for all $\tau\in I$}.
\end{equation}
In fact, $\zeta[\gamma, z_1](\tau)$ is just the continuous branch of $\sqrt{-q_2(\gamma(\tau))}$ that starts at $z_1$.

Moreover, \eqref{rct0} is equivalent to the condition that for the closed loop $\gamma$ pictured in Figure \ref{fig:five}, starting and ending at the point $(0,y_3)$ with $y_3>0$,
\beq\label{rcloop}
z_1 = \zeta[\gamma, z_1](\tau_0)\ne \zeta[\gamma, z_1](\tau_1).
\eeq
when $\ep$ (indicated in the figure) is small enough.

\begin{figure}[ht!]
\centering
\includegraphics[scale=0.4]{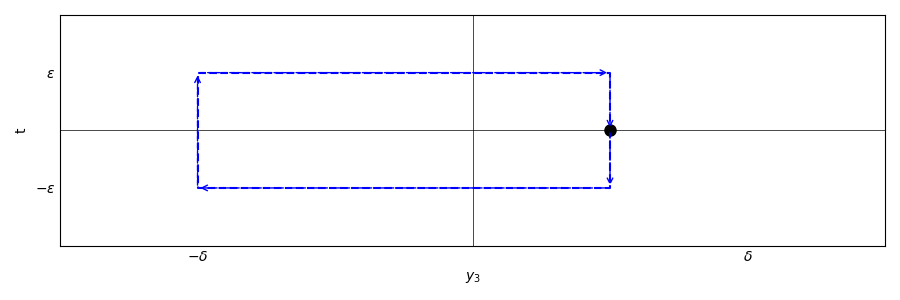}
\caption{The curve $\gamma(\tau), \tau_1\le \tau \le \tau_2$ in \eqref{rcloop}, starting at the marked spot at $\tau = \tau_1$ and ending at the same spot at $\tau = \tau_2$.}
\label{fig:five}
\end{figure}

Note also that if $\tilde \gamma:\tilde I\to M_2$ is any other path that starts and ends at $\gamma(\tau_0)$ and circles the origin once in a clockwise direction,
then
\[
\zeta[\gamma, z_1](\tau_0) - \zeta[\gamma, z_1](\tau_1) = 
\zeta[\tilde \gamma, z_1](\tilde\tau_0) - \zeta[\tilde \gamma, z_1](\tilde\tau_1)\ne 0.
\]
Thus, the existence of a loop $\gamma$ satisfying \eqref{rcloop} may be taken as a sufficient to guarantee reconnection, regardless of whether the loop has the specific form shown in Figure \ref{fig:five}.

\begin{example}\label{example2}
Let $q = (q_1, q_2): \R\times (0,T) \to \C^2\cong M_2$ solve \eqref{dwm} and define
$D:(-T,T)\times \R \to \C$ by $D = q_1^2 - 4q_2$. Thus $D(t,y_3)$ is the discriminant of $p(\cdot; q(t,y_3))$,
and 
\[
\mbox{the roots of $p(\cdot; q(t,y_3))$ are } \ \frac 12 \left(- q_1(t,y_3) \pm \sqrt {D(t, y_3)} \right)
\]
Now assume the $q$ there is some $(\bar x_3, \bar t)\in  \R\times (0,T) $ such that 
\beq\label{r.open}
D(\bar x_3, \bar t)= 0 ,\qquad \{ \pp_t D(\bar x_3, \bar t), \pp_{y_3}D(\bar x_3, \bar t)\}\mbox{ are linearly independent over $\R$}.
\eeq
Then by the implicit function theorem, $D$ is a diffeomorphism from an open neighborhood $W$ of $(\bar x_3, \bar t)\in \R\times (0,T)\times \R$ 
onto a neighborhood of $0\in \C$. We may assume that $W$ is an  open ball. Let $\gamma: I \to W$ be a closed loop that encircles $(\bar x_3, \bar t)$ once, with $\gamma(\tau_0)= \gamma(\tau_1)$. For example, we can take $\gamma(\tau) = (r \cos\tau, r\sin \tau)$ for $\tau\in [0,2\pi]$ and  $r$ sufficiently small. We can also take a loop of the form shown in Figure \ref{fig:five}.

Fixing a root $z_1$ of $p(\cdot; q(\gamma(\tau_0)))$, there is a unique map $\zeta = \zeta[\gamma,z_1]:I\to \C$ satisfying \eqref{zeta.def}, that is, a continuous branch of $\tau\mapsto -q_1(\gamma(\tau)) + \sqrt{D(\gamma(\tau)})$ such that $\zeta[\gamma, \tau_0] = z_1$. Since $D(\gamma(\tau))$ circles the origin exactly once, it is clear that
the chosen branch of $\sqrt{D(\gamma(\tau)}$ changes sign as $\tau$ increases from $\tau_0$ to $\tau_1$, and hence that \eqref{rcloop} holds.

Thus \eqref{r.open} is sufficient to guarantee reconnection of the vortex centerlines of the associated leading-order approximate solution $\wU^0$.
    
\end{example}

We now present the

\begin{proof}[proof of Theorem \ref{thm:3}]
Let $I=[\tau_0,\tau_1]$ be an interval, and $\gamma: I\to \R\times (0,T_0)$ be continuous, and let
$\gamma_\ep(\tau) = \frac 1\ep\gamma(\tau)$. We remark that the definitions imply that $q_\ep(\gamma_\ep(\tau))=q(\gamma(\tau))$ for all $\tau$.

Assume that $\zeta_{1,\ep}$ is a zero of $\Phi_\ep(\cdot, \gamma_\ep(\tau_0))$. We will then write
$\zeta_\ep[\gamma_\ep; z_{1,\ep}]:I\to \C$ to denote a continuous map, if it exists,
such that $\zeta_\ep[\gamma_\ep; z_{1,\ep}](\tau_0)  = z_{1,\ep}$ and
\[
\zeta_\ep[\gamma_\ep; z_{1,\ep}](\tau)  \mbox{ is a zero of }\Phi_\ep(\cdot ; \gamma_\ep(\tau)).
\]
Thus, $\zeta_\ep[\gamma_\ep; z_{1,\ep}]$ may be described as  a lifting of $\gamma_\ep$ to the vortex centers of $U_\ep(\cdot, \gamma_\ep)$.

We will prove that for any $\delta>0$ there is a continuous 
\[
\gamma: I = [\tau_0,\tau_0]\to \{ (x_3,t) : |x_3-\bar x_3|<\delta, |t-\bar t|<\delta \}
\]
such that $\gamma(\tau_0)=\gamma(\tau_1)$, and for all sufficiently small $\ep>0$ and each zero $z_{1,\ep}$ of $\Phi_\ep(\cdot, \gamma_\ep(\tau_0))$,
there is a {\em unique} continuous lifting $\zeta_\ep[\gamma_\ep; z_{1,\ep}]$, and moreover
\beq\label{reconn.def}
\zeta_\ep[\gamma_\ep; z_{1,\ep}](\tau_0) \ne \zeta_\ep[\gamma; z_{1,\ep}](\tau_1).  
\eeq
As discussed in the examples above, this guarantees reconnection of the vortex filaments of $\Phi_\ep$.

Our starting point is Example 2.
Given $\delta>0$, we require that the set $W$ from Example \ref{example2} be contained in 
\[
\{ (x_3, t): |x_3-\bar x_3|<\delta, \ \ |t-\bar t|<\delta, t < T_0\}.
\]
As in that example, we let $\gamma:I\to W$ be a path that encircles $(\bar x_3, \bar t)$ once, and we let 
$\zeta[\gamma, z_1]:I\to \C$ be the unique continuous map satisfying \eqref{zeta.def}.
We also define 
\[
r :=  \frac 13 \min_{\tau\in I} \sqrt{|q_2(\tau)|} >0,
\]
so that for every $\tau\in I$, the distance between the two roots of $p(\cdot ; q(\gamma(\tau)))$, equivalently the two zeros
of $\mphi(\cdot; q(\gamma(\tau)))$,  is at least $3r$. By choosing $r$ smaller if necessary and appealing to compactness and the explicit description of $\mphi(q)$ provided by Theorem \ref{thm1.new}, we can arrange that $\mphi(z; q(\gamma(\tau)))$ is uniformly nondegenerate in
\[
\{ z\in \C : \mbox{dist}(z, \mbox{roots of }p(\cdot, q(\gamma(\tau))) \} \le r
\]
 in the sense that if we view $\mphi$ as a map from $\R^2$ to $\R^2$, there is some positive constant $c$ such that $\det D\mphi(z ; q(\gamma(\tau))) \ge c$ for $z$ in the above set, for all $\tau\in I$.
 
Since we know from Theorem \ref{thm:1} that
\[
\| \Phi_\ep - \mphi(q_\ep)\|_{L^\infty_{T_0/\ep}W^{2,\infty}(\R^d)} 
+
\kz \| \pp_t(\Phi_\ep - \mphi(q_\ep))\|_{L^\infty_{T_0/\ep}W^{1,\infty}(\R^d)} \le C \ep^2,
\]
it is easy to see that for all sufficiently small $\ep$ and for every $\tau$, $\Phi_\ep(\cdot; q(\gamma_\ep(\tau)))$ has a unique zero, say $\zeta_\ep[\gamma_\ep; z_{1,\ep}](\tau)$ such that 
\beq\label{apart}
|\zeta_\ep[\gamma_\ep; z_{1,\ep}](\tau) - \zeta[\gamma; z_{1}](\tau)|<r.
\eeq
with $z_{1,\ep} = \zeta_\ep[\gamma_\ep; z_{1,\ep}]( \tau_0)$, and that the path defined in this way is continuous.
In addition, \eqref{apart} implies that
\beq\label{precon}
|\zeta_\ep[\gamma_\ep; z_{1,\ep}]( \tau_1) - \zeta_\ep[\gamma_\ep; z_{1,\ep}]( \tau_0)|
\ge
|\zeta[\gamma_\ep; z_{1}]( \tau_1) - \zeta[\gamma_\ep; z_{1}]( \tau_0)| - 2r. 
\eeq
And because $\gamma(\tau_1) = \gamma(\tau_0)$, we see that $\zeta[\gamma; z_{1}]( \tau_0), \zeta[\gamma; z_{1}]( \tau_1)$ are just the two roots of $p(\cdot, q(\gamma(\tau_0)))$, and hence are at least distance $3r$ apart. Thus \eqref{reconn.def} follows from \eqref{precon}
\end{proof}

\appendix

\section{Some linear estimates}\label{appendix}

\subsection{A scalar elliptic PDE}\label{app:A1}

In this subsection we consider the PDE
\beq\label{scalar.Rd}
(-\Delta_a + |\mphi(q)|^2)\psi = \eta \qquad\mbox{ on }\R^d\times (0,T).
\eeq
As in the proof of Proposition \ref{prop:approx.sol}, where the results we prove here are used, we will write points in $\R^d\times (0,T)$ as $y=(y_1,\ldots, y_d, y_0) = (y_a, y_{\bar a}, y_0)$
where $y_a = (y_1, y_2)$ and $y_{\bar a} = (y_3,\ldots, y_d)$.
We use the notation
\[
\Delta_a =  \pp_a\pp_a = \pp_{y_1}^2+\pp_{y_2}^2.
\]
We assume that $q:\R^{d-2}\times (0,T)\to M_N$ is a map such that 
\beq\label{qcompact}
\mbox{Image}(q)\subset K \mbox{ compact }\subset M_N.
\eeq
and
\beq\label{Thetabis}
\Theta_q :=  \| \pp_{\bar a}q( \cdot, y_0)\|_{L^\infty_T H^L(\R^{d-2})} +  \kz \| \pp_0 q(\cdot, y_0)\|_{L^\infty_T H^L(\R^{d-2})}  <\infty
\eeq
for some $L>d/2$.
We write $\mphi(q)$ to denote either
\[
\mphi(q):\R^d\times (0,T)\to \C\quad\mbox{ defined by }\mphi(q)(y) = \mphi(y_a; q(y_{\bar a}, y_0))
\]
as in \eqref{scalar.Rd}, or $\mphi(q):\R^2\to \C$ for a fixed $q\in M_N$, as
in Lemma \ref{scalar.lem1} below. What we have in mind should always be clear from the context.

Note that \eqref{scalar.Rd} may be understood as a $d-1$-parameter family of equations on $\R^2$:
\[
(-\Delta_a + |\mphi(\cdot ; q(y_{\bar a}, y_0))|^2)\psi(\cdot, y_{\bar a}, y_0) = \eta (\cdot, y_{\bar a}, y_0) \qquad \mbox{on $\R^2$ }
\]
for every $(y_{\bar a}, y_0)\in \R^{d-2}\times (0,T)$. It may also be viewed as an extremely degenerate elliptic equation on $\R^d\times (0,T)$.

We recall that the space $X^{k,j}_\gamma$ appearing in the statement below is
defined as $H^k \cap W^{k-j, \infty}_\gamma$, see \eqref{weightednorm}, \eqref{Xkj}.  This space appears naturally in our construction of an approximate solution.

\begin{proposition}\label{scalar.pde}
Fix $j,k$ such that $\frac d2  < j+1 \le k \le L$, and assume that 
$\eta: \R^d\times (0,T)\to \R$ satisfies
\beq\label{Thetaeta}
\Theta_\eta :=  \| \eta \|_{L^\infty_T X^{k,j}_\gamma(\R^{d})} +  \kz \| \pp_0 \eta \|_{L^\infty_T X^{k-1, j}_\gamma(\R^{d-2})}  <\infty
\eeq
for some $\gamma\in (0,1)$.

Then \eqref{scalar.Rd} has a unique solution $\psi$, and there exists a constant $C$, depending on $K, \Theta_q$ and $\gamma$, such that
\beq\label{psi.Hk}
 \| \psi \|_{L^\infty_T X^{k,j}_\gamma(\R^{d})} +  \kz \| \pp_0 \psi \|_{L^\infty_T X^{k-1, j}_\gamma(\R^{d})}  \le  C \Theta_\eta.
\eeq
\end{proposition}


We start with a 2d lemma.

\begin{lemma}\label{scalar.lem1}
For  compact $K\subset M_N$, given $q\in K$ and $\eta\in L^2(\R^2)$, the equation
\begin{equation}\label{scalar.d2}
(-\Delta + |\mphi(q)|^2)\psi = \eta 
\end{equation}
has a unique solution $\psi\in H^{2}(\R^2)$, and there exists a constant $C = C(K) <\infty$
such that
\begin{equation}\label{scalar.lem1.c}
\|\psi\|_{H^{2}(\R^2)} \le C \|\eta\|_{L^2(\R^2)}.
\end{equation}
In addition, if 
\[
\| \eta\|_{L^\infty_\gamma(\R^2)} 
:= \operatorname{ess\,sup_{y_a\in \R^2}} ( |\eta(y_a)a|e^{\gamma|y_a|}) < \infty
\]
for some $\gamma\in (0,1)$, then
$\psi\in L^\infty_\gamma(\R^2)$, and 
\beq\label{psi.weighted.basic}
\| \psi\|_{ L^\infty_\gamma} \le C \| \eta\|_{L^\infty_\gamma}
, \qquad C = C(K,\gamma).
\eeq
\end{lemma}

We sketch the rather standard proof for the reader's convenience:

\begin{proof}
Existence of a unique weak solution $\psi\in H^1(\R^2)$ follows directly from  Lemma \ref{basic.elliptic}, via standard arguments using the Riesz representation theorem. It also follows from Lemma \ref{basic.elliptic} that 
\[
\|\psi\|_{H^1} \le C \|\eta\|_{L^2(\R^2)}.
\]
The \eqref{scalar.lem1.c} follows by standard elliptic regularity.

Finally, assume that $\eta\in W^{k,\infty}_\gamma$ for some $\gamma\in (0,1)$. Then there exists some $A>0$ and $\gamma\in (0,1)$ such that
\begin{equation}\label{eta.decay}
|\eta(y_a)| \le A e^{-\gamma|y_a|}\qquad\mbox{ for all }y_a\in \R^2.
\end{equation}
The embedding  $H^2(\R^2)\hookrightarrow C_0(\R^2)$ implies that $\psi(y_a)\to 0$ as $|y_a|\to \infty$,with
\begin{equation}\label{psi.bound}
\| \psi\|_{L^\infty(\R^2)} \le C \| \psi\|_{H^2(\R^2)} \le  C \|\eta \|_{L^2(\R^2)} \overset{\eqref{eta.decay}}\le C A \end{equation}
for $C = C(K,\gamma)$.
Exponential decay of $\psi$ 
\begin{equation}\label{exp.decay1}
|\psi(y_a)| \le C(K,\gamma) A e^{-\gamma|y_a|}\qquad\mbox{ for all }y_a\in \R^2.
\end{equation}
is then proved by a comparison principle argument using the auxiliary function
\[
s(y_a) = C(K,\gamma) A e^{-\gamma|y_a|}
\]
which, one can check, for suitable choices of $C(K,\gamma)$ and $R(K,\gamma)$ satisfies
\[
(-\Delta_a + |\mphi(q)|^2)(s \pm u)(y_a )> 0 \qquad\mbox{ for }|y_a|\ge R(K,\gamma).
\]
This uses assumption \eqref{eta.decay} as well as the exponential decay of $(1-|\mphi(q)|)$, see \eqref{eq10}. A detailed presentation of this argument can be found in \cite{Masoud-thesis}, Proposition 29. Then \eqref{exp.decay1} is a consequence of this fact and \eqref{psi.bound} by the maximum principle, and \eqref{psi.weighted.basic} follows.
\end{proof}

We now present the 

\begin{proof}[proof of Proposition \ref{scalar.pde}]
Existence and uniqueness of a weak solution follows by applying Lemma \ref{scalar.lem1} on $\R^2 \times \{ ( y_0,  y_{\bar a}) \}$
for a.e. $( y_{\bar a}, y_0)\in \R^{d-2}\times (0,T)$. 

Next, we fix $y_0\in (0,T)$, apply \eqref{scalar.lem1.c} on a.e. $y_{\bar a}\in \R^{d-2}$, square both sides, and integrate with respect to $y_{\bar a}\in \R^{d-2}$. At this $y_0$ we obtain
\beq\label{L2L2Rd}
\begin{aligned}
    \| \psi \|_{L^2(\R^d)}^2 + \| \nabla_a\psi \|_{L^2(\R^d)}^2+ \|\nabla_a^2 \psi \|_{L^2(\R^d)}^2 
    &\le \int_{\R^{d-2}}\|\psi\|^2_{H^2(\R^2\times \{y_{\bar a}\})}dy_{\bar a} \\
&\le C\| \eta\|_{L^2(\R^d)}^2 \quad\mbox{ when \eqref{scalar.Rd} holds.}
\end{aligned}
\eeq

Now let $r$ be a multiindex on $\R^d$ such that $|r| = l \le k$.
We first assume that $\eta$ and $q$ are smooth,  and hence $\psi$ is smooth.
We can then differentiate \eqref{scalar.d2} to find that
\beq\label{Dalpha.eqn}
\begin{aligned}
(-\Delta_a + |\mphi(q)|^2)\pp^r\psi 
&= \pp^r\eta +  |\mphi(q)|^2 \pp^r\psi - \pp^r(|\mphi(q)|^2 \psi)
\end{aligned}
\eeq
We now estimate the $L^2$ norm of the right-hand side. We use Lemma \ref{lem:fq} below
to estimate
\[
\begin{aligned}
\|  |\mphi(q)|^2 \pp^r\psi - \pp^r(|\mphi(q)|^2 \psi) \|_{L^2(\R^d)}
&\le
C \sum_{r_1+r_2 = r, |r_1|\ge 1} \| \pp^{r_1}\mphi(q) \pp^{r_2}\psi\|_{L^2(\R^d)}
\\
&
\overset{\eqref{f(q)psi}}\le
C(q) \| \psi\|_{H^{l-1}(\R^d)}.
\end{aligned}
\]
This combined\footnote{We are not using the full strength of \eqref{L2L2Rd} because we will not need it, and to simplify notation in what follows.} with \eqref{L2L2Rd} and \eqref{Dalpha.eqn} implies that
\[
\| \pp^r \psi\|_{L^2(\R^d)}
\le \| \pp^r\eta\|_{L^2(\R^d)} + C\| \psi\|_{H^{l-1}(\R^d)}.
\]
Since this holds for all $r$ with  $|r|=l$, it follows that
\[
\|  \psi\|_{H^l(\R^d)} 
\le C\| \eta\|_{H^l(\R^d)} + C\| \psi\|_{H^{l-1}(\R^d)}.
\]
for $l\le k$. Then an easy induction argument implies that
\[
\|  \psi\|_{H^l(\R^d)} \le C \| \eta \|_{H^l(\R^d)}.
\]
Since this holds for all $l\le k$ and for all $t\in (0,T)$, 
\beq\label{LinfHk}
\|  \psi\|_{L^\infty_T H^k(\R^d)} \le C \| \eta \|_{L^\infty_T H^k(\R^d)}.
\eeq

We next claim that if $\eta\in L^\infty_T W^{k-j,\infty}_\gamma$
then $\psi$ belongs to the same space and
\beq\label{LinfWk-j}
\|  \psi\|_{L^\infty_T W^{k-j,\infty}_\gamma(\R^d)} \le C \| \eta \|_{L^\infty_T W^{k-j,\infty}_\gamma(\R^d)}.
\eeq
The proof is very similar to that of \eqref{LinfHk}:
we  apply the 2d estimate \eqref{psi.weighted.basic} to \eqref{Dalpha.eqn} on
$\R^2 \times \{(y_a, y_0)\}$ for {\em a.e.} $(y_{\bar a}, y_0)$, using an induction argument and the exponential decay of derivatives of $\eta$,  {\em i.e.} the finiteness of the right-hand side of \eqref{LinfWk-j}, to verify that the right-hand side satisfies \eqref{eta.decay}. 

Combining \eqref{LinfHk} and \eqref{LinfWk-j}, we obtain
\beq\label{LinfXkj}
\|  \psi\|_{L^\infty_T X^{k,j}_\gamma(\R^d)} \le C \| \eta \|_{L^\infty_T X^{k,j}_\gamma(\R^d)}.
\eeq

Finally, note that $\pp_0\psi$ satisfies
\[
(-\Delta_a + |\mphi(q)|^2)\pp_0\psi =
\pp_0\eta - \pp_0|\mphi(q)|^2 \psi \in L^\infty_TX^{k-1,j}_\gamma \quad\mbox{ if }\kz>0.
\]
Applying \eqref{LinfXkj} to this equation, we obtain the conclusion \eqref{psi.Hk} of the Proposition.

If $\eta$ and $q$ have only the regularity assumed in the theorem, equation \eqref{Dalpha.eqn} can be justified by an approximation argument or by difference quotients and induction, and the rest of the proof is unchanged.
\end{proof}

\subsection{An elliptic system}\label{app:system}

In this subsection we consider the system
\beq\label{system.Rd}
L[\modu(q)] \tilde u = \eta \qquad\mbox{ on }\R^d\times (0,T)
\eeq
for $\tilde u, \eta:\R^d\times (0,T)\to \C\times \R^2$, where $q:\R^{d-2}\times (0,T)\to M_N$ is a map satisfying \eqref{qcompact} and \eqref{Thetabis}.
One may understand \eqref{system.Rd} as a $d-1$-parameter family of equations on $\R^2$:
\[
L[\modu(q)] \tilde u = \eta\qquad\qquad\mbox{on $\R^2_{(y_{\bar a}, y_0)} := \R^2\times \{(y_{\bar a}, y_0)\}$}
\]
for every $(y_{\bar a}, y_0)\in \R^{d-2}\times (0,T)$.
The natural solvability condition for the equation is
\beq\label{solvability}
\eta \perp \ker(L[\modu(q)]). 
\eeq
In view of Stuart's Theorem \ref{thm:hessian}, this can be written more explicitly as 
\beq\label{svbl2}
(\eta, n_\mu(q))_{L^2(\R^2_{(y_{\bar a}, y_0)})} = 0 \quad\mbox{ for all }\mu = 1,\ldots,2N \mbox{ and  }(y_{\bar a}, y_0)\in \R^{d-2}\times (0,T).
\eeq
Below is our main result about this system.

\begin{proposition}\label{system.pde}
Fix $j,k$ such that $\frac d2<j+1\le k \le L$, and assume that $\eta:\R^d\times (0,T)\to  \C\times \R^2$
satisfies  \eqref{solvability} and \eqref{Thetaeta} for some $\gamma\in (0, 1/2)$.

Then \eqref{system.Rd} has a unique solution $\tilde u = \binom{\tPhi}{\tA}$ with the property that
\beq\label{system.gauge1}
\tilde u \perp  \ker(L[\modu(q)])
\eeq
and there exists a constant $C$, depending on $K, \Theta_q$ and $\gamma$, such that
\beq\label{eq:spc}
\| \tilde u\|_{L^\infty_T X^{k,j}_\gamma(\R^d)} +  \kz \| \pp_0 \tilde u\|_{L^\infty_T X^{k-1,j}_\gamma(\R^d)} 
\le C \Theta_\eta.
\eeq
Moreover, if $\eta(\cdot, y_{\bar a}, y_0)$ satisfies the gauge-orthogonality condition at $\modu(q(y_{\bar a}, y_0))$ for every $(y_{\bar a}, y_0)$:
\[
\pp_a \eta_a(y) -(i\mphi(y_a; q(y_{\bar a}, y_0)), \eta_\phi(y) = 0
\]
then so does $\tilde u$:
\beq\label{system.gauge2}
\pp_a \tA_a(y) - (i\mphi(y_a ;  q(y_{\bar a}, y_0)), \tPhi(y)) = 0.
\eeq
\end{proposition}

We start with a 2d lemma.

\begin{lemma}\label{vsolvability1}
Let $K$ be a compact subset of $M_N$. Assume that $q\in K$, and that $\eta\in L^2(\R^2; \C\times \R^2)$ satisfies
\[
(\eta, n_\mu(q))_{L^2(\R^2)} = 0\qquad\mbox{ for }\mu = 1,\ldots, 2N
\]
Then there exists a unique $\tilde u = \binom{\tPhi}{\tA_a}\in H^2(\R^2; \C\times \R^2)$ satisfying 
\beq\label{systemR2}
L[\modu(q)] (\tilde u) = \eta, \qquad\qquad (\tilde u, n_\mu(q))_{L^2(\R^2)} = 0\ \ \mbox{ for }\mu = 1,\ldots, 2N.
\eeq
Moreover, $\tilde u$ satisfies
\begin{equation}\label{system.lem1.c}
\| \tilde u \|_{H^2(\R^2)} \le C \| \eta \|_{L^2(\R^2)}
\end{equation}
and if $\eta$ satisfies the gauge-orthogonality condition at $\modu(q)$
\[
\pp_a \eta_a -(i\mphi(q), \eta_\phi) = 0
\]
then so does $\tilde u$:
\beq\label{tilde.go}
\pp_a \tA_a - (i\mphi(q), \tPhi) = 0.
\eeq
Finally, if 
\beq\label{vector.expd1}
\| \eta\|_{L^\infty_\gamma(\R^2)}
:= \operatorname{ess\,sup_{y_a\in \R^2}} ( |\eta(y_a)|e^{\gamma|y_a|}) < \infty
\eeq
for some $\gamma\in (0,1/2)$, then
$\tilde u\in L^\infty_\gamma(\R^2)$, and 
\beq\label{vector.expd2}
\| \tilde u\|_{ L^\infty_\gamma} \le C \| \eta\|_{L^\infty_\gamma}
, \qquad C = C(K,\gamma).
\eeq
\end{lemma}

\begin{proof}
Existence of a unique weak solution $\tilde u$ satisfying the orthogonality condition $(\tilde u, n_\mu(q))_{L^2(\R^2)} = 0$
for $\mu = 1,\ldots, 2N$ follows directly from Theorem \ref{thm:hessian}
via standard arguments using the Riesz representation theorem in the Hilbert space
\[
\{ \tilde u \in H^1(\R^2; \C\times \R^2) \ : \ \ (\tilde u, n_\mu(q))_{L^2(\R^2)} = 0\mbox{ for all }\mu=1,\ldots, 2N\}
\]
equipped with the $H^1$ norm. This solution satisfies $\| \tilde u \|_{H^1(\R^2)} \le C \| \eta \|_{L^2(\R^2)}$. Basic elliptic regularity implies that $\| \tilde u \|_{H^2(\R^2)} \le C( \| \eta \|_{L^2(\R^2)} + \|\tilde u\|_{L^2(\R^2)})$, and then \eqref{system.lem1.c} follows.

Next, assume that $\eta$ satisfies the gauge-orthogonality condition, which
implies that
\[
\int_{\R^2} \eta \cdot  \binom{i\mphi(q) \xi}{\pp_a \xi} \ = 0 \qquad\mbox{ for all }\xi\in H^1(\R^2).
\]
It follows from \eqref{systemR2} that
\[
\int_{\R^2} L[\modu(q)](\tilde u) \cdot  \binom{i\mphi(q) \xi}{\pp_a \xi} \ = 0 \qquad\mbox{ for all }\xi\in H^1(\R^2).
\]
Then \eqref{Lu.def} implies that 
\[
\int_{\R^2} \calL[\modu(q)](\tilde u) \cdot  \binom{i\mphi(q) \xi}{\pp_a \xi} \ = -
\int_{\R^2} \binom{i\mphi(q)(\nabla\cdot \tilde A - (i\mphi(q),\tilde \Phi) } {\pp_a (\nabla\cdot \tilde A - (i\mphi(q),\tilde \Phi)}
 \cdot  \binom{i\mphi(q) \xi}{\pp_a \xi} .
\]
for all $\xi\in H^1(\R^2)$. But the left-hand side vanishes, by \eqref{bigkernel} and the fact that $\calL[\modu(q)]$ is self-adjoint, which follows from the fact that $\calL$ is the second variation of an energy functional.
So 
\[
\int_{\R^2} \binom{i\mphi(q)(\nabla\cdot \tilde A - (i\mphi(q),\tilde \Phi) } {\pp_a (\nabla\cdot \tilde A - (i\mphi(q),\tilde \Phi)}
 \cdot  \binom{i\mphi(q) \xi}{\pp_a \xi} =0\ .
\]
Taking $\xi = \pp_a\tilde A_a - (i\mphi(q), \tilde \Phi)$, we conclude that
\[
\int_{\R^2} |\nabla\xi|^2 +|\mphi(q)|^2\xi^2 = 0
\]
and hence (by Lemma \ref{basic.elliptic}) that $\xi = 0$, which says exactly that  $\tilde u$ is gauge-orthogonal.

Finally, assume that $\eta$ satisfies \eqref{vector.expd1}. The embedding  $H^2(\R^2)\hookrightarrow C_0(\R^2)$ implies that $\tilde u(y_a)\to 0$ as $|y_a|\to \infty$, with
\begin{equation}\label{tu.Linf}
\| \tilde u\|_{L^\infty(\R^2)} \le C \| \tilde u\|_{H^2(\R^2)} \le  C \|\eta \|_{L^2(\R^2)} \le C \| \eta \|_{L^\infty_\gamma(\R^2)}
\end{equation}
for $C = C(K,\gamma)$.
Recalling that $\pp_a\mA_a(q)=0$, we can rewrite the system $L[\modu(q)](\tilde u)$ 
in the form
\beq\label{Lu.rewrite}
\begin{aligned}
-\Delta \tPhi + 2i \mA_a(q)\pp_a\tPhi + \left( \frac 12(3|\mphi(q)|^2-1) + |\mA(q)|^2\right)\tPhi 
&= \eta_\phi - 2i(D_{\mA_a} \mphi)(q)\tA_a  \\
-\Delta \tA_a + |\mphi(q)|^2 \tA_a &= \eta_a + 2 (i\tPhi, D_{\mA_a}\mphi(q)).
\end{aligned}
\eeq
It follows from \eqref{eq10} that $| D_{\mA_a}\mphi(y_a;q)|\le Ce^{-\gamma|y_a|}$, so
the desired exponential decay of $\tilde A_a$ for $a=1,2$ is a consequences of Lemma \eqref{scalar.lem1}, with the same decay rate as $\eta$.

Next, we derive a differential inequality satisfied by $|\tilde \Phi|^2$.
To do this we note that
\[
\Delta |\tilde \Phi|^2 = 2 (\Delta \tilde \Phi, \tilde \Phi) + 2 |\nabla\tilde\Phi|^2
\]
From \eqref{Lu.rewrite}, \eqref{eq10} and \eqref{vector.expd1}, we check that
\[
(\Delta\tPhi, \tPhi) \ge  \left( \frac 12(3|\mphi(q)|^2-1) + |\mA(q)|^2\right)|\tPhi|^2  + (2i \mA(q)_a\pp_a\tPhi , \tPhi) - C e^{-\gamma|y_a}|\tPhi|.
\]
Noting that
\[
|  (2i \mA(q)_a\pp_a\tPhi , \tPhi) | \le |\nabla\tPhi|^2 + |\mA(q)|^2\, |\tPhi|^2,
\]
we deduce that 
\[
\Delta |\tPhi|^2 \ge  (3|\mphi(q)|^2-1)|\tPhi|^2  -  C e^{-\gamma|y_a|}|\tPhi| .
\]
Since $\gamma<1/2$ and $|\mphi(y_a;q)|^2\to 1$ exponentially fast as $|y_a|\to\infty$, we can find $R=R(\gamma,K)$ such that 
$3|\mphi(q)|^2-1 > 1+\gamma $ whenever $|y_a|\ge R$. In this region,
\[
\Delta |\tPhi|^2 \ge  2\gamma |\tPhi|^2  + (1-\gamma) |\tPhi|^2 -  C e^{-\gamma|y_a|}|\tPhi|  \ge 2\gamma|\tPhi|^2 - C e^{-2\gamma|y_a|}.
\]
We now define $v = |\tPhi|^2 - K e^{-2\gamma|y_a|}$ for $K$ to be chosen below. We use the above inequality to compute
\[
\Delta v \ge 2\gamma|\tPhi|^2 - \left( C + 4\gamma^2 K \right) e^{-2\gamma|y_a|} = 2\gamma v +(2\gamma K - C - 4\gamma^2 K)e^{-2\gamma|y_a|}.
\]
Since $\gamma<1/2$, and recalling \eqref{tu.Linf}, we can choose $K$ large enough that
\[
\Delta v \ge 2\gamma v,\mbox{ when }|y_a|>R, \qquad \mbox{ and } v(y_a)<0 \mbox{ when }|y_a|=R.
\]
Then recalling that $\tilde u\to 0$ as $|y_a|\to \infty$, a maximum principle argument shows that $v\le 0$, and hence that \eqref{vector.expd2} holds, for $|y_a|\ge R$. The same inequality for $|y_a|\le R$ follows from \eqref{tu.Linf}, after adjusting constants if necessary. This proves \eqref{vector.expd2}.
\end{proof}

We now present the 

\begin{proof}[proof of Proposition \ref{system.pde}]
We follow the proof of Proposition \ref{scalar.pde}, with adjustments to account for complications such as the fact that $L[\modu(q)]$ has a nontrivial kernel.

Existence and uniqueness of a weak solution follows by applying Lemma \ref{vsolvability1} on $\R^2_{(y_{\bar a}, y_0 ) }$
for a.e. $(y_{\bar a}, y_0)\in \R^{d-2}\times (0,T)$. The assertion about gauge-orthogonality follows from applying the corresponding claim from Lemma \ref{vsolvability1} on $\R^2_{(y_{\bar a}, y_0)}$ for every $y_{\bar a}$.

Next we fix $y_0\in (0,T)$, and having done so we write for example $\R^2_{y_a}$ in place of $\R^2_{(y_{\bar a}, y_0)}$, and $\R^d$ instead of $\R^d\times \{y_0\}$. This $y_0$ will remain fixed
until we consider $\pp_0 \tilde u$, near the end of the proof.

We apply \eqref{system.lem1.c} on a.e. $y_{\bar a}\in \R^{d-2}$, square both sides, and integrate with respect to $y_{\bar a}\in \R^{d-2}$.  \beq\label{L2L2Rd.sys}
\int_{\R^{d-2}}\|\tilde u\|^2_{H^2(\R^2\times \{y_{\bar a}\})}dy_{\bar a}
\le C\| \eta\|_{L^2(\R^d)}^2 \quad\mbox{ when \eqref{system.Rd} holds.}
\eeq

Now let $r$ be a multiindex on $\R^d$ such that $|r| = l \le k$.
We first assume that $\eta$ and $q$ are smooth,  and hence $\tilde u$ is smooth.
We can then differentiate \eqref{system.Rd} to find that
\beq\label{Dalpha.eqn.sys}
\begin{aligned}
L[\modu(q)](\pp^r \tilde u) &= \pp^r\eta + L[\modu(q)](\pp^r \tilde u)  - \pp^r(L[\modu(q)](\tilde u))
\end{aligned}
\eeq
We write 
\beq\label{wperp.wparallel}
\pp^r \tilde u = w^\perp_r+ w^\parallel_r, \qquad\mbox{ where }
w^\perp_r \perp \ker(L[\modu(q)]).
\eeq
It is easy to see that 
\beq\label{wpar1}
w^\parallel_r(y) = (\pp^r \tilde u, n_\mu(q))_{L^2(\R^2_{y_{\bar a}})} n_\mu(q) 
\eeq
where $\{n_\mu(\cdot; q) \}$ is the orthonormal basis fixed in Definition \ref{def:nmu}, and we sum implicitly $\mu=1,\ldots, 2N$.
By differentiating the orthogonality condition \eqref{svbl2}, we find that
\beq\label{wparallel.split}
(\pp^r \tilde u, n_\mu(q))_{L^2(\R^2(y_{\bar a}))}
= 
\sum_{r_1+r_2=r, |r_2|\ge 1}\frac{r!}{r_1! r_2!} (\pp^{r_1}\tilde u, \pp^{r_2} n_\mu(q))_{L^2(\R^2(y_{\bar a}))}
\eeq
We use Lemma \ref{lem:fq} below to estimate
\begin{align}
\| w^\parallel_r\|_{L^2(\R^d)}^2
&= \int_{\R^{d-2}} \left(
(\pp^r \tilde u, n_\mu(q))_{L^2(\R^2(y_{\bar a}))}^2\int_{\R^2} |n_\mu(y_a; q(y_{\bar a}))|^2 dy_a \right) dy_{\bar a}
\nonumber \\
&= \int_{\R^{d-2}}(\pp^r \tilde u, n_\mu(q))_{L^2(\R^2(y_{\bar a}))}^2 dy_{\bar a}
\nonumber\\
&\le
 C \int_{\R^{d-2}}\sum_{r_1+r_2=r, |r_2|\ge 1}(\pp^{r_1}\tilde u, \pp^{r_2} n_\mu(q))_{L^2(\R^2(y_{\bar a}))}^2 dy_{\bar a} \nonumber\\
 &
\overset{\eqref{f(q)psi}}\le C \| \tilde u\|_{H^{l-1}(\R^d)}^2. \label{wparallel2}
\end{align}
In fact, since $\| n_\mu(q)\|_{H^k(\R^2)} \le C(k,K)$ for $q\in K$, the same argument shows that
\beq\label{wparallel3}
\|w^\parallel_r\|_{L^2(\R^d)}^2+\|\nabla_a  w^\parallel_r\|_{L^2(\R^d)}^2+\|\nabla_a^2  w^\parallel_r\|_{L^2(\R^d)}^2 \le C \|\tilde u\|_{H^{l-1}(\R^d)}^2,
\eeq
where $\nabla_a$ denotes the gradient $(\pp_{y_1}, \pp_{y_2})$ with respect to the $y_a$ variables. 

To control $w^\perp_r$, note first that $L[\modu(q)](w^\parallel_r) = 0$ by construction, so
\eqref{Dalpha.eqn.sys} implies that
\beq\label{Dalphasys2}
L[\modu(q)](w_r^\perp) = \pp^r\eta + L[\modu(q)](\pp^r \tilde u)  - \pp^r(L[\modu(q)](\tilde u)).
\eeq
We now estimate the $L^2$ norm of the commutator term on the the right-hand side. First note from the definition \eqref{eq33.14} that for every $q$,
\[
L_1[\modu(q)]  := L[\modu(q)]  - \Delta_a
\]
is a first-order differential operator involving only $\pp_{y_a}$ derivatives whose coefficients are smooth functions of $y_a$ and $q(y_{\bar a})$.  It is clear that
\beq\label{Dalp2}
L[\modu(q)](\pp^r \tilde u)  - \pp^r(L[\modu(q)](\tilde u))
= 
L_1[\modu(q)](\pp^r \tilde u)  - \pp^r(L_1[\modu(q)](\tilde u))
\eeq
The right-hand side is a sum of terms of the form
\beq\label{Dalp3}
\pp^{r_1} f(q) \pp^{r_2}\tilde u, \qquad
\pp^{r_1} f(q) \pp^{r_2} \pp_a\tilde u,
\eeq
with $r_1+r_2=r$ and $|r_1|\ge 1$, where $f(q)$ is a function of the type considered in Lemma \ref{lem:fq}.
Indeed, every $f(q)$ is is a low-degree (at most 2) polynomial in components of $\modu(y_a; q(y_{\bar a}))$ or their first $\pp_{y_a}$ derivatives,
and all components of $\modu(y_a; q)$ are smooth with respect to both
$y_a$ and $q$, as noted in Lemma \ref{lem:decayq}. Thus Lemma \ref{lem:fq} implies that every coefficient function $f(q)$ above
satisfies
\[
\| \pp^{r_1} f(q) \pp^{r_2}\tilde u \|_{L^2(\R^d)} \le C(\Theta) \| \tilde u\|_{H^{l-1}(\R^d)}.
\]
Thus, arguing as in the proof of Proposition \ref{scalar.pde}, we deduce that
\[
\| L[\modu(q)](\pp^r \tilde u)  - \pp^r(L[\modu(q)](\tilde u)) \|_{L^2(\R^d)}
\le
C(q) \left( \| \tilde u\|_{H^{l-1}(\R^d)} + \| \nabla_a \tilde u\|_{H^{l-1}(\R^d)} \right).
\]
Recalling \eqref{Dalphasys2}, the basic estimate \eqref{L2L2Rd.sys} implies that
\[
\begin{aligned}
&\| \nabla_a^2 w^\perp_r \|_{L^2(\R^d)}+\|\nabla_a w^\perp_r\|_{L^2(\R^d)}+\| w^\perp_r\|_{L^2(\R^d)}\\
&\hspace{10em}
 \le \| \pp^r\eta\|_{L^2(\R^d)} + C\| \tilde u \|_{H^{l-1}(\R^d)} + C\| \nabla_a\tilde u\|_{H^{l-1}(\R^d)}.
\end{aligned}
\]
We combine this with \eqref{wparallel3}. Since the resulting inequality
holds for all $r$ with  $|r|=l$, it follows that
\[
\begin{aligned}
&\| \tilde u \|_{H^{l}(\R^d)} + \| \nabla_a\tilde u\|_{H^{l}(\R^d)} + \| \nabla_a^2\tilde u\|_{H^{l}(\R^d)}\\
&\hspace{10em}
 \le \| \eta\|_{H^l(\R^d)} + C\| \tilde u \|_{H^{l-1}(\R^d)} + C\| \nabla_a\tilde u\|_{H^{l-1}(\R^d)}.
\end{aligned}\]
for $l\le k$. An easy induction argument, with the $l=0$ case supplied by \eqref{L2L2Rd.sys}, shows that
\[
\| \tilde u \|_{H^{l}(\R^d)}  \le C  \| \eta\|_{H^l(\R^d)} .
\]
Since this holds at arbitrary $y_0$ and for all $l\le k$, it follows that
\beq\label{tu.Linf.Hk}
\| \tilde u \|_{L^\infty_T H^{k}(\R^d)}  \le C  \| \eta\|_{L^\infty_T H^k(\R^d)} .
\eeq

We next claim that 
\beq\label{expd.ind}
\| \tilde u \|_{L^\infty_T W^{k-j,\infty}_\gamma(\R^d)}  \le C  \| \eta\|_{L^\infty_T W^{k-j,\infty}_\gamma(\R^d)} .
\eeq
We assume by induction that 
\[
\| \tilde u \|_{L^\infty_T W^{l-1,\infty}_\gamma(\R^d)}  \le C  \| \eta\|_{L^\infty_T W^{l-1,\infty}_\gamma(\R^d)} 
\]
for $1\le l \le j-k$.
For a multiindex $r$ such that $|r|=l$, we decompose $\pp^r\tilde u = w^\perp_r + w^\parallel_r$ as in \eqref{wperp.wparallel}.
In view of \eqref{wpar1}, \eqref{wparallel.split} and properties of $n_\mu$ from Lemma \ref{lem:decayq}, to estimate
$\| w^\parallel_r \|_{L^\infty_T L^\infty_\gamma(\R^d)}$ it suffices to show that if $|r_1|\le l-1$ then
\[
\mbox{ $\psi(y_0, y_{\bar a}) := (\pp^{r_1} \tilde u  , \pp^{r_2} n_\mu)_{L^2(\R^2_{(y_0,y_{\bar a})})}$ satisfies }\ 
\| \psi \|_{L^\infty_T L^\infty(\R^{d-2})} \le C  \| \eta\|_{L^\infty_T W^{l-1,\infty}(\R^d)} .
\]
But this follows immediately from the induction hypothesis and properties of $n_\mu$.
We then conclude that
\[
\| w^\parallel_r \|_{L^\infty_T L^\infty_\gamma(\R^d)}  \le C  \| \eta\|_{L^\infty_T W^{l-1,\infty}_\gamma(\R^d)} \qquad\mbox{ if }|r|=l.
\]

Next,  we write $w^\perp_r = (\tPhi^\perp_r, \tA^\perp_{r,a})$.
We write the equation satisfied by $w^\perp_r$ as in \eqref{Dalphasys2}, \eqref{Dalp2}, \eqref{Dalp3}. From the explicit form \eqref{Lu.rewrite} of $L[\modu(q)]$, we see that the $a=1,2$ components of this system can be written
\[
(-\Delta +|\mphi(q)|^2)\tA^\perp_{r,a} = \eta_{r,a} + \mbox{ linear combination of terms of the form } \pp^{r_1}f(r) \pp^{r_2}\tilde u
\]
with $|r_2|\le l-1$. Using the induction hypothesis, we check that 
\[
\| \mbox{right-hand side }   \|_{L^\infty_T L^\infty_\gamma (\R^d)}  \le C  \| \eta\|_{L^\infty_T W^{l,\infty}_\gamma(\R^d)} \ \le C  \| \eta\|_{L^\infty_T W^{l,\infty}_\gamma(\R^d)} 
\]
and then Lemma \ref{scalar.lem1} implies that
\[
\| \tA_{r,a}^\perp \|_{L^\infty_T L^\infty_\gamma(\R^d)}  \le C  \| \eta_r \|_{L^\infty_T W^{l,\infty}_\gamma(\R^d)} .
\]
The right-hand side is finite, since we have assumed that $l\le k-j$.

By inspection of \eqref{Lu.rewrite}, we see that the $\tPhi^\perp_r$ components solve an equation of the form
\begin{align}
&-\Delta  \tPhi^\perp_r  + 2i \mA(q)_a\pp_a \tPhi^\perp_r + \sum_{|s|=l} \Xi_{s} \tPhi^\perp_{s} +  \left( \frac 12(3|\mphi(q)|^2-1) + |\mA(q)|^2\right) \tPhi^\perp_r  \nonumber \\
&\hspace{6em}= \eta_{r,\phi} -\sum_{|s|=l} \Xi_s  \tPhi^\parallel_s
\label{tPhir}\\
&\hspace{10em} 
+  \mbox{ linear combination of terms of the form } \pp^{r_1}f(r) \pp^{r_2}\tilde u  \nonumber 
\end{align}
where $\sum_{s|=l} \Xi_s \tPhi^\perp_s$ and $-\sum_{s|=l} \Xi_s \tPhi^\parallel_s$ collect all terms with $l-1$ derivatives of $\pp_a\tPhi$ and one derivative of $A_a$ that arise when expanding
$\pp^r (2i \mA_a(q) \pp_a\tPhi)$. Thus, each $\Xi_s$ is either $0$ or has the form $C 2i \pp_j \mA_a(q)$ for some  $C, a,j$ and hence, by Lemma \ref{lem:decayq}, satisfies
\beq\label{needsoon}
|\Xi_s(y)| \le C(1+|y_a|)^{-2}.
\eeq
From the induction hypothesis and our earlier discussion of $w_r^\parallel$, the right-hand side of \eqref{tPhir} is bounded by
\[
\| \mbox{right-hand side}\|_{L^\infty_T L^\infty_\gamma(\R^d)} \le  \| \eta\|_{L^\infty_T W^{l,\infty}_\gamma(\R^d)}.
\]
We define $\Psi = (\tPhi^\perp_r)_{r=l}$. Multiplying \eqref{tPhir} by $\tPhi_r$, summing over multiindices $r$ such that $|r|=l$, using \eqref{needsoon} and
arguing as in the proof of Lemma \ref{vsolvability1}, we find that
\[
\Delta |\Psi|^2 \ge  (3|\mphi(q)|^2-1 - C (1+|y_a|)^{-2} )|\Psi|^2  -  C e^{-\gamma|y_a|}|\Psi| .
\]
One can then follow the proof of Lemma \ref{vsolvability1} to show that
\[
|\Psi|^2 \le K e^{-\gamma|y_a|}, \qquad\mbox{ with }K\le  \| \eta\|_{L^\infty_T W^{l,\infty}_\gamma(\R^d)}.
\]
This completes the induction argument  and hence proves \eqref{expd.ind}.

So far we have proved that 
\beq\label{almostdone}
\| \tilde u \|_{L^\infty_T X^{k,j}_\gamma(\R^d)} \le C \| \eta \|_{L^\infty_T X^{k,j}_\gamma(\R^d)}\ .
\eeq
when $L[\modu(q)]\tilde u= \eta$. By applying \eqref{almostdone} to the equation
\[
L[\modu(q)]\pp_0 \tilde u= \pp_0 \eta + L[\modu(q)]\pp_0 \tilde u - \pp_0(L[\modu(q)] \tilde u)
\]
and using \eqref{almostdone} to estimate the right-hand side. we obtain
\[
\| \pp_0 \tilde u \|_{L^\infty_T X^{k-1,j}_\gamma(\R^d)} \le C \| \pp_0 \eta \|_{L^\infty_T X^{k-1,j}_\gamma(\R^d)} + C \| \eta \|_{L^\infty_T X^{k,j}_\gamma(\R^d)}\ .
\]
This and \eqref{almostdone} immediately imply the conclusion \eqref{eq:spc} of the Proposition,
\end{proof}

\subsection{A nonlinear evolution equation}\label{app:hyp}

In this section we recall standard facts about the system of equations
equation
\beq\label{dhs.app}
(\kz\wpp_0\wpp_0 + \ko \wpp_0 - \wpp_{\bar a }\wpp_{\bar a})f_\nu
= rhs_\nu, \qquad \nu = 1,\ldots, 2N
\eeq
where $rhs_\nu$ is a polynomial function of $f_\mu, \wpp_0 f_\mu, \wpp_{\bar a}f_\mu$, for $\mu = 1,\ldots, 2N$, with coefficients that are rather smooth functions, such that the estimate
\beq\label{dhs.app2}
\| rhs(\cdot, t)\|_{H^k(\R^{d-2})}
\le C\left( 1+ \Theta_1(t) + \Theta_1(t)^3 \right)
\eeq
holds for all $t$ when $f$ is smooth enough, where
\beq\label{Theta1.bis}
\Theta_1(t) := \| f (\cdot, t)\|_{H^{k+1}} +  \kz \| \wpp_0 f (\cdot, t)\|_{H^{k}}
\eeq
The equation is posed on $\R^{d-2}\times (0,T)$, and we assume that $k>\frac d2-1$.

The point of this discussion is to check that the regularity \eqref{dhs.app2} the nonlinearity is sufficient to control the solution in the norms appearing in \eqref{Theta1.bis} above. This is presumably standard, but we do not know a handy reference.

\begin{lemma}\label{lem:dhs}
The system of equations \eqref{dhs.app} with initial data 
$(f, \wpp_0 f)|_{t=0} = (0,0)$ has a unique solution
\beq\label{f.reg}
f\in C(0,T_0, H^{k+1}(\R^{d-2}))  \quad\mbox{  with } \ \ 
\wpp_0 f\in  C(0,T_0), H^{k}(\R^{d-2}))
\eeq
for some $T_0\in (0,T)$ that can be bounded below in terms of the constant in \eqref{dhs.app2}. This solution satisfies
\beq\label{dhs.est}
\Theta_1(t) \le \begin{cases} \sqrt{2Ct} &\mbox{ if }0\le t\le \frac 1{2C}\\
(3-4Ct)^{-1/4} &\mbox{ if }\frac 1{2C}\le t  < \min\{ 3t/4C, T\} \le T_0 .
\end{cases}
\eeq
\end{lemma}

\begin{proof}
    First we, we modify the system by adding $f_\nu$ to both sides.
    It becomes
    \beq\label{dhs.app3}
(\kz\wpp_0\wpp_0 + \ko \wpp_0 - \wpp_{\bar a }\wpp_{\bar a} +1)f_\nu
= rhs_\nu, \qquad \nu = 1,\ldots, 2N
\eeq
for a new $rhs$ that still satisfies \eqref{dhs.app2}, with the constant adjusted if necessary.

We next  prove an {\em a priori} estimate for \eqref{dhs.app3}. Let $r$ be a multiindex
    on $\R^{d-2}$ such that $|r|\le k$. We apply $\wpp^r$ to \eqref{dhs.app}, then multiply both sides of the equation by $\wpp_0\wpp^r f_\nu$. After integration by parts and summing over $\nu$, this leads to
    \[
   \begin{aligned}
&\frac 12 \wpp_0 \left( \|\wpp^r\nabla_{\bar a}f\|_{L^2}^2 +
   \kz \|\wpp^r\wpp_0f\|_{L^2}^2 + \|\wpp^r f\|_{L^2}^2
   \right)
   + \ko \| \wpp^r\wpp_0f\|_{L^2}^2\\
   &\hspace{10em}
   \le \int_{\R^{d-2}}|\wpp^r(rhs)| \ |\wpp^r\wpp_0 f|\\
   &\hspace{10em}
   \le \| rhs\|_{H^{k}} \, \|\wpp_0 f\|_{H^k}.
   \end{aligned}
   \]
    Summing over all multiindices $r$ such that $0\le |r|\le k$ and appealing to \eqref{dhs.app2}, we deduce that
    \[
    \begin{aligned}
            \frac 12\wpp_0 (\Theta_1(t)^2) + \ko \| \wpp_0 f\|_{H^{k}}^2 &
        \le C(1+\Theta_1(t) + \Theta_1(t)^3) \| \wpp_0 f\|_{H^k}\\
        &\le \frac C\delta (1+\Theta_1^6)  + \delta  \| \wpp_0 f\|_{H^k}^2\quad\mbox{ for any }\delta>0.
    \end{aligned}
    \]
    Note also that $\Theta_1(0)=0$.
    
    If $\kz>0$ then we take $\delta=1$ and estimate $\|\wpp_0 f(\cdot, t)\|_{H^k}^2$ by $C(\kz) \Theta_1(t)^2$. If $\kz=0$ then $\ko>0$ and we can take $\delta = \ko$. Either way we obtain
    \[
\wpp_0(\Theta_1^2) \le C(1+ \Theta_1^6), \qquad \Theta_1^2(0)=0.
    \]
The right-hand side is bounded by $2C\max\{ 1, 2C\Theta_1^6\}$, from which it is easy to check that \eqref{dhs.est} holds.

Existence of solutions can be proved by the applying the contraction mapping principle\
in the space
\[
Y := \{ f\in C(0,T_1; H^{k+1})  :  \kz f_t\in C(0,T_1; H^k) \}
\]
for $T_1$ to be chosen, with the natural norm
\[
\| f\|_Y = \| f\|_{L^\infty_{T_1}H^{k+1}} + \kz \|\wpp_0 f\|_{L^\infty_{T_1}H^k} .
\]
The mapping in question is given by
\[
f\in Y\mapsto g\mbox{ solving }
(\kz \wpp_0\wpp_0 +\ko \wpp_0 - \wpp_{\bar a}\wpp_{\bar a})g = rhs(f)
\]
with initial data specified.
The key point is that for $f_1, f_2\in Y$, because $rhs$ is at most cubic in $f$ and its first derivatives,
\[
\| rhs(f_1) - rhs(f_2) \|_{L^\infty_{T_1}H^k} \le
C( \| f_1\|_Y^2 + \|f_2\|_Y^2)\|f_1-f_2\|_{Y}^2.
\]
This combined with estimates like those above show that if $T_1$ is small enough, the mapping is a contraction in a suitable ball in $Y$. This argument can be iterated to extend the interval of existence to at least $T_0$.
\end{proof}

\subsection{ further lemmas}\label{App:more}

\begin{lemma}\label{lem:fq}

Assume that $q:\R^{d-2}\to M_N$ satisfies
\beq\label{qcompact.again}
\mbox{ $\exists$ compact }K\subset M_N\mbox{ such that }q(y_{\bar a})\in K\mbox{ for all }(y_{\bar a})\in \R^{d-2}.
\eeq
Assume that
$\pp_{\bar a}q\in H^{L}(\R^{d-2})$ for  some $L\ge d$, and set
\[
\Theta :=  \| \pp_{\bar a}q\|_{H^L(\R^{d-2})}.
\]
Assume that $f:\R^2\times M_N \to \R$ or $\C$ is a function such that
\begin{equation}\label{fdecay}
|\pp^r f(y_a,q)| \le C(1+|y_a|)^{-|r|} \qquad\mbox{ whenever }q\in K
\end{equation}
for any multi-index $r$ denoting differentiation with respect to the $y_a$ and $q$ variables, where $C= C(r,K)$.
Let $f(q): \R^d\to \R$ denote the function 
\beq\label{fq.def}
f(q)(y) := f(y_a; q( y_{\bar a})),
\eeq

Then $f(q)\in  W^{L+1-\lfloor d/2\rfloor,\infty}(\R^d)$, and
\beq\label{fq.Wkinfty}
\|f(q)\|_{W^{L+1-\lfloor d/2\rfloor,\infty}(\R^d)} \le C.
\eeq
Moreover if $k\le L$ then
\beq\label{fqpsi2}
\| f(q)\,\tilde u\|_{H^{k}(\R^d)} \le C \|\tilde u \|_{H^k(\R^d)} 
\eeq
and for multiindices $r,s$ such that $|r|+|s|\le k$ and $|r|\ge 1$,
\beq\label{f(q)psi}
\| \pp^r f(q) \,  \pp^s\tilde u\|_{L^2(\R^d)} \le C \| \tilde u\|_{H^{k-1}(\R^d)}. 
\eeq
Finally, if 
\[
|\pp^r f(y_a,q)| \le C(1+|y_a|)^{-|r|-1} \qquad\mbox{ whenever }q\in K,
\]
then 
\beq\label{gfqpsi}
\| \pp_{\bar a}q\, f(q)\,\tilde u\|_{H^{k}(\R^d)} \le C \|\tilde u \|_{H^k(\R^d)} \qquad\mbox{ for } k\le L.
\eeq
\end{lemma}

\begin{proof}
To prove \eqref{fq.Wkinfty}, we set $j = \lfloor \frac d2 \rfloor>\frac d2-1$, and we fix a multiindex $|r|\le L+1-j$. We write $r = r_a+r_{\bar a}$, where $\pp^{r_a}$ contains all derivatives with respect to the $y_a$ variables, and $\pp^{r_{\bar a}}$ contains all derivatives with respect to $y_{\bar a}$. If $r = r_a$, then  \eqref{fdecay}
implies that
\[
|\pp^r f(q)|(y) \le C = C(k).
\]
If not, it follows from the chain rule and \eqref{fdecay} that
\beq\label{lotsofterms}
|\pp^r  f(q)|(y) \le C \sum_{l=1}^{|r_{\bar a}|} \sum_{r_1+\cdots +r_l = r_{\bar a}}
\left( (1+|y_a|)^{-l} \prod_{m=1}^l |\pp^{r_m}q|(y_{\bar a}) \right)
.
\eeq
Since $|r_m| \le |r_{\bar a}| \le L+1-j$ and $j> \frac d2-1$,
\[
\| \pp^{r_m}q \|_{L^\infty(\R^{d-2})} \le  C\| \pp^{r_m}q \|_{H^j(\R^{d-2})} \le C\| \nabla q \|_{H^{L}(\R^{d-2})}\le C\Theta.
\]
We obtain \eqref{fq.Wkinfty} directly from the previous two inequalities.

We next prove \eqref{f(q)psi}. Fix $r,s$ such that $|r|+|s| \le k \le L$.
If $k\le d/2$ then $|r| \le d/2 \le L+1-\lfloor d/2\rfloor$, since $L\ge d$, so
\eqref{fq.Wkinfty} implies that $\pp^r f(q)\in L^\infty$, and \eqref{f(q)psi} is an
easy consequence.

We now assume that $k> \frac d2$,
and we write $r=r_a+ r_{\bar a}$ as above.
If $r = r_a$ then
\[
\| \pp^r f(q) \,  \pp^s\tilde u\|_{L^2(\R^d)} 
\le
\| \pp^{r_a} f(q) \|_{L^\infty(\R^d)} \|   \pp^s\tilde u\|_{L^2(\R^d)} 
\overset{\eqref{fdecay}} \le C \|\tilde u\|_{H^k(\R^d)} .
\]
So we may assume that $|r_{\bar a}|\ge 1$, in which case \eqref{lotsofterms} holds.
It suffices to estimate the product of $\pp^s\tilde u$ with a single term in the sum on the right-hand side, say
\[
 \pp^s\tilde u\prod_{m=1}^l Q_{r_m}, \qquad \mbox{ for }Q_{r_m}(y) := |\pp^{r_m}q|(y_{\bar a}) (1+|y_a|)^{-1},
\]
where $|r_1|+\cdots +|r_l| + |s| \le k$ and $|r_m|\ge 1$ for all $m$.

It is convenient to write $k_m := |r_m|$ and $k_0 = |s|$. 

First assume that $L +1 -k_m < \frac {d-2}2$ for all $m\ge 1$. 
Then by the Sobolev embedding theorem and interpolation, for $m=1,\ldots, l$,
\begin{multline}
\|\pp^{r_m} q\|_{L^{p_m}(\R^{d-2})} \le
C\|\pp^{r_m} q\|_{H^{L+1-k_m}(\R^{d-2})} \le
C\|\nabla q\|_{H^{L}(\R^{d-2})} \le C\Theta \\
\mbox{ for all $p_m$ such that } \frac 12 \ge p_m\ge \frac 12 - \frac{L+1-k_m}{d-2}.
\end{multline}
Since the function $y_a\mapsto (1+|y_a|)^{-1}$ belongs to $L^{p_m}(\R^2)$ for $p_m>2$, it follows that
\[
\| Q_{r_m} 
\|_{L^{p_m}(\R^d)} \le C \Theta\quad
\mbox{ if } \ \ \frac 12 > p_m\ge \frac 12 - \frac{L+1-k_m}{d-2}.
\]
Similarly, if $k - k_0-1< \frac d2$, then
\[
\| \pp^s\tilde u \|_{L^{p_0}(\R^d)} \le C \| \tilde u\|_{H^{k-1}(\R^d)} \qquad\mbox{ if }\ \  \frac 12 \ge \frac 1{p_0} \ge \frac 12 - \frac{k-k_0-1}d.
\]
We observe that %
\[
\begin{aligned}
&\left(  \frac 12 - \frac{k-k_0-1}{d} \right)
+ \sum_{m=1}^l \left( \frac 12 - \frac{L+1-k_m}{d-2}\right) \\
&\hspace{4em}\le
\left(  \frac 12 - \frac{k-k_0-1}{d} \right)
+ \sum_{m=1}^l \left( \frac 12 - \frac{L+1-k_m}{d}\right) \\
&\hspace{4em}
=
(l+1) \left( \frac 12 - \frac{k}{d}\right) + \frac 1d\sum_{m=0}^l{k_m} + \frac 1d(1 - l(L+1-k))  . 
\end{aligned}
\]
Since  $\sum_{m=0}^l k_m\le k$ and $k>\frac d2$, and because $l\ge 1$ and $k\le L$, the right-hand side is less than $1/2$.
Thus we can choose $p_0,\ldots, p_l$ such that $\frac 1{p_0}+\cdots +\frac 1{p_l} = \frac 12$.
Then H\"older's inequality implies that
\beq\label{product.est}
\|  \pp^s\tilde u\prod_{m=1}^l Q_{r_m} \|_{L^2(\R^d)}
\le \| \pp^s\tilde u\|_{L^{p_0}(\R^d)} \prod_{m=1}^l  \left\| Q_{r_m} \right\|_{L^{p_m}(\R^d)}
\le C \Theta^l \| \tilde u\|_{H^k(\R^d)}.
\eeq
If $L+1-k_m\ge (d-2)/2$ for some $m\ge 1$ or if $k-k_0-1\ge \frac d2$, then the corresponding exponents $p_m$ may be chosen to be $+\infty$ if strict inequality holds, 
and arbitrarily close to $+\infty$ otherwise, and one can again arrange that $\frac 1{p_0}+\cdots +\frac 1{p_m} = \frac 12$,
again leading to \eqref{product.est}. We deduce \eqref{f(q)psi} from  \eqref{lotsofterms} and \eqref{product.est}.

To prove \eqref{fqpsi2}, we fix $|r|\le k$ and compute $\pp^r ( f(q)\tilde u)$. The terms that arise can all be estimated by \eqref{f(q)psi}, unless $|r|=k$, in which case we obtain a term $f(q)\pp^r\tilde u$ that can be estimated easily using the fact that $f(q)\in L^\infty$.

The proof of \eqref{gfqpsi} is very similar to that of \eqref{f(q)psi} and is omitted.
\end{proof}

\begin{lemma}\label{lem:25}
    Assume that $q:\R^{d-2}\times (0,T)\to M_N$ satisfies \eqref{q.compact0} and \eqref{q.decay},
    with
\[
\Theta :=  \| \pp_{\bar a}q(\cdot, t)\|_{L^\infty_T H^L(\R^{d-2})} + \| \pp_0 q(\cdot, t)\|_{L^\infty_T H^L(\R^{d-2})} <\infty.
\]
Assume that $f:\R^2\times M_N \to \R$ or $\C$ is a function such that
\begin{equation}\label{fdecayexp}
|\pp^r f(x,q)| \le Ce^{-\gamma|y_a|} \qquad\mbox{ whenever $q\in K$ and $|r|\le L$},
\end{equation}
for some $\gamma\in (0,1)$, with $C = C(r,K,\gamma)$, and define $f(q)$ is in \eqref{fq.def}.

Then for any $\alpha,\beta\in 0,3,\ldots. d$, and for $j = \lfloor \frac d2 \rfloor>\frac d2-1$,
\beq\label{pqfq}
\| \pp_{\alpha}q\, f(q) \|_{L^\infty_T W_\gamma^{L-j, \infty}(\R^d)} 
 \ + \ \| \pp_{\alpha}q\, f(q) \|_{L^\infty_T H^{L}(\R^d)} \le C
\eeq
and 
\beq\label{pqfq2}
\| \pp_{\alpha}q\,  \pp_{\beta}q\, f(q) \|_{L^\infty_T W^{L-j, \infty}_\gamma(\R^d)} 
 \ + \ \|  \pp_{\alpha}q\,  \pp_{\beta}q\, f(q) \|_{L^\infty_T H^{L}(\R^d)} \le C
\eeq
for a constant depending on $C(r,K,\gamma)$ above as well as $\Theta, L,j$.
\end{lemma}

\begin{proof}
    Since $j>\frac d2-1$, a Sobolev embedding implies that $\pp_{\alpha}q(\cdot, t) \in H^L \hookrightarrow W^{L-j, \infty}(\R^{d-2})$ for every $t\in (0,T)$, with the estimate 
    $ \| \pp_{\alpha}q\,(\cdot, t) \|_{L^\infty_T W^{L-j,\infty}(\R^{d-2})} \le C \Theta$. This and the previous lemma easily imply the $W^{L-j,\infty}$ estimates in \eqref{pqfq}, \eqref{pqfq2}.
    
    To prove the $H^L$ estimates, we consider only \eqref{pqfq2}, which is the slightly harder case. We fix $t\in (0,T)$ and suppress the $t$ variable. We next fix $|r|\le L$ and expand as in \eqref{lotsofterms}.
    A typical term on the right-hand side has the form
    \beq\label{typicalterm}
    C|\pp^{r_1}(\pp_\alpha q\, \pp_\beta q)| \prod_{m=2
}^l |\pp^{r_m}q|(y_{\bar a}) e^{-\gamma|y_a|}
    \eeq
    where $ |r_1|+ \cdots +|r_m| \le L$. We write $|r_j|=k_j$ for $j=1,\ldots, l$.
    Since $L>\frac d2-1$, the space $H^L(\R^{d-2})$ is an algebra, and 
    \[
    \| \pp^{r_1}(\pp_\alpha q\, \pp_\beta q) \|_{L^{p_1}} 
    \le C \|\pp_\alpha q\, \pp_\beta q \|_{H^L} \le C
     \|\pp_\alpha q \|_{H^L}  \| \pp_\beta q \|_{H^L} \le C\Theta^2
    \]
    as long as 
    \[
    \frac 12 \ge \frac 1{p_1} \ge \frac 12 - \frac {L-k_1}{d-2}
    \]
    and $L-k_1 < \frac d2-1$, with straightforward adjustments if $L-k_1\ge \frac d2-1$. With this observation, the verification that \eqref{typicalterm} belongs to $L^2$ is very similar to the proof of \eqref{f(q)psi} and is omitted.
\end{proof}

\begin{lemma}\label{lem:last}
    Assume that $L+1>\frac d2$ and that $q:\R^{d-2}\to M_N$ satisfies \eqref{qcompact.again}, with
    $\pp_{\bar a} q \in H^L(\R^{d-2})$.
    
    Assume that $F\in H^k(\R^d ; \C\times \R^2)$ for some $k\le L$, and 
    for $\mu\in \{1,\ldots, 2N\}$ 
    define $f_\mu:\R^{d-2}\to \R$ by
    \[
    f_\mu(y_{\bar a}) = (F , n_\mu(q))_{L^2(\R^2_{y_{\bar a}})} = \int_{\R^2 } F(y_a, y_{\bar a})\cdot n_\mu(y_a; q(y_{\bar a})) \, dy_a.
    \]
    Then $f_\mu\in H^k(\R^{d-2})$, and 
    \beq\label{eq:last}
    \| f_\mu\|_{H^k(\R^{d-2)}} \le C \| F\|_{H^k(\R^d)}.
    \eeq
    for $C$ depending on $\|\pp_{\bar a}q\|_{H^L}$.
\end{lemma}

\begin{proof}
We may assume that $F$ is smooth, since the general case then follows by an approximation argument.    Let $r$ be a multiindex on $\R^{d-2}$, so that $\pp^r$ involves only derivatives with respect to the $y_{\bar a}$ directions. Assume that $|r|\le k$. Since $F$ is smooth, it is straightforward to justify the identity
\[
\pp^r f(y_{\bar a}) = \sum_{r_0+s_0=r}\frac{r!}{r_0! s_0!}(\pp^{r_0} F , \pp^{s_0} n_\mu)_{L^2(\R^2_{y_a})}.
\]
In view of Lemma \ref{lem:decayq} and the chain rule, there exists $\gamma>0$ such that 
\[
|\pp^{r_0} F  \cdot  \pp^{s_0} n_\mu|(y)
\le C|\pp^{r_0}F(y)|\,
\left(\sum_{r_1+\cdots+r_l = s_0} \prod_{j=1}^l |\pp^{r_j} q(y_{\bar a})|\right) e^{-\gamma|y_a|}.
\]
Writing
\[
G(y_{\bar a}) := \int_{\R^2_{y_{\bar a}}} |\pp^{r_0} F(y)| \, e^{-\gamma|y_a|} 
dy_a,
\]
we deduce that
\[
\int_{\R^{d-2}}|(\pp^{r_0}F , \pp^{s_0}n_\mu)_{L^2_{y_a}(\R^2)}|^2 dy_{\bar a}
\le
C\sum_{r_1+\cdots+r_l=s_0} \int_{\R^{d-2}} G(y_{\bar a})^2\prod_{j=1}^l|\pp^{r_j}q(y_{\bar a})|^2 \ dy_{\bar a}.
\]
Arguing exactly as in the proof of Lemma \ref{lem:fq}, and using the assumptions that $L+1>\frac d2$ and $k\le L$, we can choose
$p_j\in [1,\infty]$ for $j=0,\ldots,l$ such that 
\[
\| \pp^{r_0}F\|_{L^{p_0}(\R^d)} \le C \| F\|_{H^k(\R^d)},
\qquad
\| \pp^{r_j}q\|_{L^{p_j}(\R^{d-1})} \le C \| \pp_{\bar a} q \|_{H^L(\R^{d-2})},
\]
and 
\[
\frac 2{p_0}+\cdots+\frac 2{p_l}= 1.
\]
We use H\"older's inequality  with exponents $\frac{p_j}2$ for $j=0,\ldots, l$ to estimate
\begin{align*}
&\int_{\R^{d-2}}|(\pp^{r_0}F , \pp^{s_0}n_\mu)_{L^2_{y_a}(\R^2)}|^2 dy_{\bar a}
\\
&\hspace{5em}
\le
C\sum_{r_1+\cdots+r_l=s_0}
\| G \|_{L^{p_0}(\R^{d-2})}^2\prod_{j=1^l}
\| \pp^{r_j} q\|_{L^{p_j}(\R^{d-2})}^2.
\end{align*}
Note, we define $q_0$ by requiring that $\frac 1{p_0}+\frac 1{q_0}=1$, and we again use H\"older's inequality to deduce that
\begin{align*}
    \| G \|_{L^{p_0}(\R^{d-2})}
    &= 
    \left(
    \int_{\R^{d-2}}\left(\int_{\R^2} |\pp^{r_0}F(y)|\  e^{-\gamma|y_a|} dy_a
    \right)^{p_0} dy_{\bar a}
    \right)^{\frac 1{p_0}} \\
    &\le
    \left[\int_{\R^{d-2}}\left(\int_{\R^2}|\pp^{r_0}F(y)|^{p_0}dy_{a}\right)
    \left(\int_{\R^2}e^{-\gamma q_0|y_a|} dy_a \right)^{p_0/q_0}dy_{\bar a}\right]^{1/p_0}\\
    &=
    C\| \pp^{r_0}F\|_{L^{p_0}(\R^d)}.
\end{align*}
Putting these together, we conclude that
\[
\| \pp^r f\|_{L^2(\R^{d-2})}^2 \le C(\| \pp_{\bar a}q\|_{H^L(\R^{d-2})}) \|F\|_{H^k(\R^d)}^2.
\]
Since this holds for every $r$ such that $|r|\le k$, the conclusion follows.

\end{proof}


\bigskip

{\large \bf Acknowledgments} This paper is based on the PhD thesis of the first author A. Geevechi. This was partially supported by the Connaught international scholarship at the University of Toronto. He is also grateful to his mentor Prof. Shing-Tung Yau for his encouragement and support. R. Jerrard's contributions to this work were partially supported by the Natural Sciences and Engineering Research Council of Canada under Operating Grant 261955. He would like to thank Tze-lan Sang for inspiring exchanges related to reconnection.

\bibliography{AHM.bib}

@article {CzubakJerrard,
    AUTHOR = {Czubak, Magdalena and Jerrard, Robert L.},
     TITLE = {Topological defects in the abelian {H}iggs model},
   JOURNAL = {Discrete Contin. Dyn. Syst.},
  FJOURNAL = {Discrete and Continuous Dynamical Systems. Series A},
    VOLUME = {35},
      YEAR = {2015},
    NUMBER = {5},
     PAGES = {1933--1968},
      ISSN = {1078-0947,1553-5231},
   MRCLASS = {81V35 (35B25 35L71)},
  MRNUMBER = {3294233},
       DOI = {10.3934/dcds.2015.35.1933},
       URL = {https://doi.org/10.3934/dcds.2015.35.1933},
}

@article {DemouliniStuart,
    AUTHOR = {Demoulini, Sophia and Stuart, David},
     TITLE = {Gradient flow of the superconducting {G}inzburg-{L}andau
              functional on the plane},
   JOURNAL = {Comm. Anal. Geom.},
  FJOURNAL = {Communications in Analysis and Geometry},
    VOLUME = {5},
      YEAR = {1997},
    NUMBER = {1},
     PAGES = {121--198},
      ISSN = {1019-8385,1944-9992},
   MRCLASS = {58E15 (35Q55 82D55)},
  MRNUMBER = {1456310},
       DOI = {10.4310/CAG.1997.v5.n1.a3},
       URL = {https://doi.org/10.4310/CAG.1997.v5.n1.a3},
}

@article{bou2001vortex,
  title={Vortex collisions: crossing or recombination?},
  author={Bou-Diab, Malek and Dodgson, Matthew JW and Blatter, Gianni},
  journal={Physical review letters},
  volume={86},
  number={22},
  pages={5132},
  year={2001},
  publisher={APS}
}

@article{KidaTakaoka,
  title={Vortex reconnection},
  author={Kida, Shigeo and Takaoka, M},
  journal={Annual Review of Fluid Mechanics},
  volume={26},
  number={1},
  pages={169--177},
  year={1994},
  publisher={Annual Reviews 4139 El Camino Way, PO Box 10139, Palo Alto, CA 94303-0139, USA}
}

@misc{Masoud-thesis,
      title={A Gluing Problem for a Gauged Hyperbolic PDE}, 
      author={Amirmasoud Geevechi},
      year={2024},
      eprint={2405.16092},
      archivePrefix={arXiv},
      primaryClass={math.AP},
      url={https://arxiv.org/abs/2405.16092}, 
}

@article{oshima1977interaction,
  title={Interaction of two vortex rings along parallel axes in air},
  author={Oshima, Yuko and Asaka, Saburo},
  journal={Journal of the Physical Society of Japan},
  volume={42},
  number={2},
  pages={708--713},
  year={1977},
  publisher={The Physical Society of Japan}
}

@article {EncisoPeraltaSalas1,
    AUTHOR = {Enciso, Alberto and Luc\`a, Renato and Peralta-Salas, Daniel},
     TITLE = {Vortex reconnection in the three dimensional {N}avier-{S}tokes
              equations},
   JOURNAL = {Adv. Math.},
  FJOURNAL = {Advances in Mathematics},
    VOLUME = {309},
      YEAR = {2017},
     PAGES = {452--486},
      ISSN = {0001-8708,1090-2082},
   MRCLASS = {35Q30 (35Q31 76D05 76D17)},
  MRNUMBER = {3607283},
MRREVIEWER = {Isabelle\ Gruais},
       DOI = {10.1016/j.aim.2017.01.025},
       URL = {https://doi-org.myaccess.library.utoronto.ca/10.1016/j.aim.2017.01.025},
}

@article {EncisoPeraltaSalas2,
    AUTHOR = {Enciso, Alberto and Peralta-Salas, Daniel},
     TITLE = {Approximation theorems for the {S}chr\"odinger equation and
              quantum vortex reconnection},
   JOURNAL = {Comm. Math. Phys.},
  FJOURNAL = {Communications in Mathematical Physics},
    VOLUME = {387},
      YEAR = {2021},
    NUMBER = {2},
     PAGES = {1111--1149},
      ISSN = {0010-3616,1432-0916},
   MRCLASS = {35Q55 (81Q80)},
  MRNUMBER = {4315668},
       DOI = {10.1007/s00220-021-04177-w},
       URL = {https://doi-org.myaccess.library.utoronto.ca/10.1007/s00220-021-04177-w},
}

@article{kleckner2013creation,
  title={Creation and dynamics of knotted vortices},
  author={Kleckner, Dustin and Irvine, William TM},
  journal={Nature physics},
  volume={9},
  number={4},
  pages={253--258},
  year={2013},
  publisher={Nature Publishing Group UK London}
}

@article{saffman1990model,
  title={A model of vortex reconnection},
  author={Saffman, Philip G},
  journal={Journal of Fluid Mechanics},
  volume={212},
  pages={395--402},
  year={1990},
  publisher={Cambridge University Press}
}

@article{KoplikLevine,
  title = {Vortex reconnection in superfluid helium},
  author = {Koplik, Joel and Levine, Herbert},
  journal = {Phys. Rev. Lett.},
  volume = {71},
  issue = {9},
  pages = {1375--1378},
  numpages = {0},
  year = {1993},
  month = {Aug},
  publisher = {American Physical Society},
  doi = {10.1103/PhysRevLett.71.1375},
  url = {https://link.aps.org/doi/10.1103/PhysRevLett.71.1375}
}

@article{lathropetal,
author = {Gregory P. Bewley  and Matthew S. Paoletti  and Katepalli R. Sreenivasan  and Daniel P. Lathrop },
title = {Characterization of reconnecting vortices in superfluid helium},
journal = {Proceedings of the National Academy of Sciences},
volume = {105},
number = {37},
pages = {13707-13710},
year = {2008},
doi = {10.1073/pnas.0806002105},
URL = {https://www.pnas.org/doi/abs/10.1073/pnas.0806002105},
eprint = {https://www.pnas.org/doi/pdf/10.1073/pnas.0806002105},
}

@article {MerleZaag,
    AUTHOR = {Merle, Frank and Zaag, Hatem},
     TITLE = {Reconnection of vortex with the boundary and finite time
              quenching},
   JOURNAL = {Nonlinearity},
  FJOURNAL = {Nonlinearity},
    VOLUME = {10},
      YEAR = {1997},
    NUMBER = {6},
     PAGES = {1497--1550},
      ISSN = {0951-7715,1361-6544},
   MRCLASS = {35Q99 (35B40 35K99 82D55)},
  MRNUMBER = {1483553},
MRREVIEWER = {Peter\ L.\ Christiansen},
       DOI = {10.1088/0951-7715/10/6/006},
       URL = {https://doi-org.myaccess.library.utoronto.ca/10.1088/0951-7715/10/6/006},
}

@article{shellard1987cosmic,
  title={Cosmic string interactions},
  author={Shellard, EPS},
  journal={Nuclear Physics B},
  volume={283},
  pages={624--656},
  year={1987},
  publisher={Elsevier}
}

@article {CJO,
    AUTHOR = {Chapman, S. J. and Hunton, B. J. and Ockendon, J. R.},
     TITLE = {Vortices and boundaries},
   JOURNAL = {Quart. Appl. Math.},
  FJOURNAL = {Quarterly of Applied Mathematics},
    VOLUME = {56},
      YEAR = {1998},
    NUMBER = {3},
     PAGES = {507--519},
      ISSN = {0033-569X,1552-4485},
   MRCLASS = {76B47 (76A25 76D17 82D55)},
  MRNUMBER = {1637052},
       DOI = {10.1090/qam/1637052},
       URL = {https://doi-org.myaccess.library.utoronto.ca/10.1090/qam/1637052},
}

@article{Matzner,
    author = {Matzner, Richard A.},
    title = {Interaction of U(1) cosmic strings: Numerical intercommutation},
    journal = {Computer in Physics},
    volume = {2},
    number = {5},
    pages = {51-64},
    year = {1988},
    month = {09},
    issn = {0894-1866},
    doi = {10.1063/1.168306},
    url = {https://doi.org/10.1063/1.168306},
    eprint = {https://pubs.aip.org/aip/cip/article-pdf/2/5/51/11380574/51\_1\_online.pdf},
}

@article{Superconductors,
  title = {Vortex cutting in superconductors},
  author = {Glatz, A. and Vlasko-Vlasov, V. K. and Kwok, W. K. and Crabtree, G. W.},
  journal = {Phys. Rev. B},
  volume = {94},
  issue = {6},
  pages = {064505},
  numpages = {11},
  year = {2016},
  month = {Aug},
  publisher = {American Physical Society},
  doi = {10.1103/PhysRevB.94.064505},
  url = {https://link.aps.org/doi/10.1103/PhysRevB.94.064505}
}

@article{HananyHashimoto,
doi = {10.1088/1126-6708/2005/06/021},
url = {https://dx.doi.org/10.1088/1126-6708/2005/06/021},
year = {2005},
month = {jun},
publisher = {},
volume = {2005},
number = {06},
pages = {021},
author = {Amihay Hanany and Koji Hashimoto},
title = {Reconnection of colliding cosmic strings},
journal = {Journal of High Energy Physics},
}

@book {JaffeTaubes,
    AUTHOR = {Jaffe, Arthur and Taubes, Clifford},
     TITLE = {Vortices and monopoles},
    SERIES = {Progress in Physics},
    VOLUME = {2},
      NOTE = {Structure of static gauge theories},
 PUBLISHER = {Birkh\"{a}user, Boston, MA},
      YEAR = {1980},
     PAGES = {v+287},
      ISBN = {3-7643-3025-2},
   MRCLASS = {81E10 (53C80 81-02)},
  MRNUMBER = {614447},
MRREVIEWER = {Masatsugu\ Minami},
}

@article {delPino-J-Musso,
    AUTHOR = {del Pino, Manuel and Jerrard, Robert L. and Musso, Monica},
     TITLE = {Interface dynamics in semilinear wave equations},
   JOURNAL = {Comm. Math. Phys.},
  FJOURNAL = {Communications in Mathematical Physics},
    VOLUME = {373},
      YEAR = {2020},
    NUMBER = {3},
     PAGES = {971--1009},
      ISSN = {0010-3616,1432-0916},
   MRCLASS = {35L71 (35B25 35L15)},
  MRNUMBER = {4061403},
       DOI = {10.1007/s00220-019-03632-z},
       URL = {https://doi.org/10.1007/s00220-019-03632-z},
}

@article {Manton,
    AUTHOR = {Manton, N. S.},
     TITLE = {A remark on the scattering of {BPS} monopoles},
   JOURNAL = {Phys. Lett. B},
  FJOURNAL = {Physics Letters. B. Particle Physics, Nuclear Physics and
              Cosmology},
    VOLUME = {110},
      YEAR = {1982},
    NUMBER = {1},
     PAGES = {54--56},
      ISSN = {0370-2693,1873-2445},
   MRCLASS = {81E10 (58F17)},
  MRNUMBER = {647883},
       DOI = {10.1016/0370-2693(82)90950-9},
       URL = {https://doi.org/10.1016/0370-2693(82)90950-9},
}

@article {Palvelev1,
    AUTHOR = {Palvelev, R. V.},
     TITLE = {Justification of the adiabatic principle in the {A}belian
              {H}iggs model},
   JOURNAL = {Trans. Moscow Math. Soc.},
  FJOURNAL = {Transactions of the Moscow Mathematical Society},
      YEAR = {2011},
     PAGES = {219--244},
      ISSN = {0077-1554,1547-738X},
   MRCLASS = {35Q60 (58E15 81T13 81T40)},
  MRNUMBER = {3184819},
       DOI = {10.1090/s0077-1554-2012-00189-7},
       URL = {https://doi.org/10.1090/s0077-1554-2012-00189-7},
}

@article {Palvelev2,
    AUTHOR = {Palvelev, Roman},
     TITLE = {Scattering of vortices in the abelian {H}iggs model},
   JOURNAL = {J. Geom. Symmetry Phys.},
  FJOURNAL = {Journal of Geometry and Symmetry in Physics},
    VOLUME = {10},
      YEAR = {2007},
     PAGES = {73--81},
      ISSN = {1312-5192,1314-5673},
   MRCLASS = {81T13 (58E50)},
  MRNUMBER = {2380051},
MRREVIEWER = {Pierpaolo\ Esposito},
}

@article {PalvelevSergeev,
    AUTHOR = {Palvelev, R. V. and Sergeev, A. G.},
     TITLE = {Justification of the adiabatic principle for hyperbolic
              {G}inzburg-{L}andau equations},
   JOURNAL = {Tr. Mat. Inst. Steklova},
  FJOURNAL = {Trudy Matematicheskogo Instituta Imeni V. A. Steklova},
    VOLUME = {277},
      YEAR = {2012},
     PAGES = {199--214},
      ISSN = {0371-9685,3034-1809},
      ISBN = {5-7846-0124-5; 978-5-7846-0124-7},
   MRCLASS = {35Q75 (35B25 35J20 35Q56)},
  MRNUMBER = {3052273},
MRREVIEWER = {Boris\ A.\ Malomed},
       DOI = {10.1134/s0081543812040141},
       URL = {https://doi.org/10.1134/s0081543812040141},}

@article {ParisePigatiStern,
    AUTHOR = {Parise, Davide and Pigati, Alessandro and Stern, Daniel},
     TITLE = {The parabolic {$U(1)$}-{H}iggs equations and codimension-two
              mean curvature flows},
   JOURNAL = {Geom. Funct. Anal.},
  FJOURNAL = {Geometric and Functional Analysis},
    VOLUME = {34},
      YEAR = {2024},
    NUMBER = {4},
     PAGES = {1171--1225},
      ISSN = {1016-443X,1420-8970},
   MRCLASS = {53E10},
  MRNUMBER = {4768584},
MRREVIEWER = {Kin\ Ming\ Hui},
       DOI = {10.1007/s00039-024-00684-9},
       URL = {https://doi.org/10.1007/s00039-024-00684-9},
}

@article {Samols,
    AUTHOR = {Samols, T. M.},
     TITLE = {Vortex scattering},
   JOURNAL = {Comm. Math. Phys.},
  FJOURNAL = {Communications in Mathematical Physics},
    VOLUME = {145},
      YEAR = {1992},
    NUMBER = {1},
     PAGES = {149--179},
      ISSN = {0010-3616,1432-0916},
   MRCLASS = {58E15 (58D27 81T10)},
  MRNUMBER = {1155287},
MRREVIEWER = {Richard\ W.\ Montgomery},
       URL = {http://projecteuclid.org/euclid.cmp/1104249538},
}

@article {Stuart,
    AUTHOR = {Stuart, D.},
     TITLE = {Dynamics of abelian {H}iggs vortices in the near {B}ogomolny
              regime},
   JOURNAL = {Comm. Math. Phys.},
  FJOURNAL = {Communications in Mathematical Physics},
    VOLUME = {159},
      YEAR = {1994},
    NUMBER = {1},
     PAGES = {51--91},
      ISSN = {0010-3616,1432-0916},
   MRCLASS = {58E15 (53C80 58F17 81T13)},
  MRNUMBER = {1257242},
MRREVIEWER = {Jan\ Segert},
       URL = {http://projecteuclid.org/euclid.cmp/1104254491},
}
\bibliographystyle{acm}

\end{document}